\documentclass[11pt]{amsart}

\usepackage{mathrsfs}
\usepackage{amsfonts,amssymb,amsmath,amsthm}
\usepackage{latexsym}
\usepackage{bm}
\usepackage{stmaryrd}

\usepackage{geometry}
\usepackage{indentfirst}
\usepackage{titletoc}
\usepackage{appendix}
\usepackage{enumitem}
\usepackage{needspace}
\usepackage{changepage}
\usepackage{etoolbox}
\usepackage{todonotes}

\usepackage{graphicx}
\usepackage{float}
\usepackage{booktabs}
\usepackage{longtable}
\usepackage{array}
\usepackage{arydshln}
\usepackage{tikz}
\usepackage{xcolor}
\usetikzlibrary{arrows.meta,decorations.markings}

\usepackage{cite}
\makeatletter
\patchcmd{\@sect}{\@tocwrite{#1}{#8}}{\@tocwrite{#1}{#7}}{}{\PackageError{aos}{Cannot patch displayed section contents}{}}
\patchcmd{\@sect}{\@tocwrite{#1}{#8}}{\@tocwrite{#1}{#7}}{}{\PackageError{aos}{Cannot patch run-in section contents}{}}
\makeatother
\usepackage[colorlinks=true,linkcolor=blue!45!black,citecolor=green!30!black,urlcolor=blue!50!black]{hyperref}

\allowdisplaybreaks[4]
\newcommand{\blackhyperref}[2]{%
  {\hypersetup{linkcolor=black}\hyperref[#1]{#2}}%
}

\newcommand{\R}{\mathbb R}

\newcommand{\Z}{\mathbb Z}
\newcommand{\Nat}{\mathbb N}

\newcommand{\calA}{\mathcal A}
\newcommand{\calB}{\mathcal B}

\newcommand{\calE}{\mathcal E}

\newcommand{\calH}{\mathcal H}

\newcommand{\calM}{\mathcal M}

\newcommand{\Id}{\mathrm{Id}}

\DeclareMathOperator{\diam}{diam}

\DeclareMathOperator{\supp}{supp}

\DeclareMathOperator{\dist}{dist}

\DeclareMathOperator{\sgn}{sign}

\usepackage{microtype}
\newcommand{\D}{\mathbb D}
\newcommand{\T}{\mathbb T}
\newcommand{\1}{\mathbf1}
\newcommand{\dd}{\,\mathrm d}
\newcommand{\Real}{\operatorname{Re}}
\newcommand{\Ai}{\operatorname{Ai}}
\newcommand{\Var}{\operatorname{Var}}
\newcommand{\calT}{\mathcal T}
\newcommand{\calU}{\mathcal U}
\DeclareMathOperator*{\essinf}{ess\,inf}

\numberwithin{equation}{section}
\newtheorem{theorem}{Theorem}[section]
\newtheorem{lemma}[theorem]{Lemma}
\newtheorem{cor}[theorem]{Corollary}
\newtheorem{prop}[theorem]{Proposition}

\theoremstyle{definition}
\newtheorem{defi}[theorem]{Definition}
\newtheorem*{defi*}{Definition}

\newtheoremstyle{compactintro}{0pt}{}{\normalfont}{}{\bfseries}{.}{5pt plus 1pt minus 1pt}{}
\theoremstyle{compactintro}
\newtheorem*{question}{Question}
\newtheorem*{observation}{Observation}
\AddToHook{env/question/before}{\vspace{-.5\baselineskip}}
\AddToHook{env/observation/before}{\vspace{-.5\baselineskip}}

\theoremstyle{definition}
\newtheorem{remark}[theorem]{Remark}
\newtheorem*{remark*}{Remark}

\AddToHook{env/remark/begin}{\emergencystretch=2em}
\AddToHook{env/remark*/begin}{\emergencystretch=2em}

\theoremstyle{definition}
\newtheorem{example}[theorem]{Example}

\begin{document}

\title[Asymptotic observable symmetry I]{Asymptotic observable symmetry for spacetime observability of waves I: The disc}

\author{Shengquan Xiang}
\address{School of Mathematical Sciences, Peking University, Beijing 100871, China}
\email{shengquan.xiang@math.pku.edu.cn}
\date{}

\begin{abstract}
Building on the spacetime observable symmetry discovered in \cite{NWX}, we introduce asymptotic observable symmetry (AOS) for the wave equation. Its central object is a family of generated approximate velocities constructed from fixed profiles, carrying the same propagation-induced symmetry at different scales and asymptotically realized by exact waves. We investigate AOS on the disc through radial propagation and reflection. 
This gives a family of spacetime measurable observation sets satisfying the essential spacetime measurable geometric control condition for a fixed positive observation time. However, the observability inequality fails on each of these fixed sets.
\end{abstract}

\hypersetup{pdftitle={Asymptotic observable symmetry for spacetime observability of waves I: The disc},pdfauthor={Shengquan Xiang}}

\maketitle

\setcounter{tocdepth}{1}
\tableofcontents

\section{Introduction}\label{sec:intro}

\subsection{Observability, geometry, and symmetry}\label{subsec:intro-obs}

Let $\Omega\subset\R^2$ be a smooth bounded domain. We consider the Dirichlet wave equation
\begin{equation}\label{eq:wave}
\begin{cases}
u_{tt}-\Delta u=0,& \forall(t,x)\in\R\times\Omega,\\
u=0,& \forall (t,x)\in\R\times\partial\Omega,\\
u[0]=(u_0,u_1)\in H_0^1(\Omega)\times L^2(\Omega).
\end{cases}
\end{equation}
Here $u[t]:=\bigl(u(t,\cdot),\partial_tu(t,\cdot)\bigr)$. Its conserved energy is
\begin{equation*}
E[u](t):=\frac12\int_\Omega(|\nabla u|^2+|u_t|^2)\dd x.
\end{equation*}
Given $T>0$ and a spacetime measurable set $G\subset(0,T)\times\Omega$, we ask whether
\begin{equation}\label{eq:obs}
E[u](0)\le C_G\int_G|\partial_tu(t,x)|^2\dd x\dd t
\end{equation}
holds for every solution of \eqref{eq:wave}. Throughout the paper, \emph{observability} refers to this inequality.

Geometry plays a central role in wave observability, following early work on boundary control by Russell \cite{Russell1971a} and on geometric decay by Rauch--Taylor \cite{RT1974}. Lions' $\Gamma$-condition gives boundary observability for sufficiently large time \cite{Lions1988}; see also \cite{Komornik1994,TW2009,Zuazua1990,Zhang2000}. 
In the classical smooth setting, Bardos--Lebeau--Rauch \cite{BLR1992} obtained sharp observability results by introducing the \emph{geometric control condition} (GCC), formulated in terms of generalized rays reaching the observation region within the prescribed time; see also \cite{BG1997,BZ2004,Lebeau1996,LL2016,LRLTT2017, LLZ2026}.

In this paper the observation region is only spacetime measurable, and need not be a cylinder or an open set. We use the following quantitative, measure-theoretic version of ray coverage. It concerns regular interior rays and is unchanged by spacetime null modifications of $G$; see Appendix~\ref{sec:meas}.

\vspace{4mm}
\begin{adjustwidth}{2.5mm}{2.5mm}
\begin{minipage}{\linewidth}
\begin{defi*}
A measurable set $G\subset(0,T)\times\Omega$ satisfies the \emph{essential spacetime measurable GCC} if there exists $\eta>0$ such that
\begingroup
\renewcommand{\theHequation}{gcc}
\begin{equation}\tag{GCC}\label{eq:gcc}
\essinf_{\rho\in S^*\Omega}D_G(\rho)\ge\eta \;\; \textrm{ where } \; D_G(\rho):=\int_0^T\1_G(t,\pi\Phi_t\rho)\dd t
\end{equation}
\endgroup
Thus almost every ray spends at least $\eta$ time in $G$. Here, we use the regular unit-speed billiard flow $\Phi_t$, the projection $\pi:S^*\Omega\to\Omega$, and the essential infimum relative to Liouville measure.
\end{defi*}
\end{minipage}
\end{adjustwidth}
\vspace{4mm}

Ray coverage and wave observation are different assertions, particularly in the spacetime observation setting. A ray records one propagation direction, whereas $\partial_tu$ is the sum of the components arriving at a spacetime point. In the torus setting \cite{NWX}, together with Niu and Wang, the author discovered that two transported characteristic sets can have the same trace on a spacetime observation region. 
We introduced the \emph{observable symmetry condition} (OSC), which characterizes the absence of nontrivial \emph{observable symmetries}. In the torus setting, OSC together with GCC yields a necessary and sufficient characterization of observability. The symmetry theory is then extended to the interval setting \cite{HX}.

\begin{adjustwidth}{5mm}{5mm}
\begin{question}
\itshape
The spacetime observation, geometry, and symmetry problem asks:
\begin{enumerate}[label=\textup{(\arabic*)},leftmargin=8mm]
\item Classify all observable symmetries in each geometry.
\item Determine how much GCC can be realized within each symmetry class.
\item Determine whether GCC, together with a quantitative condition on observable symmetry, implies observability.
\end{enumerate}
\end{question}
\end{adjustwidth}

\subsection{The main results}\label{subsec:intro-main}

The present paper and its sequel \cite{AOSII} address this problem through reflection in the disc and local glancing near positively curved boundary arcs of general smooth domains.
\begin{itemize}[leftmargin=6mm]
\item In this paper, we introduce asymptotic observable symmetry and complete the radial reflection construction in the disc.
\item The sequel \cite{AOSII} treats glancing and proves the corresponding non-observability result on every smooth bounded planar domain for every positive observation time.
\end{itemize}
In each construction we identify the allowed symmetry relations and study \hyperref[eq:gcc]{GCC} realization in a concrete class. The general classification and the sufficient condition in the third question remain beyond these results.

Let $\D:=\{x\in\R^2:|x|<1\}$. The following theorem gives the square-profile case of our disc construction. Define the $2\pi$-periodic profile
\begin{equation}\label{eq:square}
q_0(s):=2\1_{[0,\pi)}(s)-1,\quad \forall s\in\R/2\pi\Z.
\end{equation}

\Needspace{10\baselineskip}
\begin{theorem}\label{thm:disc}
There exist explicit time $T_\D>0$ and spacetime measurable set $G\subset(0,T_\D)\times\D$ satisfying \eqref{eq:gcc}\footnote{In fact, in our construction, every unit-speed generalized ray spends a uniformly positive amount of time in $G\cup\bigl((0,T_\D)\times\partial\D\bigr)$; see Remark~\ref{rem:disc-completion}.} such that the observability inequality \eqref{eq:obs} fails.

More precisely, $G$ is the limit of decreasing regions $G^{(N)}$. There exist explicit functions $\{S_k\}_{k\in\Nat^*}$ generated from the same profile $q_0$ at every scale, and solutions $u_k$ of \eqref{eq:wave} such that 
\begin{equation}\label{eq:disc-finite}
\begin{cases}
\partial_tu_k=S_k+r_k \textrm{ \; and \; } E[u_k](0)=1, \; \; \;\;\; \; \;\;\forall k\in \mathbb{N}^*,\\
\1_{G^{(k)}}S_k=\1_GS_k=0,\; \; \forall k\in \mathbb{N}^*, \\
\|r_k\|_{L^2((0,T_\D)\times\D)}\rightarrow 0 \textrm{ \; as \; } k\rightarrow \infty,
\end{cases}
\end{equation}
and, in polar coordinates, each $S_k$ vanishes outside $(0,T_\D)\times(0,1)\times(\calA_k\cup\calB_k)$, where $\calA_k,\calB_k\subset\R/2\pi\Z$ are angular intervals centered at opposite angles.
\end{theorem}

The more general result Theorem~\ref{thm:disc-gcc} treats the larger class of nonzero real finite-step antiperiodic profiles.
\vspace{2mm}

The main ingredient is an asymptotic version of observable symmetry, called \emph{asymptotic observable symmetry} (AOS), which we introduce in this paper. In the torus example, an exact propagation formula turns the profile balance into exact cancellation in a single wave. In the disc construction, the profile formulas describe the leading parts of wave velocities. The profile symmetry and its associated geometric balance remain exact at every scale, while their realization by a sequence of waves is asymptotic. AOS emphasizes {\it profile, symmetry, scale}, and {\it asymptotics}. This yields non-observability on a fixed spacetime measurable set satisfying \eqref{eq:gcc}.

\subsection{Asymptotic observable symmetry}\label{subsec:intro-aos}

Recall the torus example of \cite{NWX}, using the square profile $q_0$ in \eqref{eq:square}. It has a primitive $Q_0$, with $Q'_0=q_0$. The exact wave and its velocity are
\begin{gather*}
u(t,x)= Q_0(x+t)+Q_0(x-t),\\
\partial_tu(t,x)= q_0(x+t)-q_0(x-t).
\end{gather*}
The set $G_{\T}:=\{q_0(x+t)=q_0(x-t)\}$ matches the two transported half-circles and their complements. On $(0,2\pi)\times\T$, every characteristic spends time $\pi$ in $G_{\T}$, yet $\partial_tu=0$ there.

\smallskip
Passing from characteristic translations to curved boundaries changes this description. A prescribed pair of transported profiles need no longer give an exact wave. It can, however, give the \emph{leading part} of a \emph{high-frequency} wave. This suggests keeping the geometric profile relation exact while allowing its wave realization to be asymptotic.

We call this phenomenon asymptotic observable symmetry (AOS). In the two-branch case, its four core components are the following. The last two introduce the scale dependence and asymptotic realization, which are central to AOS; Section~\ref{sec:aos} gives the precise formulation and also treats multiple branches $\{f_\alpha\}_{\alpha=1}^J$. Appendix~\ref{sec:supp-aos} provides supplementary discussion of compatible profiles, velocity approximation, and localization.
\begin{itemize}[leftmargin=6mm]
\item \emph{Propagated profiles.} Two profiles $f,g$ are generated from one reference profile through the relations imposed by propagation.
\item \emph{Symmetry relations.} A fixed symmetry relation between the profiles gives a reference matching region $\mathcal M$ on which
\begin{equation*}
f(\alpha)=g(\beta),\quad \forall(\alpha,\beta)\in\mathcal M.
\end{equation*}
\item \emph{A scaled family of generated approximate velocities.} Keep this symmetry relation and propagate it at \emph{different scales} $\varepsilon_k\to0^+$:
\begin{equation}\label{eq:intro-family}
S_k(t,x)=A_{\varepsilon_k}(t,x)\Bigl[f\bigl(\Theta_k^+(t,x)\bigr)-g\bigl(\Theta_k^-(t,x)\bigr)\Bigr].
\end{equation}
The propagation phases $\Theta_k^\pm$ and amplitude $A_{\varepsilon_k}$ change with the scale. The same symmetry, encoded by the reference matching region $\mathcal M$, gives spacetime matching regions depending on the scales:
\begin{equation*}
\Gamma_k:=\{(\Theta_k^+,\Theta_k^-)\in\mathcal M\},
\quad\text{thus}\quad S_k=0\quad\text{on }\Gamma_k.
\end{equation*}
\item \emph{Asymptotic realization.} We first arrange an observation region for the first $N$ generated approximate velocities $S_1,\ldots,S_N$ and denote it by $G^{(N)}$. We choose $G^{(N)}\to G$. Exact waves $u_k$ then realize the profiles $f,g$ at scale $k$.

\Needspace{7\baselineskip}
With $Q_T:=(0,T)\times\Omega$, the \emph{asymptotic realization formula} is, for $E[u_k](0)=1$,
\begin{equation}\label{eq:intro-real}
\left\{\begin{aligned}
& \partial_tu_k=\underbrace{A_{\varepsilon_k}\bigl[f(\Theta_k^+)-g(\Theta_k^-)\bigr]}_{S_k}+r_k,\\[3pt]
&\|\1_GS_k\|_{L^2(Q_T)}+\|r_k\|_{L^2(Q_T)}\longrightarrow0.
\end{aligned}\right.
\end{equation}
We call $\1_GS_k$ the \emph{profile observation part}, controlled through profile matching with symmetry, and $r_k$ the \emph{exact-wave realization error part}, arising from realizing the propagated profiles by exact waves.
\end{itemize}

\begin{remark}
The generated approximate velocities $S_k$ in AOS describe the leading parts of the time derivatives of the high-frequency waves $u_k$. The profile matching is exact, whereas its realization by solutions is asymptotic.
The leading-order nature of the construction suggests that the AOS approach can also accommodate lower-order perturbations; see \cite{AOSII} for a discussion in the glancing setting.
\end{remark}

\Needspace{10\baselineskip}
\subsection{A visible example in the disc}\label{subsec:intro-disc}

We explain the radial construction in Theorem~\ref{thm:disc}. The intuition for AOS has five parts, followed by its \hyperref[eq:gcc]{GCC} realization.

\noindent{\it 1. From one-dimensional propagation to the disc.}
In polar coordinates the wave equation reads
\begin{equation*}
u_{tt}-u_{rr}=\frac1r u_r+\frac1{r^2}u_{\theta\theta} \textrm{ \; with \; }
 u(t,1,\theta)=0.
\end{equation*}
The additional terms and the circular boundary prevent us from directly repeating the one-dimensional translation formula. Instead, we construct solutions of \eqref{eq:wave} whose \emph{leading} time derivatives retain a prescribed profile structure.
\vspace{2mm}

\noindent{\it 2. Profile realization in opposite sectors.}
We start from a general profile and localize its generated expression in two opposite sectors.

\begin{adjustwidth}{5mm}{5mm}
\begin{observation}
Fix $T>0$ and a nonzero real $2\pi$-periodic ``profile'' $q\in L^2(0,2\pi)$ satisfying $q(s+\pi)=-q(s)$. Given shrinking angular intervals $\calA_k,\calB_k\subset\R/2\pi\Z$ centered at opposite angles, we can coordinate modes and scales to obtain solutions $u_k$ of \eqref{eq:wave} such that
\begin{gather*}
\partial_tu_k=S_k[q]+r_k \textrm{ \; and \; } E[u_k](0)=1,\\
\|r_k\|_{L^2((0,T)\times\D)}\longrightarrow0 \textrm{ \;  as \; } k\rightarrow \infty,\\
S_k[q]=0\quad\text{outside }(0,T)\times(0,1)\times(\calA_k\cup\calB_k).
\end{gather*}
Section~\ref{subsubsec:disc-generated} defines this expression and explains its sector support; Section~\ref{subsec:disc-real} proves its realization by exact waves.
\end{observation}
\end{adjustwidth}

\vspace{2mm}
\noindent{\it 3. Additional cancellation through profile matching.}
We can now obtain additional cancellation \emph{inside} the first sector, where $\theta\in\calA_k$. There the generated expression is
\begin{equation}\label{eq:intro-disc-prof}
S_k[q](t,r,\theta)=A_{\varepsilon_k}(r,\theta)
\Bigl[q\bigl(K_k(t-r)\bigr)-q\bigl(K_k(t+r)+\pi/2\bigr)\Bigr].
\end{equation}
Choose $q=q_0$ from \eqref{eq:square}. Its constant intervals give a ``symmetry'' through ``profile matching'': the two nonzero branch values agree on moving regions $\Gamma_k\subset(0,T)\times(0,1)\times\calA_k$, and hence $\1_{\Gamma_k}S_k=0$. See Figure~\ref{fig:disc-localization}. We retain $\Gamma_k$ in the first sector and leave the opposite sector unobserved.

\begin{figure}[htbp]
\centering
\includegraphics[
  width=\textwidth,
  trim={0mm 9mm 0mm 12mm},
  clip]{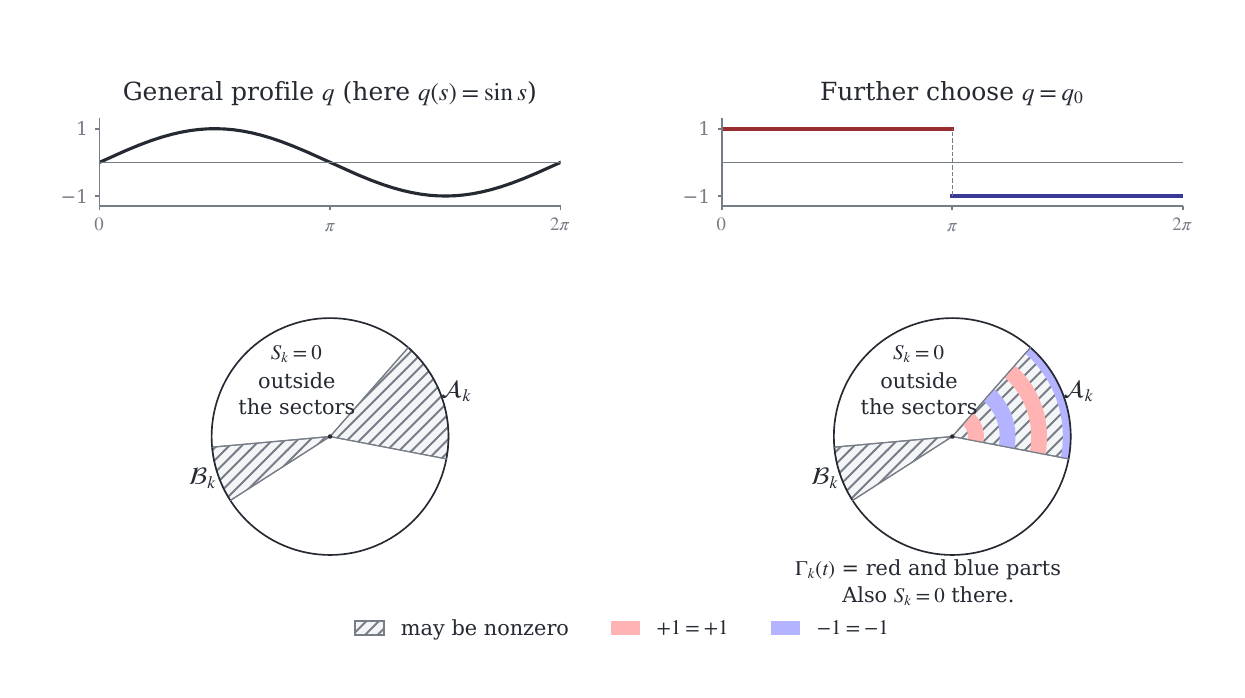}
\caption{Left: a general antiperiodic profile $q$ generates $S_k[q]$ that vanishes outside the two outlined sectors. Right: choosing $q=q_0$ also gives the colored matching region $\Gamma_k(t)$ in the first sector. Red and blue mean equal branch values $+1$ and $-1$, respectively; in both colors $S_k=0$. Both panels describe generated approximate velocities $S_k$; exact wave derivatives retain an $L^2$ realization error.}\label{fig:disc-localization}
\end{figure}

\vspace{2mm}
\noindent{\it 4. Scales and a fixed observation region.}
The same profile symmetry can be placed at every \textbf{scale} $K_k^{-1}$, in pairwise disjoint sector pairs. Outside the first $N$ pairs we observe everything; inside them we retain only $\Gamma_1,\ldots,\Gamma_N$. This defines $G^{(N)}$; see Figure~\ref{fig:disc-scales}. Adding a pair does not change earlier matching regions, so
\begin{gather*}
G^{(N+1)}\subset G^{(N)} \;\;\forall N,\textrm{ \;  \; and \; } G:=\bigcap_{N\ge1}G^{(N)},\\
\1_{G^{(N)}}S_k=0, \quad\forall k\in\{1,\ldots,N\}.
\end{gather*}

\begin{figure}[htbp]
\centering
\includegraphics[ width=\textwidth,
  trim={0mm 8mm 0mm 13mm},
  clip]{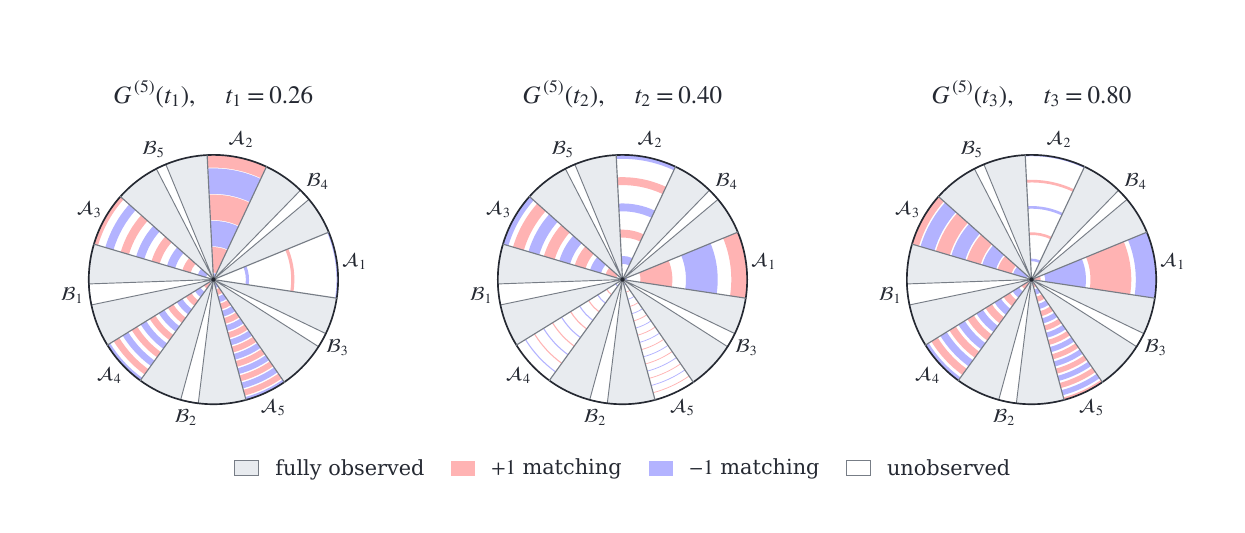}
\caption{Three time slices of the same $G^{(5)}$, with all five sector pairs and radial scales fixed. Gray is fully observed; red and blue are the matching parts of the first sectors; white is unobserved, including every opposite sector.}\label{fig:disc-scales}
\end{figure}

\noindent{\it 5. Asymptotics and non-observability.}
As $G^{(k)}$ decreases to $G$, the realization error $r_k$ tends to zero in $L^2$. Using $G\subset G^{(k)}$ and $\1_{G^{(k)}}S_k=0$, we obtain
\begin{align*}
\|\1_G\partial_tu_k\|_{L^2((0,T)\times\D)}
&\le\|\1_{G^{(k)}}\partial_tu_k\|_{L^2((0,T)\times\D)}\\*
&=\|\1_{G^{(k)}}r_k\|_{L^2((0,T)\times\D)}
\le\|r_k\|_{L^2((0,T)\times\D)}\longrightarrow0.
\end{align*}

\vspace{2mm}
\noindent{\it GCC realization.} It remains to make $G$ satisfy \eqref{eq:gcc}. For this purpose, we may place all first sectors in one short angular arc and all opposite sectors in its half-turn image, with the latter much narrower. This confines the wholly unobserved opposite sectors to a narrow region. Together with profile matching where $\theta\in\calA_k$ and suitable radial scales, the ray estimates in Section~\ref{sec:disc-class} give a uniform positive observation time.

\subsection{Literature review}\label{subsec:literature}

The literature on observability and controllability for wave equations is too extensive to review in full. We briefly discuss further contributions beyond those recalled in Section~\ref{subsec:intro-obs}. For a recent survey on the geometric characteristics of observable regions, we refer to Wang \cite{WangICM2026}.

\subsubsection{Control and related results for waves}
The geometry of wave propagation also plays an important role in decay, stabilization and observation in moving or partially observed regions.
Lebeau \cite{Lebeau1996} studies decay for damped waves, and Burq \cite{Burq1998} treats local-energy decay in exterior obstacle problems. Joly--Laurent \cite{JL2013} later prove exponential stabilization under GCC for subcritical, defocusing, analytic nonlinearities, while Burq--G\'erard \cite{BG2020} treat rough damping on the torus. Other works address moving observation regions \cite{CCM2014,LRLTT2017}, time-dependent domains \cite{Shao2019}, and regional or partial observability when global GCC fails \cite{DEZ2025}. Related geometric conditions also arise for degenerate Kolmogorov-type parabolic equations \cite{BHHR2015}.

Further questions in wave control concern nonlinear dynamics and qualitative or quantitative unique continuation. Zuazua \cite{Zuazua1990} establishes exact controllability for a class of semilinear wave equations. For bilinear control, Beauchard--Perrin--Pozzoli \cite{BPP2026} establish global approximate controllability from nonzero initial states for a Klein--Gordon equation on tori and determine the minimum control time in dimensions one and two. On the unique continuation side, Tataru \cite{Tataru1995} and Robbiano--Zuily \cite{RZ1998} treat operators with partially analytic coefficients. Laurent--L\'eautaud \cite{LL2019} develop quantitative versions yielding approximate wave control with cost estimates.

\vspace{2mm}

Besides the observation geometry, observability and its quantitative estimates depend on the structure and regularity of the coefficients. Zhang \cite{Zhang2000Potential,Zhang2000} establishes explicit observability estimates for wave equations with potentials and more general lower-order terms through explicit Carleman estimates. Duyckaerts--Zhang--Zuazua \cite{DZZ2008} study the dependence on potentials and the optimality of observability estimates for parabolic and hyperbolic systems. 
Very recently, Liao--L{\"u}--Zhang \cite{LLZ2026} construct, on a suitably chosen planar domain, a smooth real-valued time-periodic potential and an open observation set satisfying the classical GCC, together with a nonzero smooth Dirichlet wave that vanishes on the observed cylinder for all time. Their counterexample concerns a failure of unique continuation. 

Under geometric control, Laurent--L\'eautaud \cite{LL2016} quantify the dependence on a bounded potential and on observation times near the geometric control time. At lower coefficient regularity, Dehman--Ervedoza \cite{DE2017} study wave observability with rough coefficients, while Burq--Dehman--Le Rousseau \cite{BDLR2024} prove observability for $C^1$ metrics on $C^2$ manifolds with boundary under a geometric condition.

For geometric wave equations, the geometry and topology of the target impose additional constraints on controllability. For wave maps from a circle to a sphere, Krieger--Xiang \cite{KX2024} prove semi-global controllability and stabilization below the first nonconstant harmonic-map energy. Coron--Krieger--Xiang \cite{CKX2025} then establish global exact controllability for spheres of dimension at least two. 
The subsequent work \cite{CKXgeneral2025} treats compact orientable Riemannian targets without boundary, establishing global controllability precisely between states with homotopic position maps. We refer to \cite{Coron2007,TW2009} for broader accounts.

\subsubsection{Observability on measurable spatial sets}

For cylindrical observation $(0,T)\times\omega$, a basic question is which conditions on a merely measurable spatial set $\omega$ ensure observability.
On $\R^n$, thickness of $\omega$, namely uniform positive density in translates of a fixed box, characterizes heat null controllability \cite{EV2018,WWZZ2019}. Wang--Wang--Zhang--Zhang \cite{WWZZ2019} also establish its equivalence with interpolation and spectral inequalities. Burq--Moyano \cite{BM2023} prove spectral inequalities from measurable spatial sets and apply them to heat observation and control with Lipschitz coefficients. For Schr\"odinger equations on two-dimensional tori, Bourgain--Burq--Zworski \cite{BBZ2013} prove observability from every nonempty open spatial set for real-valued $L^2$ potentials. For the free equation, Burq--Zworski \cite{BZ2019} obtain observability from every spatial set of positive measure. Related observability and unique continuation inequalities for the free equation on $\R^n$, involving one or two observation times, are obtained by Wang--Wang--Zhang \cite{WWZ2019}.

\subsubsection{Observability on spacetime measurable sets}

Allowing observation to be restricted in both time and space leads to observability problems on measurable spacetime sets. Phung--Wang \cite{PW2013} prove parabolic observability on products $E\times\omega$, with $E\subset(0,T)$ of positive measure and $\omega$ nonempty and open. Going beyond product sets, Apraiz--Escauriaza--Wang--Zhang \cite{AEWZ2014} obtain heat observability from arbitrary positive-measure subsets of spacetime on suitable domains.

For Schr{\"o}dinger, KdV and wave equations, these questions involve different mechanisms of dispersion, propagation and cancellation. For Schr\"odinger equations on tori with bounded potentials, Burq--Zhu \cite{BurqZhu2025} prove observability from arbitrary positive-measure spacetime sets in one dimension, with conditional criteria and results for certain measurable regions in higher dimensions. On products of positive-measure spatial and temporal sets, Niu--Wang--Xiang \cite{NWXKdV2026} prove local exact controllability near equilibria for the periodic KdV equation. For waves, the observable symmetries in \cite{NWX,HX} and the AOS constructions of the present paper identify cancellation obstructions not detected by ray coverage alone.

\subsection{History of the project and use of artificial intelligence}\label{subsec:history-ai}

This project began in September 2025, shortly after the author and his coauthors completed the first draft of \cite{NWX}. The original goal was substantially broader: to investigate, using ideas from harmonic analysis and geometric measure theory, whether GCC together with symmetry could fully characterize spacetime observability. The first stage explored reflection phenomena in the interval, rectangle, and disc settings, with the expectation that the underlying symmetry structure should extend to convex domains. It was in the disc setting that the author formulated the notion of asymptotic observable symmetry (AOS). AI tools were used at this stage primarily for technical calculations.

In early August 2026, an AI-assisted exploration pointed to a different propagation structure in the disc, involving glancing propagation and a caustic. This propagation picture was new to the author, and its possible connection with AOS had not previously been considered. After studying the resulting glancing construction, the author realized that its essential ingredients were local and could be adapted to general smooth bounded domains, without a convexity assumption. 
The resulting general-domain construction, presented in the sequel \cite{AOSII} coauthored by the author and Niu, is considerably more involved, and AI tools contributed extensively to its technical development, including asymptotic calculations, parameter analysis, and intermediate derivations. Their use made it practical to develop this technically demanding construction to its present level of detail within the timescale of the project; without such assistance, the author would likely have continued exploring these questions for much longer.  AI tools were also used in both papers for figure preparation and LaTeX typesetting.

\medskip\noindent\emph{Organization of the paper.}
Section~\ref{sec:aos} develops the AOS framework. Section~\ref{sec:disc} constructs the radial disc example, and Section~\ref{sec:disc-class} proves \hyperref[eq:gcc]{GCC} for its fixed observation region. Appendix~\ref{sec:supp-aos} provides supplementary discussion of compatible profiles, velocity approximation, and localization. Appendix~\ref{sec:meas} explains spacetime measurable observation and GCC. The sequel \cite{AOSII} treats glancing and general smooth bounded planar domains.

\begingroup
\renewcommand{\addcontentsline}[3]{}
\subsection*{Acknowledgments}
Shengquan Xiang is partially  supported by NSFC 12301562 and by the Shanghai Qiguang Natural Science Development Foundation.
\endgroup

\section{Asymptotic observable symmetry}\label{sec:aos}

As illustrated in Section~\ref{subsec:intro-aos}, AOS retains a fixed profile symmetry at different scales, while exact waves realize the resulting family asymptotically. We first recall the torus construction, then develop the four core components of AOS. We conclude with the definition and its consequence for observability.

\subsection{Observable symmetry in the torus case}\label{subsec:aos-torus}

Recall the construction of \cite{NWX} on $\T:=\R/2\pi\Z$. To generate two opposite travelling branches from the same real \emph{profile} $q$, we use the \emph{profile space}
\begin{equation*}
X:=\left\{q\in L^2(\T;\R):\int_\T q(s)\dd s=0\right\}.
\end{equation*}
The zero-mean condition gives a periodic primitive $Q\in H^1(\T)$, that is, $q=Q'$. Thus $u(t,x):=Q(x+t)+Q(x-t)$ is an exact energy solution of $u_{tt}-u_{xx}=0$. Introduce its \emph{propagation phases}
\begin{equation}\label{eq:tor-pha}
\Theta^+(t,x):=x+t \textrm{ \; and \; }
\Theta^-(t,x):=x-t \quad\text{in }\T.
\end{equation}
Using \eqref{eq:tor-pha}, we write the observed derivative as
\begin{equation}\label{eq:tor-gen}
\partial_tu(t,x)=q\bigl(\Theta^+(t,x)\bigr)-q\bigl(\Theta^-(t,x)\bigr).
\end{equation}
Both branch profiles are $q$, with signed \emph{amplitudes} $1$ and $-1$.
\vspace{2mm}

We now choose $q=q_0$, the square profile in \eqref{eq:square}. Set $P=[0,\pi)\subset\T$, so that $q_0=2\1_P-1$. Equal profile values occur when both reference points\footnote{Here ``reference'' refers to the variable $s\in\T$ in which $q$ is specified, before composition with the phases.} belong to $P$ or both belong to $P^c$. This gives the two matching regions
\begin{equation}\label{eq:tor-match}
\begin{split}
\calM_{q_0}&:=(P\times P)\cup(P^c\times P^c),\\
G_\T&:=\{(t,x)\in(0,2\pi)\times\T:\bigl(\Theta^+(t,x),\Theta^-(t,x)\bigr)\in\calM_{q_0}\}.
\end{split}
\end{equation}
Consequently,
\begin{equation*}
\1_{G_\T}\partial_tu
=\1_{G_\T}\bigl[q_0(\Theta^+)-q_0(\Theta^-)\bigr]=0.
\end{equation*}
Figure~\ref{fig:torus} shows this matching in spacetime. Along either characteristic $x=x_0\pm t$, one phase is fixed while the other traverses $\T$ twice during $0<t<2\pi$.

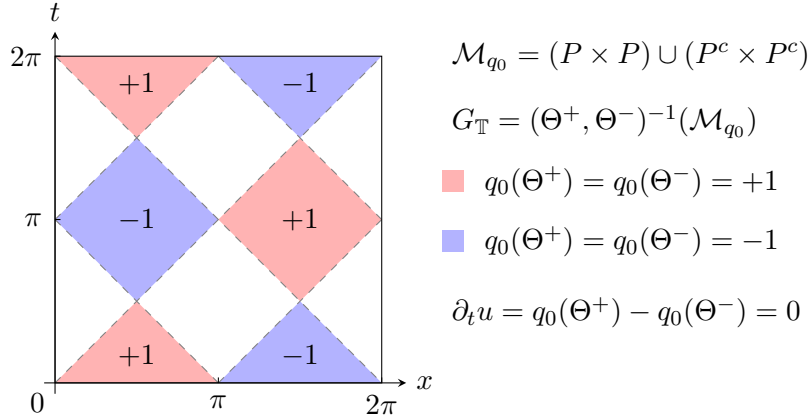
\begin{figure}[htbp]
\centering
\begin{tikzpicture}[x=0.72cm,y=0.72cm,>=stealth]
  \fill[red!30] (0,0) -- (3,0) -- (1.5,1.5) -- cycle;
  \fill[blue!30] (3,0) -- (6,0) -- (4.5,1.5) -- cycle;
  \fill[blue!30] (0,3) -- (1.5,1.5) -- (3,3) -- (1.5,4.5) -- cycle;
  \fill[red!30] (3,3) -- (4.5,1.5) -- (6,3) -- (4.5,4.5) -- cycle;
  \fill[red!30] (0,6) -- (1.5,4.5) -- (3,6) -- cycle;
  \fill[blue!30] (3,6) -- (4.5,4.5) -- (6,6) -- cycle;

  \draw[dashed,gray] (0,3) -- (3,0);
  \draw[dashed,gray] (0,6) -- (6,0);
  \draw[dashed,gray] (3,6) -- (6,3);
  \draw[dashed,gray] (0,3) -- (3,6);
  \draw[dashed,gray] (0,0) -- (6,6);
  \draw[dashed,gray] (3,0) -- (6,3);
  \draw (0,0) rectangle (6,6);
  \draw[->] (-0.2,0) -- (6.45,0) node[right] {$x$};
  \draw[->] (0,-0.2) -- (0,6.45) node[above] {$t$};
  \draw (3,0) -- (3,0.1) node[below=3pt] {$\pi$};
  \draw (6,0) -- (6,0.1) node[below=3pt] {$2\pi$};
  \draw (0,3) -- (0.1,3) node[left=3pt] {$\pi$};
  \draw (0,6) -- (0.1,6) node[left=3pt] {$2\pi$};
  \node[below left] at (0,0) {$0$};

  \node at (1.5,0.5) {$+1$};
  \node at (4.5,0.5) {$-1$};
  \node at (1.5,3) {$-1$};
  \node at (4.5,3) {$+1$};
  \node at (1.5,5.5) {$+1$};
  \node at (4.5,5.5) {$-1$};

  \node[anchor=west] at (7.1,6.0) {$\calM_{q_0}=(P\times P)\cup(P^c\times P^c)$};
  \node[anchor=west] at (7.1,4.8) {$G_\T=(\Theta^+,\Theta^-)^{-1}(\calM_{q_0})$};
  \fill[red!30] (7.1,3.5) rectangle (7.5,3.9);
  \node[anchor=west] at (7.7,3.7) {$q_0(\Theta^+)=q_0(\Theta^-)=+1$};
  \fill[blue!30] (7.1,2.4) rectangle (7.5,2.8);
  \node[anchor=west] at (7.7,2.6) {$q_0(\Theta^+)=q_0(\Theta^-)=-1$};
  \node[anchor=west] at (7.1,1.3) {$\partial_tu=q_0(\Theta^+)-q_0(\Theta^-)=0$};
\end{tikzpicture}

\caption{The torus counterexample. Red and blue cells mark the common profile values $+1$ and $-1$, respectively. Their union is $G_\T$, where $\partial_tu=0$, although the observation region satisfies GCC.}\label{fig:torus}
\end{figure}

\subsection{Propagated profiles}\label{subsec:aos-profiles}

In \eqref{eq:tor-gen}, both branches read the same profile. Other propagation geometries can change the profile when passing from one branch to the other: choosing the first branch profile $f$ determines the second, $g$, and the boundary may impose a compatibility relation between them.

We write the \emph{branch profiles} as $f_\alpha:=L_\alpha q$, $\alpha\in\{1,\ldots,J\}$, where $L_\alpha$ are fixed linear maps defined on $X$.\footnote{The torus construction and the caustic construction in \cite[Section~2]{AOSII} use, respectively,
\begin{equation*}
(f,g)=(q,q) \textrm{ \; and \; }
 (f,g)=(q,\calH q).
\end{equation*}
The corresponding maps are $(L_1,L_2)=(\Id,\Id)$ and $(\Id,\calH)$. In the second model, $\calH$ has multiplier $-i\sgn n$ on $e^{2\pi in\theta}$. The chosen quarter-period boundary shift also requires $f=\tau_{1/4}g$, where $\tau_aq(\theta)=q(\theta+a)$. Thus $q$ must satisfy $q=\tau_{1/4}\calH q$.} A branch at scale $\varepsilon$ reads its branch profile at the propagation phase and carries an amplitude:
\begin{equation*}
\underbrace{q\in X}_{\text{profile}}
\ \longmapsto\
\underbrace{f_\alpha:=L_\alpha q}_{\text{branch profile}}
\ \longmapsto\
\underbrace{A_{\alpha,\varepsilon}(t,x)f_\alpha\bigl(\Theta_{\alpha,\varepsilon}(t,x)\bigr)}_{\text{branch contribution on spacetime}}.
\end{equation*}
We collect this information in the following definition. $\calT_\varepsilon q$ is designed to approximate $u_t$.

\begin{defi}\label{def:model}
Fix a Hilbert space\footnote{All generated approximate velocities and waves are real-valued; the real part in \eqref{eq:reconstruct} also allows complex Fourier representations of their branches. We regard $X$ as a real Hilbert space when using such representations.} $X$ of profiles, independently of scale. Fix a profile $q\in X$ and write its branch profiles as $f_\alpha:=L_\alpha q$. A propagation model specifies a family of real-linear operators, one for each admissible scale $\varepsilon$, of the form
\begin{equation}\label{eq:reconstruct}
\calT_\varepsilon q(t,x)
 =\Real\sum_{\alpha=1}^J A_{\alpha,\varepsilon}(t,x)
 f_\alpha\bigl(\Theta_{\alpha,\varepsilon}(t,x)\bigr),
 \qquad\text{where }f_\alpha:=L_\alpha q.
\end{equation}
The data and requirements are as follows:
\begin{enumerate}[label=\textup{(\roman*)},leftmargin=9mm]
\item The number of branches $J$ and the maps $L_\alpha$, defined on $X$, are fixed.
\item The propagation phases $\Theta_{\alpha,\varepsilon}$ and branch amplitudes $A_{\alpha,\varepsilon}$ may vary with the scale.
\item The operators $\calT_\varepsilon:X\to L^2((0,T)\times\Omega)$ satisfy, for every $q\in X$,
\begin{equation*}
\|\calT_\varepsilon q\|_{L^2((0,T)\times\Omega)}\le C\|q\|_X.
\end{equation*}
\item We group branch contributions localized in the same region into a {\it packet}. For their index set $I\subseteq\{1,\ldots,J\}$, the packet is
\begin{equation*}
\Real\sum_{\alpha\in I}A_{\alpha,\varepsilon}(t,x)f_\alpha\bigl(\Theta_{\alpha,\varepsilon}(t,x)\bigr).
\end{equation*}
\end{enumerate}
\end{defi}

The representation above specifies the branch structure. Velocity approximation and localization require appropriate choices of profiles, phases and amplitudes, consistent with propagation and the boundary condition. See Appendix~\ref{sec:supp-aos} for supplementary details.

\subsection{Symmetry relations}\label{subsec:aos-symmetry}

We next identify the geometric relation that makes the propagated profiles cancel. There are two regions to distinguish:
\begin{itemize}[leftmargin=6mm]
\item The matching region $\calM_q$ lies in the profile variables.
\item The matching region $\Gamma_\varepsilon$ lies in spacetime.
\end{itemize}

Again we look at \cite{NWX}. Let $P_+$ and $P_-$ be measurable sets in $\T$. Their transported sets $(\Theta^+)^{-1}(P_+)$ and $(\Theta^-)^{-1}(P_-)$ lie in spacetime. Observable symmetry compares their traces on $G$:
\begin{equation*}
\1_G\1_{P_+}(\Theta^+)=\1_G\1_{P_-}(\Theta^-).
\end{equation*}
Taking $P_+=P_-=P$ and $G=G_\T$ from \eqref{eq:tor-match} recovers the example.
\vspace{2mm}

\noindent\textit{Step 1. Balance between branch profiles: $\calM_q$.}
Fix the weights $w_\alpha$ of the symmetry between the branch profiles. Write $\theta=(\theta_1,\ldots,\theta_J)$ for the tuple of profile coordinates and choose a matching region
\begin{equation}\label{eq:ref-balance}
\calM_q\subseteq\left\{\theta:\Real\sum_{\alpha=1}^J w_\alpha(\theta)f_\alpha(\theta_\alpha)=0\right\}.
\end{equation}
We require a nontrivial balance: on a positive-measure part of $\calM_q$, at least two real branch contributions are nonzero and cancel. For two real branch profiles with $w_1=1$ and $w_2=-1$, this is the equality
\begin{equation*}
f_1(\theta_1)=f_2(\theta_2),
\qquad \forall(\theta_1,\theta_2)\in\calM_q.
\end{equation*}
Thus $\calM_q$ records which profile values are matched, before choosing a physical scale.
\vspace{2mm}

\noindent\textit{Step 2. Matching in spacetime: $\Gamma_\varepsilon$.}
We now pass from profile coordinates to spacetime. At scale $\varepsilon$, the tuple of phases and the spacetime matching region are
\begin{equation}\label{eq:matching-regions}
\left\{\begin{aligned}
\Theta_\varepsilon(t,x)&:=\bigl(\Theta_{1,\varepsilon}(t,x),\ldots,\Theta_{J,\varepsilon}(t,x)\bigr),\\
\Gamma_\varepsilon&:=\{(t,x)\in(0,T)\times\Omega:\Theta_\varepsilon(t,x)\in\calM_q\}.
\end{aligned}\right.
\end{equation}

\Needspace{6\baselineskip}
\noindent\textit{Step 3. Preservation of the balance by amplitudes $A_{\alpha,\varepsilon}$.}
After composition with the phases, the weights in \eqref{eq:ref-balance} become $w_\alpha\bigl(\Theta_\varepsilon(t,x)\bigr)$. A sufficient condition on $\Gamma_\varepsilon$ is
\begin{equation}\label{eq:balance-amp}
A_{\alpha,\varepsilon}(t,x)=a_\varepsilon(t,x)w_\alpha\bigl(\Theta_\varepsilon(t,x)\bigr),
\qquad a_\varepsilon(t,x)\in\R.
\end{equation}
The \textit{real factor} $a_\varepsilon(t,x)$ may depend on spacetime and scale. Substituting \eqref{eq:balance-amp} into \eqref{eq:reconstruct} and using \eqref{eq:ref-balance}, we obtain
\begin{equation}\label{eq:balance-transfer}
\calT_\varepsilon q
=a_\varepsilon\Real\sum_{\alpha=1}^J
w_\alpha(\Theta_\varepsilon)f_\alpha(\Theta_{\alpha,\varepsilon})=0
\quad\text{on }\Gamma_\varepsilon.
\end{equation}
When $\calT_\varepsilon q$ consists of several disjoint packets, the matching region for a given balance is restricted to the corresponding packet region.

\begin{example}\label{ex:matching}
In the two-branch case $J=2$, suppose the real profiles $f=L_1q$ and $g=L_2q$ take finitely many values. Write
\begin{gather*}
f=\sum_c c\,\1_{F_c},\qquad g=\sum_c c\,\1_{J_c},\\
F_c:=\{f=c\},\quad J_c:=\{g=c\}.
\end{gather*}
Choose $w_1=1$ and $w_2=-1$. Equal values give the reference matching region
\begin{equation*}
\mathcal M_q=\bigcup_c(F_c\times J_c),\qquad
f(\theta_1)-g(\theta_2)=0\quad\text{on }F_c\times J_c.
\end{equation*}
At scale $\varepsilon$, each product $F_c\times J_c$ gives the spacetime set
\begin{equation*}
\{(t,x):\Theta_{1,\varepsilon}(t,x)\in F_c,\ 
\Theta_{2,\varepsilon}(t,x)\in J_c\}.
\end{equation*}
Taking their union gives $\Gamma_\varepsilon$, where the two transported profile values agree. To retain their cancellation, suppose that the amplitudes are
\begin{equation*}
A_{1,\varepsilon}(t,x)=A_\varepsilon(t,x),\qquad
A_{2,\varepsilon}(t,x)=-A_\varepsilon(t,x),\qquad A_\varepsilon(t,x)\in\R.
\end{equation*}
Consequently,
\begin{equation*}
\calT_\varepsilon q(t,x)
=A_\varepsilon(t,x)\Bigl[f\bigl(\Theta_{1,\varepsilon}(t,x)\bigr)-g\bigl(\Theta_{2,\varepsilon}(t,x)\bigr)\Bigr]
=0\quad\text{on }\Gamma_\varepsilon.
\end{equation*}
\end{example}

\subsection{A scaled family of generated approximate velocities}\label{subsec:aos-scales}

Choose an admissible coordinated family of propagation parameters and label it by a scalar $\varepsilon$. Along a sequence $\varepsilon_k\to0^+$, define the \emph{generated approximate velocities}
\begin{equation*}
S_k:=\calT_{\varepsilon_k}q
=\Real\sum_{\alpha=1}^J A_{\alpha,\varepsilon_k}
f_\alpha(\Theta_{\alpha,\varepsilon_k}).
\end{equation*}
The phases $\Theta_{\alpha,\varepsilon_k}$ and amplitudes $A_{\alpha,\varepsilon_k}$ change with the scale.

In Example~\ref{ex:matching}, the reference matching region gives a physical region at every scale:
\begin{equation*}
\Gamma_k:=\{(t,x):\bigl(\Theta_{1,\varepsilon_k}(t,x),\Theta_{2,\varepsilon_k}(t,x)\bigr)\in\mathcal M_q\}.
\end{equation*}
Consequently,
\begin{equation*}
S_k=A_{\varepsilon_k}\bigl[f(\Theta_{1,\varepsilon_k})-g(\Theta_{2,\varepsilon_k})\bigr]=0
\quad\text{on }\Gamma_k.
\end{equation*}
Thus the matching cells can move and become finer; see Figure~\ref{fig:aos-family}.

\begin{figure}[htbp]
\centering
\includegraphics[width=.9\textwidth,
  trim={0mm 8mm 0mm 2mm},
  clip]{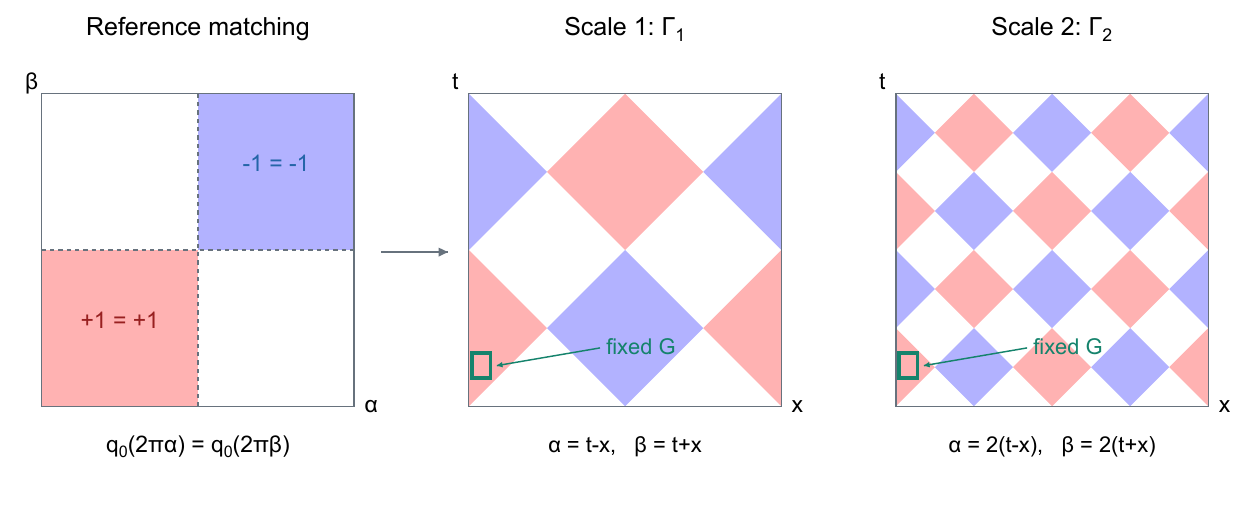}
\caption{One reference matching region, two physical scales. Red and blue denote matching values $+1$ and $-1$. The reference coordinates have period one, with $q_0(2\pi\alpha)=q_0(2\pi\beta)$. Different propagation phases give $\Gamma_1$ and $\Gamma_2$; the outlined patch illustrates a common part of a fixed observation region. The cells describe cancellation of generated approximate velocities.}\label{fig:aos-family}
\end{figure}

\begin{remark}\label{rem:scales}
In Section~\ref{sec:disc}, the profile $q$ and its half-turn relation remain fixed. The radial and temporal matching scale is $K_k^{-1}$, while the packet concentration scale and sector widths satisfy
\begin{equation*}
Q_k^{-1}\ll b_k\le a_k\longrightarrow0.
\end{equation*}
We choose the sectors, then the finite spectral cutoffs, and finally $K_k$ sufficiently large for exact-wave realization. We label the resulting family by $\varepsilon_k:=K_k^{-1}$.
\end{remark}

\subsection{Asymptotic realization}\label{subsec:aos-real}

The matching regions $\Gamma_k$ are now available, but observability concerns one fixed set $G$. We therefore arrange the matching regions together. Write $Q_T=(0,T)\times\Omega$. Each $S_k$ may contain several packets. Let $\mathcal U_k\subset Q_T$ be a measurable region containing all their supports, so that
\begin{equation*}
S_k=\1_{\mathcal U_k}S_k,\qquad\text{namely }S_k=0\text{ outside }\mathcal U_k.
\end{equation*}
In the radial disc construction, $\mathcal U_k$ is the union of the two packet regions, given by the spacetime lifts of opposite sectors; see Figure~\ref{fig:aos-packets}.

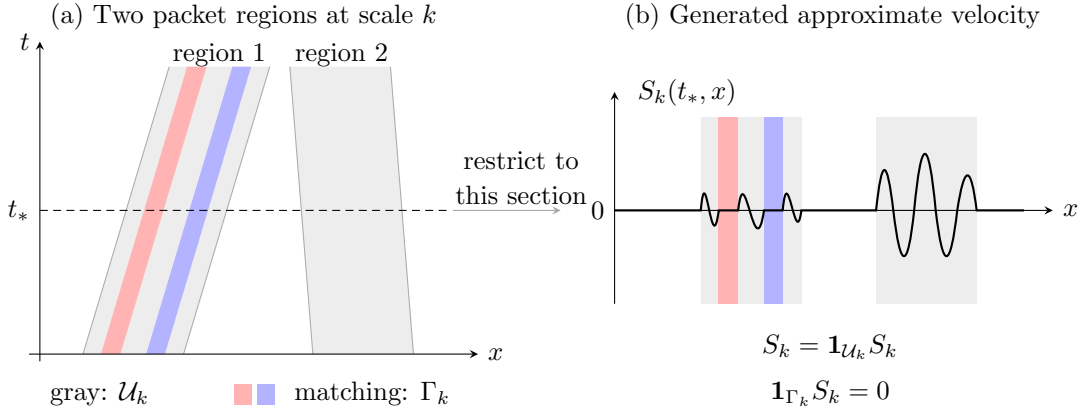
\begin{figure}[htbp]
\centering
\begin{tikzpicture}[x=0.95cm,y=0.95cm,>=stealth,font=\small]
  \node[anchor=west] at (0,4.7) {(a) Two packet regions at scale $k$};
  \node[anchor=west] at (8,4.7) {(b) Generated approximate velocity};

  \fill[gray!14] (0.6,0) -- (2,0) -- (3.2,4) -- (1.8,4) -- cycle;
  \fill[gray!14] (3.8,0) -- (5.2,0) -- (4.88,4) -- (3.48,4) -- cycle;
  \fill[red!30] (0.85,0) -- (1.12,0) -- (2.32,4) -- (2.05,4) -- cycle;
  \fill[blue!30] (1.48,0) -- (1.74,0) -- (2.94,4) -- (2.68,4) -- cycle;
  \draw[gray!65] (0.6,0) -- (1.8,4) (2,0) -- (3.2,4);
  \draw[gray!65] (3.8,0) -- (3.48,4) (5.2,0) -- (4.88,4);
  \draw[->] (-0.15,0) -- (6.1,0) node[right] {$x$};
  \draw[->] (0,-0.1) -- (0,4.35) node[left] {$t$};
  \draw[densely dashed] (0,2) node[left] {$t_*$} -- (5.7,2);
  \draw[->,gray!70] (5.75,2) -- (7.25,2);
  \node[align=center] at (6.7,2.45) {restrict to\\this section};
  \node at (2.5,4.17) {region 1};
  \node at (4.2,4.17) {region 2};
  \node[anchor=west] at (0,-0.55) {gray: $\mathcal U_k$};
  \fill[red!30] (2.7,-0.7) rectangle (2.95,-0.42);
  \fill[blue!30] (3.02,-0.7) rectangle (3.27,-0.42);
  \node[anchor=west] at (3.4,-0.55) {matching: $\Gamma_k$};

  \begin{scope}[xshift=7.6cm]
    \fill[gray!14] (1.2,0.7) rectangle (2.6,3.3);
    \fill[gray!14] (3.64,0.7) rectangle (5.04,3.3);
    \fill[red!30] (1.45,0.7) rectangle (1.72,3.3);
    \fill[blue!30] (2.08,0.7) rectangle (2.34,3.3);
    \draw[->] (0,2) -- (6.1,2) node[right] {$x$};
    \draw[->] (0,0.7) -- (0,3.65);
    \node[anchor=west] at (0.18,3.65) {$S_k(t_*,x)$};
    \node[left] at (0,2) {$0$};
    \draw[thick] (0,2) -- (1.2,2)
      .. controls (1.27,2.8) and (1.35,1.25) .. (1.45,2)
      -- (1.72,2)
      .. controls (1.81,2.8) and (1.95,1.15) .. (2.08,2)
      -- (2.34,2)
      .. controls (2.41,2.75) and (2.52,1.4) .. (2.6,2)
      -- (3.64,2)
      .. controls (3.72,2.75) and (3.79,2.75) .. (3.88,2)
      .. controls (3.98,1.15) and (4.08,1.15) .. (4.17,2)
      .. controls (4.27,3.05) and (4.37,3.05) .. (4.47,2)
      .. controls (4.57,1.15) and (4.67,1.15) .. (4.77,2)
      .. controls (4.87,2.65) and (4.96,2.65) .. (5.04,2)
      -- (5.7,2);
    \node at (3,0.1) {$S_k=\1_{\mathcal U_k}S_k$};
    \node at (3,-0.55) {$\1_{\Gamma_k}S_k=0$};
  \end{scope}
\end{tikzpicture}

\caption{Schematic at one fixed scale $k$. Left: two packet regions with union $\mathcal U_k$ (gray), and the matching region $\Gamma_k$ (red and blue) on a spacetime section. Right: a trace of $S_k$ at the marked time $t_*$, vanishing outside $\mathcal U_k$ and on $\Gamma_k$.}\label{fig:aos-packets}
\end{figure}

\Needspace{12\baselineskip}
It is useful to keep in view the four ``asymptotic features'' of our constructions:
\begin{itemize}[leftmargin=6mm]
\item \emph{A family of functions.} As $\varepsilon_k\to0^+$, the generated approximate velocities $S_k$ carry the same profile symmetry. At each scale we construct a corresponding exact wave $u_k$.
\item \emph{Converging observation regions.} We arrange observation regions $G^{(N)}$ for the first $N$ scales so that they approach one fixed region $G$.
\item \emph{Small additional observation region.} We may allow $G^{(N)}$ to meet a packet outside $\Gamma_k$, in a region $\calE_k$ where the observed energy of $S_k$ tends to zero.
\item \emph{Exact-wave realization with asymptotic error.} The difference $r_k:=\partial_tu_k-S_k$ tends to zero.
\end{itemize}

\subsubsection{Finite-stage arrangement to find $G^{(N)}$}\label{subsubsec:aos-stages}

We first consider observation only on the matching part of each packet. Let $G^{(N)}$ accommodate the first $N$ generated approximate velocities $S_1,\ldots,S_N$, retaining the earlier matching at every later stage:\footnote{All set inclusions are understood modulo spacetime null sets.}
\begin{equation}\label{eq:aos-stages}
G^{(N)}\cap\mathcal U_k\subseteq\Gamma_k,
\qquad \forall k\in\{1,\ldots,N\}.
\end{equation}
Since $S_k=0$ on $\Gamma_k$ and outside $\mathcal U_k$, this gives
\begin{equation*}
\1_{G^{(N)}}S_k=0,\qquad \forall k\in\{1,\ldots,N\}.
\end{equation*}
The index $k$ fixes one function and its scale; $N$ counts how many such functions the observation region accommodates. We choose the finite-stage regions so that
\begin{equation}\label{eq:aos-set-limit}
G^{(N)}\longrightarrow G
\quad\text{via}\quad
|G^{(N)}\mathbin\triangle G|\longrightarrow0.
\end{equation}
To pass to $G$, fix $k$ and let $N\to\infty$. Indeed,
\begin{equation*}
|(G\cap\mathcal U_k)\setminus\Gamma_k|
\le |G\mathbin\triangle G^{(N)}|\longrightarrow0.
\end{equation*}
Thus $G\cap\mathcal U_k\subseteq\Gamma_k$ and $\1_GS_k=0$. The corresponding passage in $L^2$ is
\begin{equation*}
\|(\1_{G^{(N)}}-\1_G)S_k\|_{L^2(Q_T)}^2
=\int_{G^{(N)}\triangle G}|S_k|^2\dd x\dd t
\longrightarrow0
\qquad\text{for each fixed }k.
\end{equation*}

\Needspace{11\baselineskip}
\begin{remark}\label{rem:aos-limit}
The radial construction here and the glancing construction in \cite{AOSII} use decreasing finite-stage observation regions. In that case,
\begin{equation*}
G=\bigcap_{N\ge1}G^{(N)}\subseteq G^{(k)}
\end{equation*}
makes the inclusion immediate.
\end{remark}

\subsubsection{Additional observation regions $\calE_k$}\label{subsubsec:aos-additional}

A geometric estimate may require observing part of a packet outside its matching region. We then allow an \emph{additional observation region} $\calE_k$. In our constructions, $|\calE_k|\to0$. We replace \eqref{eq:aos-stages} by
\begin{equation}\label{eq:aos-stages-add}
G^{(N)}\cap\mathcal U_k\subseteq\Gamma_k\cup\calE_k,
\qquad \forall k\in\{1,\ldots,N\}.
\end{equation}
Observation inside $\Gamma_k$ still contributes zero. Hence
\begin{equation*}
\|\1_{G^{(N)}}S_k\|_{L^2(Q_T)}\le \|\1_{\calE_k}S_k\|_{L^2(Q_T)},
\quad \forall k\in\{1,\ldots,N\}.
\end{equation*}
Fixing $k$ and passing to $G$ as above gives the same bound on the fixed observation region. We therefore require
\begin{equation}\label{eq:aos-region}
G\cap\mathcal U_k\subseteq\Gamma_k\cup\calE_k
\quad\text{ and }\quad \|\1_{\calE_k}S_k\|_{L^2(Q_T)}\longrightarrow0.
\end{equation}
Consequently,
\begin{equation*}
\|\1_GS_k\|_{L^2(Q_T)}\le\|\1_{\calE_k}S_k\|_{L^2(Q_T)}\longrightarrow0.
\end{equation*}

The torus and radial disc examples take $\calE_k=\varnothing$. For example, in the glancing reduction of \cite[Section~2]{AOSII}, this region is near the corners of the matching rectangles.

\begin{figure}[htbp]
\centering
\begin{tikzpicture}[x=1cm,y=1cm,>=stealth,font=\small]
  \fill[gray!14] (0,0) rectangle (6.4,1.25);
  \draw[gray!65] (0,0) rectangle (6.4,1.25);
  \fill[red!30] (0.6,0) rectangle (1.9,1.25);
  \fill[blue!30] (3,0) rectangle (4.3,1.25);
  \fill[orange!45] (4.3,0) rectangle (4.8,1.25);
  \node at (1.25,0.64) {$\Gamma_k$};
  \node at (3.65,0.64) {$\Gamma_k$};
  \node[anchor=south] at (4.55,1.5) {$\calE_k$};
  \draw[->] (4.55,1.5) -- (4.55,1.05);
  \node[anchor=west] at (0,1.65) {A section through $\mathcal U_k$};
  \node[anchor=west,align=left] at (7.05,1.05)
    {$G\cap\mathcal U_k\subseteq\Gamma_k\cup\calE_k$};
  \node[anchor=west,align=left] at (7.05,0.35)
    {$\|\1_GS_k\|_{L^2}\le\|\1_{\calE_k}S_k\|_{L^2}$};
  \node[anchor=west] at (0,-0.48) {matching observation: $S_k=0$};
  \node[anchor=west] at (7.05,-0.48) {additional observation: small in $L^2$};
\end{tikzpicture}

\caption{Additional observation on a schematic section through one packet region. Matching observation contributes zero; the orange part lies in $\calE_k$ and need not do so. Its contribution $\|\1_{\calE_k}S_k\|_{L^2(Q_T)}$ tends to zero.}\label{fig:aos-additional}
\end{figure}
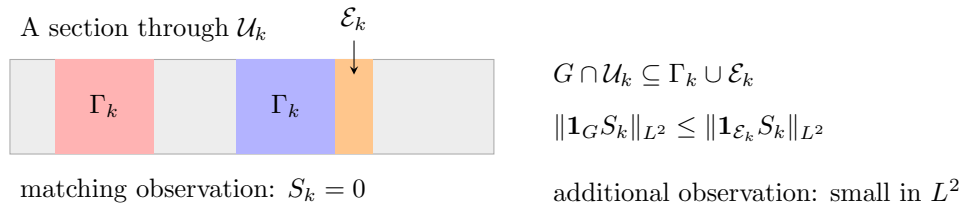

\subsubsection{Exact waves $u_k$ and their asymptotic errors}\label{subsubsec:aos-waves}

It remains to realize the generated approximate velocities $S_k$ by exact waves $u_k$. We divide this realization into two steps.
\begin{itemize}[leftmargin=6mm]
\item \emph{Construction of exact waves.} For convenience of construction, we first seek exact waves $u_k$ such that
\begin{equation*}
\partial_tu_k=S_k+\widetilde r_k \textrm{ \;\; with \; }
 \|\widetilde r_k\|_{L^2(Q_T)}\longrightarrow0.
\end{equation*}
We also need $0<c\le E[u_k](0)\le C$, with constants independent of $k$.

\item \emph{Normalization.} Set
\begin{equation*}
c_k:=E[u_k](0)^{-1/2},
\qquad
(u_k,S_k)\longmapsto(c_ku_k,c_kS_k),
\textrm{ \; and \; }
  r_k:=c_k\widetilde r_k.
\end{equation*}
In \eqref{eq:reconstruct}, this means replacing the branch amplitudes by
\begin{equation*}
A_{\alpha,\varepsilon_k}\longmapsto c_kA_{\alpha,\varepsilon_k} \textrm{ \; and \; }
 \calT_{\varepsilon_k}\longmapsto c_k\calT_{\varepsilon_k}.
\end{equation*}
We continue to write $u_k$, $S_k$, and $\calT_{\varepsilon_k}$ for the normalized quantities. With this convention, asymptotic realization reads $E[u_k](0)=1$ and
\begin{equation}\label{eq:aos-real}
\partial_tu_k=S_k+r_k \textrm{ \; \; with \; }
\|r_k\|_{L^2(Q_T)}\longrightarrow0.
\end{equation}
\end{itemize}

\Needspace{22\baselineskip}
\subsection{Asymptotic observable symmetry}\label{subsec:aos-def}

We now collect the requirements on profile, symmetry, scale and asymptotics in the following definition.

\begin{defi}\label{def:aos}
Consider the operators $\{\calT_\varepsilon\}_{\varepsilon}$ of Definition~\ref{def:model}. Fix a profile $q\in X$ and write its branch profiles as $f_\alpha=L_\alpha q$. Let $(w_\alpha,\calM_q)$ be fixed symmetry data satisfying the nontrivial balance of Section~\ref{subsec:aos-symmetry}, in particular \eqref{eq:ref-balance}.

We say that a fixed measurable observation region $G\subset Q_T$ admits AOS if, along a sequence of scales $\varepsilon_k\to0^+$, the generated approximate velocities
\begin{equation*}
S_k:=\calT_{\varepsilon_k}q
=\Real\sum_{\alpha=1}^J A_{\alpha,\varepsilon_k}f_\alpha(\Theta_{\alpha,\varepsilon_k})
\end{equation*}
satisfy the following three properties.
\begin{enumerate}[label=\textup{(\roman*)},leftmargin=9mm]
\item \emph{The same symmetry at every scale.} The phases $\Theta_{\alpha,\varepsilon_k}$ transport the fixed branch profiles and symmetry data as in Section~\ref{subsec:aos-symmetry}. They give matching regions $\Gamma_k$ on which $S_k=0$.
\item \emph{One fixed observation region.} There are packet regions $\mathcal U_k$ supporting $S_k$ and additional observation regions $\calE_k$ such that
\begin{equation*}
G\cap\mathcal U_k\subseteq\Gamma_k\cup\calE_k \textrm{ \; and \; }
  \|\1_{\calE_k}S_k\|_{L^2(Q_T)}\longrightarrow0.
\end{equation*}
In our constructions, $G^{(N)}\to G$, and the same inclusion holds with $G^{(N)}$ in place of $G$ for $k\le N$, as in \eqref{eq:aos-stages-add}.
\item \emph{Asymptotic realization by exact waves.} There are solutions $u_k$ of \eqref{eq:wave} and remainders $r_k$ satisfying
\begin{equation*}
\left\{\begin{aligned}
& \partial_tu_k=S_k+r_k, \;\;\;\;\; \; E[u_k](0)=1, \\
&\|\1_GS_k\|_{L^2(Q_T)}\longrightarrow0,\\ 
&\|r_k\|_{L^2(Q_T)}\longrightarrow0.
\end{aligned}\right.
\end{equation*}
\end{enumerate}
\end{defi}

\subsection{Non-observability}\label{subsec:aos-nonobs}

AOS obstructs observability as follows.

\begin{prop}\label{prop:aos-obs}
If $G$ admits AOS in the sense of Definition~\ref{def:aos}, then \eqref{eq:obs} fails. More precisely, define
\begin{equation}\label{eq:zero-part}
Z_k:=\1_{\calE_k^c}S_k \textrm{ \; and \; }
  e_k:=r_k+\1_{\calE_k}S_k.
\end{equation}
Then $\partial_tu_k=Z_k+e_k$, $\1_GZ_k=0$, and $\|e_k\|_{L^2}\to0$.
\end{prop}

\begin{proof}
The transported balance gives $\1_GZ_k=0$. Hence
\begin{equation*}
\|\1_G\partial_tu_k\|_{L^2}
\le\|e_k\|_{L^2}
\le\|r_k\|_{L^2}+\|\1_{\calE_k}S_k\|_{L^2}\longrightarrow0.
\end{equation*}
Since $E[u_k](0)=1$, no constant can satisfy \eqref{eq:obs}.
\end{proof}

\section{An explicit example of AOS on the disc}\label{sec:disc}

We now construct the disc example in Theorem~\ref{thm:disc}. The aim is to recover the profile matching of Section~\ref{sec:aos} while satisfying the circular boundary condition. The construction proceeds as follows.
\begin{itemize}[leftmargin=*,itemsep=3pt]
\item Section~\ref{subsec:disc-polar} prepares the exact Dirichlet modes from the wave equation in polar coordinates.
\item Section~\ref{subsec:disc-modes} coordinates their radial indices and angular amplitudes, obtaining propagated profiles with controlled approximation errors.
\item Section~\ref{subsec:disc-symmetry-geometry} determines the branch profiles. Proposition~\ref{prop:disc-support} gives sector support and profile matching for $S_k$; Proposition~\ref{prop:disc-observation} assembles the scales into $G^{(N)}\downarrow G$.
\item Proposition~\ref{prop:disc-real} constructs exact waves $u_k$ and proves the asymptotic realization formula. Theorem~\ref{thm:disc-gcc} then chooses the same parameters to obtain \hyperref[eq:gcc]{GCC} for $G$.
\end{itemize}

\setcounter{subsection}{-1}
\subsection{Obstructions to the earlier construction}\label{subsec:disc-obstruction}

Let $\xi\in\mathbb S^1$ and consider a wave in $\R^2$ of the form
\begin{equation*}
u(t,x)=F_+(t-x\cdot\xi)-F_-(t+x\cdot\xi).
\end{equation*}
Writing $f=F_+'$ and $g=F_-'$, we obtain
\begin{equation*}
\partial_tu(t,x)=f(t-x\cdot\xi)-g(t+x\cdot\xi).
\end{equation*}
Unlike a straight reflecting boundary, the circle has a normal that varies with the contact point. Restricting the preceding wave to it gives boundary values
\begin{equation*}
F_+\bigl(t-\cos(\theta-\theta_\xi)\bigr)-F_-\bigl(t+\cos(\theta-\theta_\xi)\bigr),
\qquad \xi=(\cos\theta_\xi,\sin\theta_\xi),
\end{equation*}
which generally fail the Dirichlet boundary condition. Rotation symmetry suggests separating the angular variable and allowing the radial equation to select oscillations compatible with the boundary.

\subsection{Dirichlet waves in polar coordinates}\label{subsec:disc-polar}

We collect the standard spectral facts for the disc used in the construction; see \cite{Watson}. For the Fourier--Bessel description of disc eigenfunctions in the context of Schr{\"o}dinger observability, see \cite[Section~1.3]{ALM2016}. Write $x=(r\cos\theta,r\sin\theta)$. In polar coordinates, the wave equation \eqref{eq:wave} becomes
\begin{equation}\label{eq:disc-polar}
\begin{cases}
u_{tt}-u_{rr}-\dfrac1r u_r-\dfrac1{r^2}u_{\theta\theta}=0,&\;\;\forall r\in(0,1),\\
u(t,1,\theta)=0,
\end{cases}
\end{equation}
where $t\in\R$ and $\theta\in\R/2\pi\Z$. Since the coefficients are independent of $\theta$, we expand a smooth wave as
\begin{equation}\label{eq:disc-angular-exp}
u(t,r,\theta)=\sum_{m\in\Z}u_m(t,r)e^{im\theta}.
\end{equation}

Each angular Fourier component satisfies
\begin{equation}\label{eq:disc-radial-wave}
\partial_t^2u_m-\partial_r^2u_m-\frac1r\partial_ru_m+\frac{m^2}{r^2}u_m=0
\quad\textrm{with}\quad u_m(t,1)=0.
\end{equation}
To find elementary oscillations, take a single wave frequency. The resulting real modes have the form
\begin{equation*}
u(t,r,\theta)=\Real[e^{-i\omega t}e^{im\theta}R(r)],  \textrm{ \; for some \; }
\omega>0,  \;\;  m\in\Z,
\end{equation*}
where $R$ solves
\begin{equation*}
R''+\frac1rR'+\left(\omega^2-\frac{m^2}{r^2}\right)R=0 \textrm{ \; with \; }
 R(1)=0.
\end{equation*}
We require $R(r)e^{im\theta}$ to extend smoothly across the origin. Among solutions of this radial equation, this selects the regular solution: $R(r)=O(r^{|m|})$ for $m\ne0$, while $R'(0)=0$ for $m=0$.
\vspace{2mm}

With $z=\omega r$ and $y(z)=R(z/\omega)$, the radial equation becomes Bessel's equation, with derivatives taken in the $z$-variable:
\begin{equation}\label{eq:bessel-ode}
z^2y''+zy'+(z^2-m^2)y=0.
\end{equation}
\begin{defi}\label{def:disc-bessel}
The \emph{Bessel function of the first kind of order $|m|$}, denoted by $J_{|m|}$, is the regular solution of \eqref{eq:bessel-ode} normalized by
\begin{equation*}
J_{|m|}(z)=\frac{z^{|m|}}{2^{|m|} |m|!}\bigl(1+O(z^2)\bigr)
\qquad\text{as }z\to0^+.
\end{equation*}
In particular, $J_0(0)=1$, whereas $J_{|m|}$ vanishes to order $|m|$ at the origin when $|m|\ge1$. Denote its positive zeros, in increasing order, by
\begin{equation*}
0<j_{|m|,1}<j_{|m|,2}<\cdots
\qquad\text{with}\qquad j_{|m|,n}\to\infty.
\end{equation*}
\end{defi}

\begin{lemma}\label{lem:bessel-radial}
Let $m\in\Z$ and $\omega>0$. Every solution of
\begin{equation}\label{eq:disc-radial-ode}
R''+\frac1rR'+\left(\omega^2-\frac{m^2}{r^2}\right)R=0
\end{equation}
that remains bounded at the origin has the form $R(r)=cJ_{|m|}(\omega r)$. A nonzero such solution satisfies $R(1)=0$ if and only if $\omega=j_{|m|,n}$ for some $n\ge1$. These positive zeros are simple, and
\begin{equation}\label{eq:bessel-norm}
\int_0^1J_{|m|}(j_{|m|,n}r)^2r\dd r
=\frac12J_{|m|+1}(j_{|m|,n})^2>0.
\end{equation}
\end{lemma}

\begin{proof}
The change of variable above reduces the equation to Bessel's equation. Its second independent solution is unbounded at zero, logarithmically for $|m|=0$ and as $z^{-|m|}$ for $|m|\ge1$. 
Thus boundedness selects $J_{|m|}$, and the boundary condition selects its positive zeros.

If $J_{|m|}$ and $J_{|m|}'$ both vanished at a positive point, uniqueness for the ordinary differential equation would give $J_{|m|}\equiv0$. Hence every positive zero is simple.
For the normalization, Bessel's equation gives
\begin{equation*}
\frac{\dd}{\dd z}
\left[\frac12\bigl(z^2J_{|m|}'(z)^2+(z^2-|m|^2)J_{|m|}(z)^2\bigr)\right]
=zJ_{|m|}(z)^2.
\end{equation*}
Integrating from $0$ to $j=j_{|m|,n}$, we use the local expansion to see that the contribution at zero vanishes. Since $J_{|m|}(j)=0$ and $J_{|m|}'(j)=-J_{|m|+1}(j)$, division by $j^2$ gives \eqref{eq:bessel-norm}.
\end{proof}

The Bessel zeros are therefore the frequencies allowed by the circular boundary. Define the normalized radial functions solving \eqref{eq:disc-radial-ode}, with $\omega=\omega_{m,n}$, and their wave frequencies by
\begin{equation}\label{eq:disc-radial-norm}
R_{|m|,n}(r):=\frac{\sqrt2J_{|m|}(j_{|m|,n}r)}{|J_{|m|+1}(j_{|m|,n})|}
\qquad\text{and}\qquad \omega_{m,n}:=j_{|m|,n}.
\end{equation}
These functions are real, regular at the origin, vanish at $r=1$, and satisfy
\begin{equation*}
\int_0^1|R_{|m|,n}(r)|^2r\dd r=1.
\end{equation*}
Combining them with the angular Fourier modes gives the following spectral basis.

\begin{lemma}[Dirichlet eigenfunctions in the disc]\label{lem:disc-basis}
Define
\begin{equation}\label{eq:bessel-basis}
\Phi_{m,n}(r,\theta):=\frac1{\sqrt{2\pi}}R_{|m|,n}(r)e^{im\theta}
=\frac{J_{|m|}(j_{|m|,n}r)}{\sqrt\pi\,|J_{|m|+1}(j_{|m|,n})|}e^{im\theta},
\quad \forall(m,n)\in\Z\times\Nat.
\end{equation}
These functions form an orthonormal basis of $L^2(\D;\mathbb C)$ and satisfy
\begin{equation*}
-\Delta\Phi_{m,n}=j_{|m|,n}^2\Phi_{m,n}\quad\textrm{with}\quad
\Phi_{m,n}|_{\partial\D}=0.
\end{equation*}
For a fixed $n$, the modes $m$ and $-m$ have the same wave frequency $\omega_{m,n}=\omega_{-m,n}$ and radial factor $R_{|m|,n}$, and $\Phi_{-m,n}=\overline{\Phi_{m,n}}$.
\end{lemma}

\begin{proof}
The eigenvalue equation follows from separation of variables. Angular Fourier orthogonality, radial Sturm--Liouville orthogonality and \eqref{eq:bessel-norm} give orthonormality.
Completeness follows by Fourier expansion in $\theta$ and the completeness of the regular radial Dirichlet eigenfunctions in $L^2((0,1),r\dd r)$; this is the Fourier--Bessel expansion \cite{Watson}. The final assertion follows because the radial equation depends on $m^2$.
\end{proof}

Consequently, every finite-energy Dirichlet wave has the expansion
\begin{equation}\label{eq:disc-wave-exp}
\begin{aligned}
u(t,r,\theta)
&=\sum_{m\in\Z}\sum_{n\ge1}
\left[a_{m,n}\cos(\omega_{m,n}t)+b_{m,n}\frac{\sin(\omega_{m,n}t)}{\omega_{m,n}}\right]\Phi_{m,n}(r,\theta)\\
&=\frac12\sum_{m\in\Z}\sum_{n\ge1}\left[
\left(a_{m,n}+\frac{ib_{m,n}}{\omega_{m,n}}\right)e^{-i\omega_{m,n}t}
+\left(a_{m,n}-\frac{ib_{m,n}}{\omega_{m,n}}\right)e^{i\omega_{m,n}t}
\right]\Phi_{m,n}(r,\theta),
\end{aligned}
\end{equation}
where $a_{m,n}$ and $b_{m,n}$ are the coefficients of $u(0)$ and $\partial_tu(0)$ in this basis. Its time derivative is
\begin{equation}\label{eq:disc-vel-exp}
\begin{aligned}
\partial_tu(t,r,\theta)
=\frac12\sum_{m\in\Z}\sum_{n\ge1}\Bigl[
&(b_{m,n}-i\omega_{m,n}a_{m,n})e^{-i\omega_{m,n}t}\\
\;\;\;\;\;\;\;\;\;\;\;    {}+{}&(b_{m,n}+i\omega_{m,n}a_{m,n})e^{i\omega_{m,n}t}
\Bigr]\Phi_{m,n}(r,\theta).
\end{aligned}
\end{equation}
The expansions converge in $H_0^1(\D)$ and $L^2(\D)$, respectively, uniformly on bounded time intervals. For real initial data, one has $a_{-m,n}=\overline{a_{m,n}}$ and $b_{-m,n}=\overline{b_{m,n}}$.
We also use the classical fixed-order asymptotics for the Bessel functions and their positive zeros; see \cite{Watson}.

\begin{lemma}\label{lem:bessel-asym}
For every fixed $m\in\Z$,
\begin{align}
J_{|m|}(z)&=\left(\frac2{\pi z}\right)^{1/2}
\left[\cos\left(z-\frac{|m|\pi}2-\frac\pi4\right)+O_{|m|}(z^{-1})\right],
&&z\to\infty,\label{eq:bessel-asym}\\
j_{|m|,n}&=\pi\left(n+\frac{|m|}2-\frac14\right)+O_{|m|}(n^{-1}),
&&n\to\infty.\label{eq:bessel-zeros}
\end{align}
Fix $\delta\in(0,1)$. With $z=\omega r$, the error inside the brackets in \eqref{eq:bessel-asym} is $O_{|m|,\delta}(\omega^{-1})$, uniformly for $r\in[\delta,1]$ as $\omega\to\infty$.
\end{lemma}

\subsection{From Dirichlet modes to profiles and approximate velocities}\label{subsec:disc-modes}

Throughout Sections~\ref{sec:disc}--\ref{sec:disc-class}, we use the profile space
\begin{equation*}
X:=\{q\in L^2(\R/2\pi\Z;\R):q(s+\pi)=-q(s)\ \text{a.e.}\}.
\end{equation*}
Its elements have only odd Fourier frequencies. Fix a nonzero $q\in X$.\footnote{We fix Borel representatives when evaluating profiles pointwise.} With the convention $\widehat q_p=(2\pi)^{-1}\int_0^{2\pi}q(s)e^{ips}\dd s$, its Fourier expansion is
\begin{equation*}
q(s)=2\Real\sum_{\substack{p>0\\p\ \mathrm{odd}}}\widehat q_p e^{-ips}
\qquad\text{in }L^2(\R/2\pi\Z).
\end{equation*}
Reality gives $\widehat q_{-p}=\overline{\widehat q_p}$, and antiperiodicity implies zero mean.

This subsection prepares the parameters and estimates used in Sections~\ref{subsec:disc-symmetry-geometry}--\ref{subsec:disc-real}. On a first reading, one may read Section~\ref{subsubsec:disc-mode-plan} and then proceed to Section~\ref{subsec:disc-symmetry-geometry}.

\setcounter{subsubsection}{-1}
\subsubsection{From the profile formula to Dirichlet modes}\label{subsubsec:disc-mode-plan}

We seek solutions of \eqref{eq:wave} whose leading time derivatives have the profile form described in \eqref{eq:intro-disc-prof}. The modes $\Phi_{m,n}$ of Section~\ref{subsec:disc-polar} already satisfy the boundary condition. The question is how to combine them to recover both the prescribed profile and the concentration in sectors.

\vspace{2mm}
\noindent{\it The target expression.} For $q\in X$, we preview the full expression to be constructed in Section~\ref{subsubsec:disc-generated}:
\begin{equation}\label{eq:disc-target}
\begin{split}
S_k[q](t,r,\theta)={}&A_{\varepsilon_k}(r,\theta)\1_{\calA_k}(\theta)\Bigl[q\bigl(K_k(t-r)\bigr)-q\bigl(K_k(t+r)+\pi/2\bigr)\Bigr]\\
&{}+A_{\varepsilon_k}(r,\theta)\1_{\calB_k}(\theta)\Bigl[\calH q\bigl(K_k(t-r)+\pi/2\bigr)+\calH q\bigl(K_k(t+r)\bigr)\Bigr].
\end{split}
\end{equation}
This reduces to \eqref{eq:intro-disc-prof} in the first sector $\calA_k$. Here $\calA_k,\calB_k$ are angular intervals centered at $\psi_k,\psi_k+\pi$, with half-widths $a_k,b_k$, and $\varepsilon_k=K_k^{-1}$. The operator $\calH$ is the periodic Hilbert transform.\footnote{The time-independent amplitude $A_{\varepsilon_k}(r,\theta)$ is consistent with the transport equation in Appendix~\ref{subsec:supp-phase}: its radial factor $r^{-1/2}$ satisfies the transport equations for both phases $t-r$ and $t+r$ on $r>0$.}

Section~\ref{subsubsec:disc-branches} derives the opposite packet and the shifts in these formulas. The two indicators make $S_k[q]$ exactly zero outside the sectors.

\vspace{2mm}
\noindent{\it Radial indices and profile frequencies.}\footnote{Notation on different type of frequencies: we use $p$ for the profile Fourier frequency in $e^{-ips}$, $m$ for the angular index in $e^{im\theta}$, and $n$ for the radial index numbering the positive zeros of $J_{|m|}$. The wave frequency is $\omega=j_{|m|,n}$.}
The branch in $t-r$ is \emph{outgoing}, namely propagating radially toward the boundary; the branch in $t+r$ is incoming, toward the center. Suppressing the subscript $k$, the outgoing term in \eqref{eq:disc-target} has the Fourier expansion
\begin{equation}\label{eq:disc-target-fourier}
A(r,\theta)q\bigl(K(t-r)\bigr)
=2\Real\sum_{\substack{p>0\\p\ \mathrm{odd}}}
\widehat q_p\bigl[A(r,\theta)e^{ipKr}\bigr]e^{-ipKt}.
\end{equation}
Since $A$ is real and independent of time, the required temporal frequency is still $pK$. Comparing with the factor $e^{-ij_{|m|,n}t}$ in a Dirichlet mode, we therefore seek radial indices $n_{p,|m|}$ such that
\begin{equation}\label{eq:disc-freq-match}
j_{|m|,n_{p,|m|}}-pK\longrightarrow0.
\end{equation}
This condition should hold for every participating angular index $m$, so the radial index depends on both $p$ and $|m|$.

\vspace{2mm}
\noindent{\it Angular concentration and the cutoff $Q$.} We now recover the spatial amplitude in \eqref{eq:disc-target} from the selected modes. The integer condition on the radial indices pairs $p\equiv1\pmod4$ with even $m$ and $p\equiv3\pmod4$ with odd $m$. Thus the angular sums with $|m|\le2Q$ initially give two real amplitudes $E_Q,O_Q$.

Their mass concentrates near two opposite directions. In the angular $L^2$ norm, $O_Q$ becomes close to $E_Q$ in the first sector and to $-E_Q$ in the opposite sector. They give the common real amplitude
\begin{equation*}
A_{\varepsilon_k}(r,\theta)
=\frac{C_k}{2\sqrt\pi}
\underbrace{r^{-1/2}}_{\text{Bessel amplitude}}
\underbrace{E_{Q_k}(\theta-\psi_k)}_{\text{angular sum}},
\end{equation*}
where $C_k>0$ will normalize the wave energy. Before these replacements, the spectral sum gives two naturally concentrated packets without sector cutoffs; Lemma~\ref{lem:disc-leading} identifies their expression. 
Choosing $Q_k\min\{a_k,b_k\}\to\infty$ makes the angular tails small. Thus the replacements and cutoffs turn angular concentration into the exact sector support of $S_k[q]$.

\vspace{2mm}
\noindent{\it Profile truncation and the cutoff $P$.} To obtain exact wave solutions, we also use a finite sum of profile frequencies:
\begin{equation*}
q^{(P)}(s):=2\Real\sum_{\substack{0<p\le P\\p\ \mathrm{odd}}}\widehat q_p e^{-ips},\qquad
\|q^{(P)}-q\|_{L^2(0,2\pi)}\longrightarrow0\quad\text{as }P\to\infty.
\end{equation*}
Thus $P$ truncates the profile, whereas $Q$ limits the angular indices. With $P,Q$ fixed, we choose a sufficiently large base radial index $n_0$ to control the frequency and Bessel errors. 
Section~\ref{subsec:disc-real} coordinates these choices to approximate $S_k[q]$ by $\partial_tu_k$.

Although $P_k$ does not enter the formula for $S_k[q]$ directly, it can affect the choices of $K_k$ and $C_k$.

We now prepare these choices and estimates in four steps.
\begin{itemize}[leftmargin=*,itemsep=3pt]
\item Read the $t-r$ and $t+r$ branches from the Bessel asymptotics.
\item Choose $n_{p,|m|}$, determine the allowed angular parity and explain the choice of $X$.
\item Construct and compare the angular amplitudes $E_Q,O_Q$.
\item Control the radial approximation on the whole disc.
\end{itemize}

\begin{defi}\label{def:disc-indices}
Given integers $P,Q\ge1$, let $\mathcal I_{P,Q}$ be the set of pairs $(p,m)\in\mathbb N\times\Z$ with $p$ odd, satisfying
\begin{equation}\label{eq:disc-cutoffs}
0<p\le P \textrm{ \; and \; }  |m|\le2Q,
\end{equation}
and the parity restriction
\begin{equation}\label{eq:disc-parities}
\begin{cases}
m\in2\Z,&p\equiv1\pmod4,\\
m\in2\Z+1,&p\equiv3\pmod4.
\end{cases}
\end{equation}
For each such pair and $n_0\in\mathbb N$, define
\begin{equation}\label{eq:disc-mode-scales}
K=K(n_0):=\pi(n_0-\tfrac14) \textrm{ \; and \; }
n_{p,|m|;n_0}:=pn_0-\frac{|m|}2-\frac{p-1}4.
\end{equation}
We write $n_{p,|m|}$ for $n_{p,|m|;n_0}$ when the base index is understood.
\end{defi}

\begin{prop}\label{prop:disc-modes}
Fix integers $P,Q\ge1$ and use the indices and scale from Definition~\ref{def:disc-indices}. They have the following properties.
\begin{enumerate}[label=(\roman*),leftmargin=*]
\item \textit{Choice of radial indices.} For sufficiently large $n_0$, the indices $n_{p,|m|}$ are positive integers. The corresponding wave frequencies satisfy
\begin{equation*}
\omega=j_{|m|,n_{p,|m|}}=pK+O_{P,Q}(n_0^{-1}).
\end{equation*}
The difference $\omega-pK$ tends to zero for every pair $(p,m)$ satisfying \eqref{eq:disc-cutoffs}--\eqref{eq:disc-parities}, uniformly for fixed $P,Q$ as $n_0\to\infty$.

\item \textit{Real angular amplitudes.} Define
\begin{equation}\label{eq:angular}
\left\{\begin{aligned}
E_Q(\varphi)&:=\frac1{\sqrt{2Q+1}}\sum_{j=-Q}^Qe^{2ij\varphi}
=\frac{\sin\bigl((2Q+1)\varphi\bigr)}{\sqrt{2Q+1}\sin\varphi},\\[4pt]
O_Q(\varphi)&:=\frac1{\sqrt{2Q}}\sum_{j=-Q}^{Q-1}e^{i(2j+1)\varphi}
=\frac{\sin(2Q\varphi)}{\sqrt{2Q}\sin\varphi}.
\end{aligned}\right.
\end{equation}
These sums provide angular localization while recovering a common real amplitude near $\varphi=0$. More precisely,
\begin{gather*}
E_Q(\varphi+\pi)=E_Q(\varphi)\textrm{ \; and \; }
 O_Q(\varphi+\pi)=-O_Q(\varphi),\\
\|E_Q-O_Q\|_{L^2(-w,w)}\to0 \textrm{ \; and \; }
\|E_Q+O_Q\|_{L^2(\pi-w,\pi+w)}\to0
\end{gather*}
when $Q\to\infty$, $w\to0^+$ and $Qw\to\infty$. Moreover, for $0<w<\pi/2$,
\begin{equation}\label{eq:angular-tail}
\int_{\dist(\varphi,\{0,\pi\})>w}(|E_Q|^2+|O_Q|^2)\dd\varphi
\le\frac C{Qw},
\end{equation}
where distance and integration are on $\R/2\pi\Z$, and $C$ is independent of $Q,w$.

\item \textit{Radial approximation.} The normalized radial functions in \eqref{eq:disc-radial-norm} satisfy
\begin{equation*}
\left\|R_{|m|,n_{p,|m|}}-\sqrt2r^{-1/2}\cos(pKr-|m|\pi/2-\pi/4)\right\|_{L^2((0,1),r\dd r)}\to0.
\end{equation*}
For fixed $P,Q$, the convergence is uniform for $(p,m)$ satisfying \eqref{eq:disc-cutoffs}--\eqref{eq:disc-parities} as $n_0\to\infty$.
\end{enumerate}
\end{prop}

\subsubsection{Bessel modes and the propagation variables \texorpdfstring{$t\pm r$}{t +/- r}}\label{subsubsec:disc-propagation}

We first recover the phases $t\pm r$ and the radial amplitude in \eqref{eq:disc-target} from the Dirichlet modes. Multiplying \eqref{eq:disc-radial-wave} by $\sqrt r$ removes the first radial derivative:
\begin{equation*}
\left(\partial_t^2-\partial_r^2+\frac{m^2-\tfrac14}{r^2}\right)(\sqrt r\,u_m)=0.
\end{equation*}
By contrast, for any smooth functions $F_+,F_-$, the two translated branches satisfy
\begin{equation*}
(\partial_t^2-\partial_r^2)\bigl[F_+(t-r)+F_-(t+r)\bigr]=0.
\end{equation*}
The potential $(m^2-\tfrac14)/r^2$ is present even when $m=0$, so free translations do not directly give exact radial waves. We seek this propagation structure in the leading part of the Bessel modes.
\vspace{2mm}

Fix $m\in\Z$ and $\delta\in(0,1)$. For $r\in[\delta,1]$, the potential satisfies
\begin{equation*}
\frac{|m^2-\tfrac14|}{\omega^2r^2}
\le\frac{|m^2-\tfrac14|}{\omega^2\delta^2}\longrightarrow0
\qquad\text{as }\omega\to\infty.
\end{equation*}
This comparison applies away from the origin; the remaining region $r\in(0,\delta)$ requires a separate estimate. By Lemma~\ref{lem:bessel-asym}, the leading Bessel oscillation has two propagation branches. Indeed, set $\delta_{|m|}=|m|\pi/2+\pi/4$ and include the time factor:
\begin{equation}\label{eq:disc-bessel-branches}
e^{-i\omega t}\cos(\omega r-\delta_{|m|})
=\tfrac12e^{-i\delta_{|m|}}e^{-i\omega(t-r)}
+\tfrac12e^{i\delta_{|m|}}e^{-i\omega(t+r)}.
\end{equation}
The first term propagates outward along $t-r=\mathrm{constant}$; the second propagates inward along $t+r=\mathrm{constant}$. Both carry the same real radial amplitude proportional to $r^{-1/2}$.

\subsubsection{Choosing the profile \texorpdfstring{$q$}{q} and the radial indices \texorpdfstring{$n_{p,|m|}$}{n(p,|m|)}}\label{subsubsec:disc-frequencies}

Recall from \eqref{eq:disc-target-fourier} that recovering the outgoing term requires the frequency matching \eqref{eq:disc-freq-match}. We now derive the radial indices in Proposition~\ref{prop:disc-modes}(i) and the associated restrictions on the profile frequencies and angular indices.

\begin{proof}[Proof of Proposition~\ref{prop:disc-modes}(i)]
Let $n_0\ge1$. To choose the scaling parameter $K$, we use the radial index $n_0$ at the angular index $m=0$. Equation \eqref{eq:bessel-zeros} gives
\begin{equation*}
j_{0,n_0}=\pi\left(n_0-\frac14\right)+o(1).
\end{equation*}
Thus $K(n_0)$ in \eqref{eq:disc-mode-scales} is the leading term of $j_{0,n_0}$.

For another fixed angular index $m$, the leading term in \eqref{eq:bessel-zeros} is $\pi(n+|m|/2-1/4)$. To align it with $pK$, we require
\begin{equation*}
n_{p,|m|}+\frac{|m|}2-\frac14=p\left(n_0-\frac14\right).
\end{equation*}
Note that the indices $n_{p,|m|}$ and $|m|$ must be integers. Multiplying the last relation by four gives
\begin{equation*}
4n_{p,|m|}+2|m|-1=p(4n_0-1).
\end{equation*}
The left-hand side and $4n_0-1$ are odd, so $p$ must be odd. Solving for $n_{p,|m|}$ gives
\begin{equation}\label{eq:radial-index}
n_{p,|m|}=pn_0-\frac{|m|}2-\frac{p-1}4.
\end{equation}

The remaining integer condition determines the angular parity. For $p=1$ and $p=3$, the formula reads
\begin{equation*}
n_{1,|m|}=n_0-\frac{|m|}2 \textrm{ \; and \; }
n_{3,|m|}=3n_0-\frac{|m|+1}{2}.
\end{equation*}
Thus the integer condition gives precisely the parity rule \eqref{eq:disc-parities}. This explains the choice of $X$ in this spectral construction: the allowed profile frequencies are odd, equivalently $q(s+\pi)=-q(s)$ in $L^2$.

\vspace{2mm}
We now fix $P,Q$ and let $n_0\to\infty$ for the pairs satisfying \eqref{eq:disc-cutoffs}--\eqref{eq:disc-parities}. The indices $n_{p,|m|}$ are then positive for large $n_0$, and
\begin{equation}\label{eq:action}
n_{p,|m|}-\frac14+\frac{|m|}2=p\left(n_0-\frac14\right).
\end{equation}
Combining \eqref{eq:action} with \eqref{eq:bessel-zeros} gives
\begin{equation*}
j_{|m|,n_{p,|m|}}=pK+O_{P,Q}(n_0^{-1}),
\end{equation*}
where the implicit constant depends on $P,Q$ but is independent of $n_0$. This proves part (i).
\end{proof}

\begin{remark}\label{rem:disc-frequencies}
We call \eqref{eq:action} the \emph{common radial action}. It aligns the wave frequencies with $pK$ to leading order and assigns each profile Fourier mode to an angular parity through \eqref{eq:disc-parities}. The wave frequencies are close to $pK$; on a fixed time interval,
\begin{equation*}
\left|e^{-ij_{|m|,n_{p,|m|}}t}-e^{-ipKt}\right|
\le T|j_{|m|,n_{p,|m|}}-pK|=o(1),
\quad \forall t\in[0,T].
\end{equation*}
\end{remark}

\subsubsection{Constructing real angular amplitudes \texorpdfstring{$E_Q,O_Q$}{E_Q,O_Q}}\label{subsubsec:disc-amplitudes}

We now choose coefficients in the spectral expansion \eqref{eq:disc-wave-exp} to recover the angular part of the amplitude in \eqref{eq:disc-target}. Recall the two branches of each mode in \eqref{eq:disc-bessel-branches}. For each $p$, let $\alpha_{m,p}$ be real angular weights supported on the permitted indices. To recover the outgoing profile and center its angular concentration at $\psi$, choose the velocity coefficient
\begin{equation}\label{eq:disc-coeff}
c_{p,m}:=
\widehat q_p\,
\underbrace{\alpha_{m,p}}_{\text{angular weight}}\,
\underbrace{e^{-im\psi}}_{\text{angular translation}}\,
\underbrace{e^{i\delta_{|m|}}}_{\text{phase compensation}},
\end{equation}
where $\delta_{|m|}=|m|\pi/2+\pi/4$. Fix a permitted pair $(p,m)$ and set $j:=j_{|m|,n_{p,|m|}}$. By the normalization \eqref{eq:disc-radial-norm} and Lemma~\ref{lem:bessel-asym}, for each fixed $\delta\in(0,1)$,
\begin{equation*}
\Phi_{m,n_{p,|m|}}(r,\theta)
=\frac{e^{im\theta}}{\sqrt{\pi r}}\cos(jr-\delta_{|m|})+O_{|m|,\delta}(j^{-1}),
\quad \forall r\in[\delta,1].
\end{equation*}
Multiplying by the velocity coefficient and the time factor, we obtain
\begin{equation*}
c_{p,m}e^{-ijt}\Phi_{m,n_{p,|m|}}(r,\theta)
=\frac{\widehat q_p\alpha_{m,p}e^{im(\theta-\psi)}}{\sqrt{\pi r}}
e^{i\delta_{|m|}}e^{-ijt}\cos(jr-\delta_{|m|})
+O_{q,p,m,\delta}(j^{-1}),
\end{equation*}
with the angular weights fixed. The two oscillations satisfy
\begin{align*}
e^{i\delta_{|m|}}e^{-ijt}\cos(jr-\delta_{|m|})
&=\frac12e^{i\delta_{|m|}}
\left[e^{-i\delta_{|m|}}e^{-ij(t-r)}+e^{i\delta_{|m|}}e^{-ij(t+r)}\right]\\
&=\frac12\left[e^{-ij(t-r)}+i(-1)^{|m|}e^{-ij(t+r)}\right],
\end{align*}
since $e^{2i\delta_{|m|}}=i(-1)^{|m|}$. Thus the leading term of the full velocity mode is
\begin{equation}\label{eq:disc-mode-leading}
\frac{\widehat q_p\alpha_{m,p}e^{im(\theta-\psi)}}{2\sqrt{\pi r}}
\left[e^{-ij_{|m|,n_{p,|m|}}(t-r)}+i(-1)^{|m|}e^{-ij_{|m|,n_{p,|m|}}(t+r)}\right].
\end{equation}
The chosen coefficient cancels the outgoing phase in \eqref{eq:disc-bessel-branches}. The remaining incoming factor determines the reflected profile in Section~\ref{subsubsec:disc-branches}.

Using \eqref{eq:disc-freq-match}, we replace $j_{|m|,n_{p,|m|}}$ by $pK$ in the leading oscillations on bounded time intervals. Apart from the common radial factor, summing the outgoing terms over $m$ gives
\begin{equation*}
\widehat q_p e^{-ips}\sum_m\alpha_{m,p}e^{im\varphi},
\end{equation*}
where $s=K(t-r)$ and $\varphi=\theta-\psi$. Thus the angular concentration is determined by the Fourier sum. For this example, we use the same angular coefficients for every $p$ within each parity class:
\begin{equation}\label{eq:disc-angular-coefficients}
\alpha_{m,p}:=
\begin{cases}
\alpha_{m,Q}^{e},&p\equiv1\pmod4,\quad m\in2\Z,\\
\alpha_{m,Q}^{o},&p\equiv3\pmod4,\quad m\in2\Z+1,\\
0,&\text{otherwise},
\end{cases}
\end{equation}
with coefficients supported in $|m|\le2Q$. The superscripts $e,o$ refer to even and odd angular indices, respectively; the dependence of $\alpha_{m,p}$ on the fixed cutoff $Q$ is implicit. Their angular sums are
\begin{equation*}
\sum_m\alpha_{m,p}e^{im\varphi}
=\begin{cases}
E_Q(\varphi),&p\equiv1\pmod4,\\
O_Q(\varphi),&p\equiv3\pmod4.
\end{cases}
\end{equation*}
Summing the outgoing contributions now gives
\begin{equation*}
2\Real\left[
E_Q(\varphi)\sum_{\substack{p>0\\p\equiv1\ (4)}}\widehat q_p e^{-ips}
+O_Q(\varphi)\sum_{\substack{p>0\\p\equiv3\ (4)}}\widehat q_p e^{-ips}
\right].
\end{equation*}
If $O_Q(\varphi)=E_Q(\varphi)$, with both factors real, this expression would reduce to
\begin{equation*}
2\Real\left[E_Q(\varphi)\sum_{\substack{p>0\\p\ \mathrm{odd}}}\widehat q_p e^{-ips}\right]=E_Q(\varphi)q(s).
\end{equation*}
We therefore seek real angular sums that concentrate near $\varphi=0$ and become close there in $L^2$.

Within either parity class, equal coefficients on a finite symmetric set of angular indices reinforce one another at $\varphi=0$. This gives the choice of $E_Q,O_Q$ in \eqref{eq:angular}.
More precisely, we take
\begin{equation}\label{eq:disc-equal-coeff}
\left\{\begin{aligned}
\alpha_{m,Q}^{e}&=(2Q+1)^{-1/2},&&m=-2Q,-2Q+2,\ldots,2Q,\\
\alpha_{m,Q}^{o}&=(2Q)^{-1/2},&&m=\pm1,\pm3,\ldots,\pm(2Q-1),
\end{aligned}\right.
\end{equation}
with all other coefficients zero.
Both sums are real because each angular index is paired with its negative. The even and odd indices also have different half-turn relations:
\begin{equation}\label{eq:halfturn}
E_Q(\varphi+\pi)=E_Q(\varphi)\textrm{ \; and \; }
O_Q(\varphi+\pi)=-O_Q(\varphi).
\end{equation}
Thus the first and opposite packets concentrate near $\theta=\psi$ and $\theta=\psi+\pi$. Each contains an outgoing and an incoming branch, computed in Section~\ref{subsubsec:disc-branches}.
Adding and subtracting the angular sums separates these two concentrations. Lemma~\ref{lem:disc-leading} will combine them with the corresponding packet expressions $V_+,V_-$.

\begin{lemma}\label{lem:angular}
Let $w=w(Q)>0$ satisfy $w\to0^+$ and $Qw\to\infty$ as $Q\to\infty$. Then
\begin{equation}\label{eq:angular-local}
\|E_Q-O_Q\|_{L^2(-w,w)}\to0,
\qquad \|E_Q+O_Q\|_{L^2(\pi-w,\pi+w)}\to0.
\end{equation}
The first convergence also holds on any shrinking larger neighborhood of $0$. Moreover, the tail estimate \eqref{eq:angular-tail} holds.
\end{lemma}

\begin{proof}
Using \eqref{eq:angular}, we separate the difference of the normalizing factors:
\begin{align*}
E_Q(\varphi)-O_Q(\varphi)
&=\left(1-\sqrt{\frac{2Q+1}{2Q}}\right)E_Q(\varphi)
+\frac{\sin\bigl((2Q+1)\varphi\bigr)-\sin(2Q\varphi)}{\sqrt{2Q}\sin\varphi}\\
&=\left(1-\sqrt{\frac{2Q+1}{2Q}}\right)E_Q(\varphi)
+\frac{\cos\bigl((4Q+1)\varphi/2\bigr)}{\sqrt{2Q}\cos(\varphi/2)}.
\end{align*}
Since $\|E_Q\|_{L^2(0,2\pi)}^2=2\pi$, the first term has $L^2$ norm at most $C/Q$. For $0<w<\pi/2$, the second term is bounded by $C/\sqrt Q$ on $(-w,w)$. Thus
\begin{equation*}
\|E_Q-O_Q\|_{L^2(-w,w)}^2
\le C\left(\frac1{Q^2}+\frac wQ\right)\longrightarrow0.
\end{equation*}
This also proves convergence on any shrinking larger neighborhood of $0$. The comparison near $\pi$ follows from \eqref{eq:halfturn}.

For the tails, \eqref{eq:angular} gives $|E_Q|+|O_Q|\le C/(\sqrt Q|\sin\varphi|)$. Hence
\begin{equation*}
\int_{\dist(\varphi,\{0,\pi\})>w}(|E_Q|^2+|O_Q|^2)\dd\varphi
\le\frac C Q\int_w^{\pi/2}\frac{\dd s}{s^2}
\le\frac C{Qw},
\end{equation*}
which proves \eqref{eq:angular-tail}.
\end{proof}

\subsubsection{Controlling the radial approximation on the whole disc}\label{subsubsec:disc-radial}

The leading velocity modes in \eqref{eq:disc-mode-leading} have so far been justified away from the origin. To obtain the full $L^2(\D)$ approximation, we prove Proposition~\ref{prop:disc-modes}(iii). Recall the normalized radial functions \eqref{eq:disc-radial-norm} and split the radial squared norm as
\begin{equation*}
\|h\|_{L^2((0,1),r\dd r)}^2
=\int_0^\delta|h(r)|^2r\dd r+\int_\delta^1|h(r)|^2r\dd r.
\end{equation*}
The second term is controlled by the Bessel expansion; the first will be small with $\delta$.

\begin{lemma}\label{lem:bessel}
For every fixed pair $(p,m)$ satisfying \eqref{eq:disc-parities},
\begin{equation}\label{eq:bessel-limit}
\left\|R_{|m|,n_{p,|m|}}-\sqrt2r^{-1/2}\cos(pKr-|m|\pi/2-\pi/4)\right\|_{L^2((0,1),r\dd r)}\to0
\end{equation}
as $n_0\to\infty$. For fixed $P,Q$, the convergence is uniform over the pairs satisfying \eqref{eq:disc-cutoffs}--\eqref{eq:disc-parities}.
\end{lemma}

\begin{proof}
Fix $\delta\in(0,1)$ and write $j=j_{|m|,n_{p,|m|}}$. Recall that $\delta_{|m|}=|m|\pi/2+\pi/4$. At the zero $j$, Lemma~\ref{lem:bessel-asym} gives
\begin{equation*}
\cos(j-\delta_{|m|})=O_{|m|}(j^{-1}),\qquad
|J_{|m|+1}(j)|=\left(\frac2{\pi j}\right)^{1/2}\bigl(1+O_{|m|}(j^{-1})\bigr).
\end{equation*}
Substituting in the normalized radial function, we obtain
\begin{equation*}
R_{|m|,n_{p,|m|}}(r)
=\sqrt2r^{-1/2}\cos(jr-\delta_{|m|})+O_{|m|,\delta}(j^{-1}),
\qquad \forall r\in[\delta,1].
\end{equation*}
Moreover, \eqref{eq:action} and \eqref{eq:bessel-zeros} imply
\begin{equation*}
\sup_{r\in[\delta,1]}
|\cos(jr-\delta_{|m|})-\cos(pKr-\delta_{|m|})|
\le |j-pK|\longrightarrow0.
\end{equation*}
This proves convergence on the fixed annulus.
\vspace{2mm}

Near the origin, we use the identity
\begin{equation*}
\int_0^\delta J_{|m|}(jr)^2r\dd r
=\frac{\delta^2}2\bigl[J_{|m|}(j\delta)^2-J_{|m|-1}(j\delta)J_{|m|+1}(j\delta)\bigr]
\end{equation*}
with $J_{-1}=-J_1$ when $|m|=0$. For fixed $\delta$, the bracket is $O_{|m|}\bigl((j\delta)^{-1}\bigr)$ as $j\to\infty$. Dividing by the normalization factor therefore gives
\begin{align*}
\limsup_{n_0\to\infty}\int_0^\delta|R_{|m|,n_{p,|m|}}(r)|^2r\dd r
&\le C_{|m|}\delta,\\
\int_0^\delta|\sqrt2r^{-1/2}\cos(pKr-\delta_{|m|})|^2r\dd r
&\le2\delta.
\end{align*}
Combining these bounds with convergence on $[\delta,1]$, we obtain
\begin{equation*}
\limsup_{n_0\to\infty}
\left\|R_{|m|,n_{p,|m|}}-\sqrt2r^{-1/2}\cos(pKr-\delta_{|m|})\right\|_{L^2((0,1),r\dd r)}^2
\le C_{|m|}\delta.
\end{equation*}

Letting $\delta\to0^+$ proves \eqref{eq:bessel-limit}. For fixed $P,Q$, there are only finitely many admissible pairs, so the convergence is uniform over them.
\end{proof}

\begin{proof}[Completion of the proof of Proposition~\ref{prop:disc-modes}]
Part (i) was proved in Section~\ref{subsubsec:disc-frequencies}. The definition \eqref{eq:angular}, the identities \eqref{eq:halfturn}, and Lemma~\ref{lem:angular} give part (ii). Lemma~\ref{lem:bessel} gives part (iii), uniformly over the retained modes for fixed $P,Q$.
\end{proof}

\subsection{Profile symmetry, scale, and observation region}\label{subsec:disc-symmetry-geometry}

Now we proceed as in Section~\ref{sec:aos}. For a fixed $q\in X$, we first compute the four branch profiles carried by the two opposite angular packets and their reference matching region. Lemma~\ref{lem:disc-leading} identifies the leading spectral expression, and Proposition~\ref{prop:disc-support} defines $S_k$ and $\Gamma_k$ for fixed sector pairs.

With disjoint sector pairs, Proposition~\ref{prop:disc-observation} assembles the scales into $G^{(N)}\downarrow G$, with $\1_GS_k=0$ for every $k$. The parameters remain available for exact-wave realization in Section~\ref{subsec:disc-real} and \hyperref[eq:gcc]{GCC} in Section~\ref{sec:disc-class}.

\subsubsection{Determining the branch profiles}\label{subsubsec:disc-branches}

We identify the outgoing and incoming profiles in the first and opposite packets, using the two-branch formula \eqref{eq:disc-mode-leading} and the parity rule \eqref{eq:disc-parities}.

\noindent{\it The first packet.} Recall that the coefficient $c_{p,m}$ in \eqref{eq:disc-coeff} compensates the outgoing Bessel phase. The remaining outgoing and incoming factors in \eqref{eq:disc-mode-leading} are $1$ and $i(-1)^{|m|}$. The parity rule gives
\begin{equation*}
(-1)^{|m|}=ie^{-ip\pi/2},
\qquad i(-1)^{|m|}=-e^{-ip\pi/2}.
\end{equation*}
The outgoing profile is therefore $q\bigl(K(t-r)\bigr)$. For the incoming profile, the extra Fourier factor produces a quarter-period shift:
\begin{equation*}
2\Real\sum_{\substack{p>0\\p\ \mathrm{odd}}}
\widehat q_p\bigl(-e^{-ip\pi/2}\bigr)e^{-ipK(t+r)}
=-q\bigl(K(t+r)+\pi/2\bigr).
\end{equation*}

\noindent{\it The opposite packet.} Writing $\varphi=\theta-\psi$, translation of the angle by $\pi$ gives
\begin{equation*}
e^{im(\varphi+\pi)}=(-1)^m e^{im\varphi}.
\end{equation*}
Thus the outgoing and incoming factors each acquire $(-1)^m$, becoming
\begin{equation*}
(-1)^m=ie^{-ip\pi/2},
\qquad
(-1)^m\,i(-1)^{|m|}=i,
\end{equation*}
respectively. To interpret multiplication by $i$, recall the periodic Hilbert transform $\calH$, with multiplier $-i\sgn n$ on $e^{ins}$. In our Fourier convention,
\begin{equation*}
P_+q(s):=\sum_{\substack{p>0\\p\ \mathrm{odd}}}\widehat q_p e^{-ips},
\qquad q=2\Real P_+q,\qquad \calH q=2\Real(iP_+q).
\end{equation*}
Consequently, the opposite outgoing and incoming profiles are
\begin{align*}
2\Real\sum_{\substack{p>0\\p\ \mathrm{odd}}}
i\widehat q_p e^{-ip\pi/2}e^{-ipK(t-r)}
&=\calH q\bigl(K(t-r)+\pi/2\bigr),\\
2\Real\sum_{\substack{p>0\\p\ \mathrm{odd}}}
i\widehat q_p e^{-ipK(t+r)}
&=\calH q\bigl(K(t+r)\bigr).
\end{align*}

\noindent{\it The branch profiles.} In the reference coordinates $(d_1,d_2)\in(\R/2\pi\Z)^2$, the four profiles above give the two packet expressions
\begin{equation}\label{eq:two-prof}
\begin{split}
V_+(d_1,d_2)&:=q(d_1)-q(d_2+\pi/2),\\
V_-(d_1,d_2)&:=\calH q(d_1+\pi/2)+\calH q(d_2).
\end{split}
\end{equation}
Their spacetime expressions are $V_\pm\bigl(K(t-r),K(t+r)\bigr)$. In the notation of Definition~\ref{def:model}, these four branch profiles are $f_\alpha=L_\alpha q$, with
\begin{equation*}
(L_1,L_2,L_3,L_4)=(\Id,\tau_{\pi/2},\tau_{\pi/2}\calH,\calH),
\qquad \tau_\alpha h(s):=h(s+\alpha).
\end{equation*}
The groups $I=\{1,2\}$ and $I=\{3,4\}$ give the first and opposite packets.

\subsubsection{The reference matching region}\label{subsubsec:disc-reference}

The two branch profiles of the first packet suggest the reference matching region
\begin{equation}\label{eq:disc-ref}
\calM_q:=\{(d_1,d_2)\in(\R/2\pi\Z)^2:q(d_1)=q(d_2+\pi/2)\}.
\end{equation}
In the four-branch notation of Section~\ref{subsec:aos-symmetry}, this balance has weights $(w_1,w_2,w_3,w_4)=(1,-1,0,0)$.
For the profile $q_0$ from \eqref{eq:square}, the two types of matching cells have the values
\begin{equation*}
\bigl(q_0(d_1),q_0(d_2+\pi/2)\bigr)=(1,1)\quad\hbox{or}\quad(-1,-1).
\end{equation*}
Thus the nonzero branch values cancel in $V_+(d_1,d_2)$. This is the reference balance of Section~\ref{subsec:aos-symmetry}, with its shift determined by disc propagation; see Figure~\ref{fig:disc-reference}.

\begin{figure}[htbp]
\centering
\includegraphics[width=.9\textwidth,
  trim={0mm 1mm 0mm 12mm},
  clip]{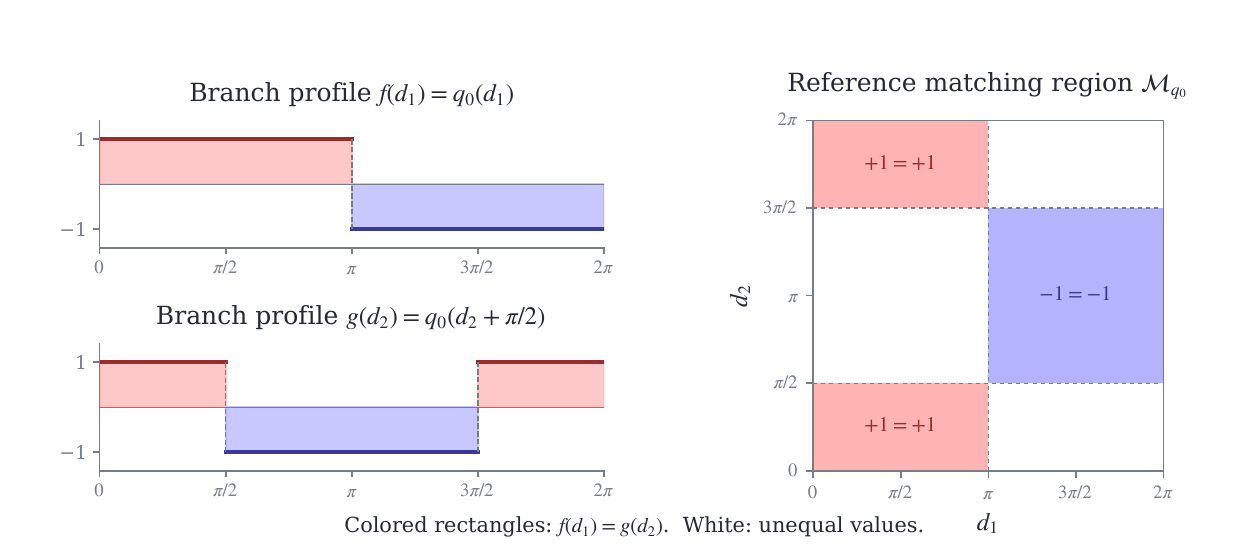}
\caption{Reference matching for $q_0(d_1)=q_0(d_2+\pi/2)$. Red and blue correspond to the common values $+1$ and $-1$, respectively.}\label{fig:disc-reference}
\end{figure}

\subsubsection{Generated approximate velocities, sector support, and matching regions}\label{subsubsec:disc-generated}

The following lemma proves velocity approximation for the leading profile expression before sector cutoffs. Its two packet expressions are those derived from the Bessel phases in \eqref{eq:two-prof}. See also the discussion of multi-branch velocity approximation in Appendix~\ref{subsubsec:supp-assembly}.

\begin{lemma}[Leading profile expression]\label{lem:disc-leading}
Let $q\in X$ be a finite Fourier sum. Let $T>0$. Fix an integer $Q\ge1$, $\psi\in\R/2\pi\Z$ and $C>0$. Set
\begin{equation}\label{eq:disc-packet-amp}
E_Q^+:=\frac{E_Q+O_Q}{2}
\qquad\text{and}\qquad
E_Q^-:=\frac{E_Q-O_Q}{2}.
\end{equation}
Let $n_0\in\Nat$. For $K=K(n_0)$ from \eqref{eq:disc-mode-scales}, define
\begin{equation}\label{eq:disc-leading-packets}
\begin{split}
\widetilde S[q](t,r,\theta):=\frac{C}{2\sqrt{\pi r}}\Bigl[
&E_Q^+(\theta-\psi)V_+\bigl(K(t-r),K(t+r)\bigr)\\
{}+{}&E_Q^-(\theta-\psi)V_-\bigl(K(t-r),K(t+r)\bigr)\Bigr].
\end{split}
\end{equation}
For all sufficiently large $n_0$, there are real solutions $u_{n_0}$ of \eqref{eq:wave}, each given by a finite sum of Dirichlet modes, such that
\begin{equation}\label{eq:disc-leading-real}
\|\partial_tu_{n_0}-\widetilde S[q]\|_{L^2((0,T)\times\D)}\longrightarrow0
\quad\text{ as } n_0\to\infty.
\end{equation}
\end{lemma}

By Lemma~\ref{lem:angular}, $E_Q^+$ and $E_Q^-$ concentrate near $0$ and $\pi$, respectively. We replace each by $E_Q$ in its own sector and by zero elsewhere. Using the full profile $q\in X$, this gives the expression below, with exact sector support and profile matching.

\begin{prop}\label{prop:disc-support}
Let $T>0$ and a nonzero $q\in X$. For each $k\ge1$, let $\psi_k\in\R/2\pi\Z$, $0<a_k,b_k<\pi/4$, $Q_k\in\Nat$, and $K_k,C_k>0$. Define the angular sectors on $\R/2\pi\Z$ by
\begin{equation}\label{eq:sectors}
\calA_k:=\{|\theta-\psi_k|<a_k\},\qquad
\calB_k:=\{|\theta-\psi_k-\pi|<b_k\},
\end{equation}
and the spacetime sectors by
\begin{equation}\label{eq:disc-spacetime-sectors}
\widehat{\calA}_k:=(0,T)\times(0,1)\times\calA_k,\qquad
\widehat{\calB}_k:=(0,T)\times(0,1)\times \calB_k.
\end{equation}
Set $\varepsilon_k=K_k^{-1}$.

With $E_{Q_k}$ from \eqref{eq:angular}, define the common real amplitude
\begin{equation}\label{eq:disc-amp}
A_{\varepsilon_k}(r,\theta):=\frac{C_kE_{Q_k}(\theta-\psi_k)}{2\sqrt{\pi r}},
\end{equation}
and, using $V_\pm$ from \eqref{eq:two-prof}, the generated approximate velocity
\begin{equation}\label{eq:disc-S}
\begin{aligned}
S_k(t,r,\theta)
&:=A_{\varepsilon_k}(r,\theta)\Bigl[
\1_{\calA_k}(\theta)V_+\bigl(K_k(t-r),K_k(t+r)\bigr)\\
&\hphantom{{}:=A_{\varepsilon_k}(r,\theta)\Bigl[}{}+\1_{\calB_k}(\theta)V_-\bigl(K_k(t-r),K_k(t+r)\bigr)\Bigr]\\
&=\frac{C_kE_{Q_k}(\theta-\psi_k)}{2\sqrt{\pi r}}\Bigl[
\1_{\{|\theta-\psi_k|<a_k\}}V_+\bigl(K_k(t-r),K_k(t+r)\bigr)\\
&\hphantom{{}=\frac{C_kE_{Q_k}(\theta-\psi_k)}{2\sqrt{\pi r}}\Bigl[}{}+\1_{\{|\theta-\psi_k-\pi|<b_k\}}V_-\bigl(K_k(t-r),K_k(t+r)\bigr)\Bigr].
\end{aligned}
\end{equation}
The two terms in \eqref{eq:disc-S} are the first and opposite packets, supported in $\widehat{\calA}_k$ and $\widehat{\calB}_k$, respectively. Set the value at the origin to zero.

Then $S_k\in L^2((0,T)\times\D)$ and the following properties hold.
\begin{enumerate}[label=(\roman*),leftmargin=*]
\item \emph{(Sector support).} With $\mathcal U_k:=\widehat{\calA}_k\cup\widehat{\calB}_k$, the two sector indicators $\1_{\calA_k}$ and $\1_{\calB_k}$ give the exact support property
\begin{equation*}
S_k=0\quad\text{on }\bigl((0,T)\times\D\bigr)\setminus\mathcal U_k.
\end{equation*}

\item \emph{(Additional cancellation through profile matching).} Define
\begin{equation}\label{eq:disc-match}
\Gamma_k:=\{(t,r,\theta)\in\widehat{\calA}_k:q\bigl(K_k(t-r)\bigr)=q\bigl(K_k(t+r)+\pi/2\bigr)\}.
\end{equation}
Then $\1_{\Gamma_k}S_k=0$. If, in addition,
\begin{equation*}
\left|\{(d_1,d_2)\in(\R/2\pi\Z)^2:q(d_1)=q(d_2+\pi/2)\ne0\}\right|>0,
\end{equation*}
then, for all sufficiently large $K_k$, $\Gamma_k$ contains a positive-measure set on which the matched profile values are nonzero.
\end{enumerate}
\end{prop}

\begin{remark}[Parameters of the generated approximate velocity]\label{rem:disc-S-params}
Separating the overall scalar factor in \eqref{eq:disc-S}, we write its dependence as follows.\footnote{For readers' convenience, parameter remarks summarize the roles and dependencies of the construction parameters. The selection order is recalled in Step~1 of the proof of Proposition~\ref{prop:disc-real}. These summaries may be skipped on a first reading.}
\begin{equation*}
S_k=C_k\,S\!\left(q;\psi_k;K_k,a_k,b_k,Q_k\right).
\end{equation*}
The fixed profile $q$ determines $V_\pm$; $\psi_k$ gives the angular position. The remaining parameters correspond to the scales $K_k^{-1},a_k,b_k,Q_k^{-1}$.
\end{remark}

\begin{proof}[Proof of Lemma~\ref{lem:disc-leading}]
\noindent{\it Step 1. Choice of the spectral coefficients.}
Choose $P$ such that $\widehat q_p=0$ for $|p|>P$. We use the finite index set $\mathcal I_{P,Q}$ and radial indices from Definition~\ref{def:disc-indices}, together with the angular coefficients \eqref{eq:disc-angular-coefficients} and \eqref{eq:disc-equal-coeff}. By Proposition~\ref{prop:disc-modes}(i), every pair in $\mathcal I_{P,Q}$ satisfies
\begin{equation*}
n_{p,|m|}=pn_0-\frac{|m|}{2}-\frac{p-1}{4},
\qquad
j_{|m|,n_{p,|m|}}=pK+O_{P,Q}(n_0^{-1}).
\end{equation*}
The parity choice in $\alpha_{m,p}$ makes these radial indices integers, and they are positive for sufficiently large $n_0$.

Recall the velocity coefficients $c_{p,m}$ from \eqref{eq:disc-coeff}, which compensate the outgoing Bessel phase and translate the angular concentration to $\psi$.

\vspace{2mm}
\noindent{\it Step 2. Construction of an exact wave.}
Recall the exponential spectral expansion \eqref{eq:disc-wave-exp} and its velocity form \eqref{eq:disc-vel-exp} from Section~\ref{subsec:disc-polar}. We retain the modes selected in Step~1 and prescribe the coefficients of their negative-frequency velocity terms to be $Cc_{p,m}$. For a fixed pair, writing $j=j_{|m|,n_{p,|m|}}$, we have
\begin{equation*}
\partial_t\left(\frac{iCc_{p,m}}j e^{-ijt}\Phi_{m,n_{p,|m|}}(x)\right)
=Cc_{p,m}e^{-ijt}\Phi_{m,n_{p,|m|}}(x).
\end{equation*}
Thus we define the real wave by
\begin{equation}\label{eq:disc-waves}
u_{n_0}(t,x):=2\Real\left[iC\sum_{(p,m)\in\mathcal I_{P,Q}}
\frac{c_{p,m}}{j_{|m|,n_{p,|m|}}}e^{-ij_{|m|,n_{p,|m|}}t}\Phi_{m,n_{p,|m|}}(x)\right].
\end{equation}
By Lemma~\ref{lem:disc-basis}, each mode satisfies the Dirichlet wave equation exactly. A finite sum and its real part remain solutions, and differentiation gives
\begin{equation*}
\partial_tu_{n_0}(t,x)=2\Real\left[C\sum_{(p,m)\in\mathcal I_{P,Q}}
c_{p,m}e^{-ij_{|m|,n_{p,|m|}}t}\Phi_{m,n_{p,|m|}}(x)\right].
\end{equation*}

For later normalization, we identify the displacement and velocity coefficients in \eqref{eq:disc-wave-exp}. Recall from Lemma~\ref{lem:disc-basis} that $\overline{\Phi_{m,n}}=\Phi_{-m,n}$, and notice that $\mathcal I_{P,Q}$ is invariant under $m\mapsto-m$. The conjugate modes have the same radial index and frequency. Thus the coefficients of $u_{n_0}(0)$ and $\partial_tu_{n_0}(0)$ along $\Phi_{m,n_{p,|m|}}$ are, respectively,
\begin{equation*}
\frac{iC}{j_{|m|,n_{p,|m|}}}\bigl(c_{p,m}-\overline{c_{p,-m}}\bigr),
\qquad
C\bigl(c_{p,m}+\overline{c_{p,-m}}\bigr).
\end{equation*}
By orthonormality and the eigenvalue equation in Lemma~\ref{lem:disc-basis}, we obtain
\begin{equation*}
E[u_{n_0}](0)=\frac{C^2}{2}\sum_{(p,m)\in\mathcal I_{P,Q}}
\left(\bigl|c_{p,m}-\overline{c_{p,-m}}\bigr|^2
+\bigl|c_{p,m}+\overline{c_{p,-m}}\bigr|^2\right).
\end{equation*}
Relabeling $m$ as $-m$ in the conjugate contribution gives
\begin{equation}\label{eq:disc-leading-energy}
E[u_{n_0}](0)=2C^2\sum_{(p,m)\in\mathcal I_{P,Q}}|c_{p,m}|^2
=2C^2\sum_{\substack{0<p\le P\\p\ \mathrm{odd}}}|\widehat q_p|^2.
\end{equation}
The last equality uses $\sum_m|\alpha_{m,p}|^2=1$ for every positive odd $p$.

\Needspace{8\baselineskip}
\vspace{2mm}
\noindent{\it Step 3. The two branches and the packet expression.}
Write $\varphi=\theta-\psi$. Recall from \eqref{eq:bessel-basis} that $\Phi_{m,n}=(2\pi)^{-1/2}R_{|m|,n}e^{im\theta}$. We substitute the coefficient \eqref{eq:disc-coeff}, replace the wave frequency by $pK$ using Proposition~\ref{prop:disc-modes}(i), and replace the radial function by its leading cosine using part~(iii). The resulting leading expression is
\begin{equation*}
\frac{2C}{\sqrt{\pi r}}\Real\sum_{(p,m)\in\mathcal I_{P,Q}}
\widehat q_p\alpha_{m,p}e^{im\varphi}e^{i\delta_{|m|}}e^{-ipKt}
\cos(pKr-\delta_{|m|}),
\end{equation*}
where $\delta_{|m|}=|m|\pi/2+\pi/4$. Step~4 estimates the errors in these two replacements. We now identify this expression with $\widetilde S[q]$.
\vspace{2mm}

The compensated Bessel oscillation splits into two branches:
\begin{align*}
e^{i\delta_{|m|}}e^{-ipKt}\cos(pKr-\delta_{|m|})
&=\frac12\left[e^{-ipK(t-r)}+e^{2i\delta_{|m|}}e^{-ipK(t+r)}\right]\\
&=\frac12\left[e^{-ipK(t-r)}+i(-1)^{|m|}e^{-ipK(t+r)}\right].
\end{align*}
Substituting this identity and grouping the angular sums gives
\begin{equation*}
\frac{C}{\sqrt{\pi r}}\Real\sum_{\substack{0<p\le P\\p\ \mathrm{odd}}}
\widehat q_p\left[
e^{-ipK(t-r)}\sum_m\alpha_{m,p}e^{im\varphi}
+ie^{-ipK(t+r)}\sum_m(-1)^{|m|}\alpha_{m,p}e^{im\varphi}
\right].
\end{equation*}
Here the sums over $m$ use the coefficients \eqref{eq:disc-angular-coefficients}, which vanish outside the permitted angular indices. By \eqref{eq:disc-equal-coeff} and \eqref{eq:angular},
\begin{align*}
\sum_m\alpha_{m,p}e^{im\varphi}
&=\begin{cases}
E_Q(\varphi),&p\equiv1\pmod4,\\
O_Q(\varphi),&p\equiv3\pmod4,
\end{cases}\\
\sum_m(-1)^{|m|}\alpha_{m,p}e^{im\varphi}
&=\begin{cases}
E_Q(\varphi),&p\equiv1\pmod4,\\
-O_Q(\varphi),&p\equiv3\pmod4.
\end{cases}
\end{align*}

Thus the incoming coefficient is $i$ for the even angular sum and $-i$ for the odd angular sum. The leading expression becomes
\begin{equation}\label{eq:disc-finite-leading}
\begin{split}
\frac{C}{\sqrt{\pi r}}\Real\Bigl[
&E_Q(\varphi)\sum_{\substack{0<p\le P\\p\equiv1\ (4)}}\widehat q_p
\bigl(e^{-ipK(t-r)}+ie^{-ipK(t+r)}\bigr)\\
{}+{}&O_Q(\varphi)\sum_{\substack{0<p\le P\\p\equiv3\ (4)}}\widehat q_p
\bigl(e^{-ipK(t-r)}-ie^{-ipK(t+r)}\bigr)\Bigr].
\end{split}
\end{equation}
To identify its packets, set
\begin{equation*}
F_j(s):=\sum_{\substack{0<p\le P\\p\equiv j\ (4)}}\widehat q_p e^{-ips},
\qquad \forall j\in\{1,3\}.
\end{equation*}
For the reference coordinates $(d_1,d_2)$, the definitions of $q$ and $\calH q$ give
\begin{align*}
V_+(d_1,d_2)&=2\Real\bigl[F_1(d_1)+iF_1(d_2)+F_3(d_1)-iF_3(d_2)\bigr],\\
V_-(d_1,d_2)&=2\Real\bigl[F_1(d_1)+iF_1(d_2)-F_3(d_1)+iF_3(d_2)\bigr].
\end{align*}
Indeed, translation by $\pi/2$ multiplies $F_1$ by $-i$ and $F_3$ by $i$, while $\calH q=2\Real(iP_+q)$. Adding and subtracting these identities yields
\begin{align*}
V_+(d_1,d_2)+V_-(d_1,d_2)&=4\Real\bigl[F_1(d_1)+iF_1(d_2)\bigr],\\
V_+(d_1,d_2)-V_-(d_1,d_2)&=4\Real\bigl[F_3(d_1)-iF_3(d_2)\bigr].
\end{align*}
Evaluate these identities at $\bigl(K(t-r),K(t+r)\bigr)$. Since $E_Q,O_Q$ are real and $E_Q^\pm=(E_Q\pm O_Q)/2$, \eqref{eq:disc-finite-leading} equals
\begin{equation*}
\begin{split}
\frac{C}{2\sqrt{\pi r}}\Bigl[&E_Q^+(\varphi)V_+\bigl(K(t-r),K(t+r)\bigr)\\
&+E_Q^-(\varphi)V_-\bigl(K(t-r),K(t+r)\bigr)\Bigr]
=\widetilde S[q].
\end{split}
\end{equation*}
This is the expression in \eqref{eq:disc-leading-packets}.

\vspace{2mm}
\noindent{\it Step 4. Approximation on the whole disc.}
It remains to justify the two replacements used above: the wave frequency by $pK$, and the radial eigenfunction by its leading cosine. Recall from Remark~\ref{rem:disc-frequencies} that
\begin{equation*}
\sup_{t\in[0,T]}|e^{-ij_{|m|,n_{p,|m|}}t}-e^{-ipKt}|
\le T|j_{|m|,n_{p,|m|}}-pK|\longrightarrow0.
\end{equation*}
For the radial replacement, write
\begin{equation*}
\rho_{p,m}(n_0):=\left\|R_{|m|,n_{p,|m|}}-\sqrt2r^{-1/2}\cos(pKr-\delta_{|m|})\right\|_{L^2((0,1),r\dd r)}.
\end{equation*}
Lemma~\ref{lem:bessel} gives $\rho_{p,m}(n_0)\to0$ on the full radial interval, including the region near the origin. Since the exact eigenfunctions have norm one, the triangle inequality gives
\begin{equation*}
\|\partial_tu_{n_0}-\widetilde S[q]\|_{L^2((0,T)\times\D)}
\le2C\sqrt T\sum_{(p,m)\in\mathcal I_{P,Q}}
|c_{p,m}|\bigl(T|j_{|m|,n_{p,|m|}}-pK|+\rho_{p,m}(n_0)\bigr).
\end{equation*}
The coefficients and the finite set of active pairs are fixed as $n_0\to\infty$. Every term tends to zero, proving \eqref{eq:disc-leading-real}.
\end{proof}

\begin{proof}[Proof of Proposition~\ref{prop:disc-support}]

\noindent{\it Sector support.} Both indicators in \eqref{eq:disc-S} vanish outside $\mathcal U_k$, proving (i).

Each branch belongs to spacetime $L^2$: periodic integration controls its $t$-integral, $E_{Q_k}$ is a trigonometric polynomial, and the polar area density $r\dd r\dd\theta$ cancels the square of $r^{-1/2}$.
More precisely, for every $2\pi$-periodic $h\in L^2(0,2\pi)$ and $\beta\in\R$,
\begin{equation*}
\int_0^T\int_0^1
|r^{-1/2}h\bigl(K_k(t\pm r)+\beta\bigr)|^2r\dd r\dd t
\le\left(\frac T{2\pi}+\frac1{K_k}\right)\|h\|_{L^2(0,2\pi)}^2.
\end{equation*}

This is the generated approximate velocity $S_k=\calT_{\varepsilon_k}q$ of Section~\ref{subsec:aos-scales}; we also write $S_k[q]$ when emphasizing its profile dependence. In the notation of Section~\ref{subsec:aos-profiles}, the signed branch amplitudes are
\begin{equation*}
\begin{aligned}
(A_{1,\varepsilon_k},A_{2,\varepsilon_k})
&=\1_{\calA_k}A_{\varepsilon_k}(1,-1),\\
(A_{3,\varepsilon_k},A_{4,\varepsilon_k})
&=\1_{\calB_k}A_{\varepsilon_k}(1,1).
\end{aligned}
\end{equation*}
The first two propagation phases are $\Theta_{1,\varepsilon_k}(t,r)=K_k(t-r)$ and $\Theta_{2,\varepsilon_k}(t,r)=K_k(t+r)$. The third and fourth branches reuse these phases, respectively. Thus both packets use the same amplitude $A_{\varepsilon_k}$, with different branch signs.
\vspace{2mm}

\noindent{\it Additional cancellation through profile matching.} We now use equal profile values in the first sector; the opposite sector carries a different packet expression.

On the first sector, the formula takes the particularly simple form
\begin{equation*}
S_k(t,r,\theta)
=A_{\varepsilon_k}(r,\theta)
\Bigl[q\bigl(K_k(t-r)\bigr)-q\bigl(K_k(t+r)+\pi/2\bigr)\Bigr],
\qquad \forall(t,r,\theta)\in\widehat{\calA}_k.
\end{equation*}
On $\Gamma_k$, the four signed amplitudes satisfy $(A_{1,\varepsilon_k},A_{2,\varepsilon_k},A_{3,\varepsilon_k},A_{4,\varepsilon_k})=A_{\varepsilon_k}(1,-1,0,0)$. Thus their common real factor preserves the reference balance, as in \eqref{eq:balance-amp}.

The set $\Gamma_k$ pulls back $\calM_q$ through the two propagation phases. Hence
\begin{align*}
\1_{\Gamma_k}S_k
&=\1_{\Gamma_k}A_{\varepsilon_k}
\Bigl[q\bigl(K_k(t-r)\bigr)-q\bigl(K_k(t+r)+\pi/2\bigr)\Bigr]\\
&=0.
\end{align*}
The opposite sector carries $V_-$ and will be left unobserved.
\vspace{2mm}

For the positive-measure assertion, the map $(t,r)\mapsto\bigl(K_k(t-r),K_k(t+r)\bigr)$ has determinant $2K_k^2$. Its image of $(0,T)\times(0,1)$ contains a full periodicity cell when $K_k$ is sufficiently large. The additional hypothesis in (ii) therefore gives a positive-measure preimage with nonzero matching values. Multiplication by the angular interval $\calA_k$ proves the assertion in spacetime.

Inside $\calA_k$, the matching condition is independent of $\theta$ because the two branches share the same angular amplitude. The whole matching region still includes this angular restriction. In polar coordinates,
\begin{equation*}
\Gamma_k=
\bigl\{(t,r)\in(0,T)\times(0,1):
q\bigl(K_k(t-r)\bigr)=q\bigl(K_k(t+r)+\pi/2\bigr)\bigr\}\times\calA_k.
\qedhere
\end{equation*}
\end{proof}

\begin{remark}\label{rem:disc-support-matching}
A nonzero constant interval of $q$ suffices for the additional hypothesis, as for the square profile $q_0$. By contrast, $q(s)=\sin s$ has the same sector support and later exact-wave realization, but its reference matching set consists of curves and has zero area.
\end{remark}

\begin{figure}[htbp]
\centering
\includegraphics[width=\textwidth,
  trim={0mm 7mm 0mm 3mm},
  clip]{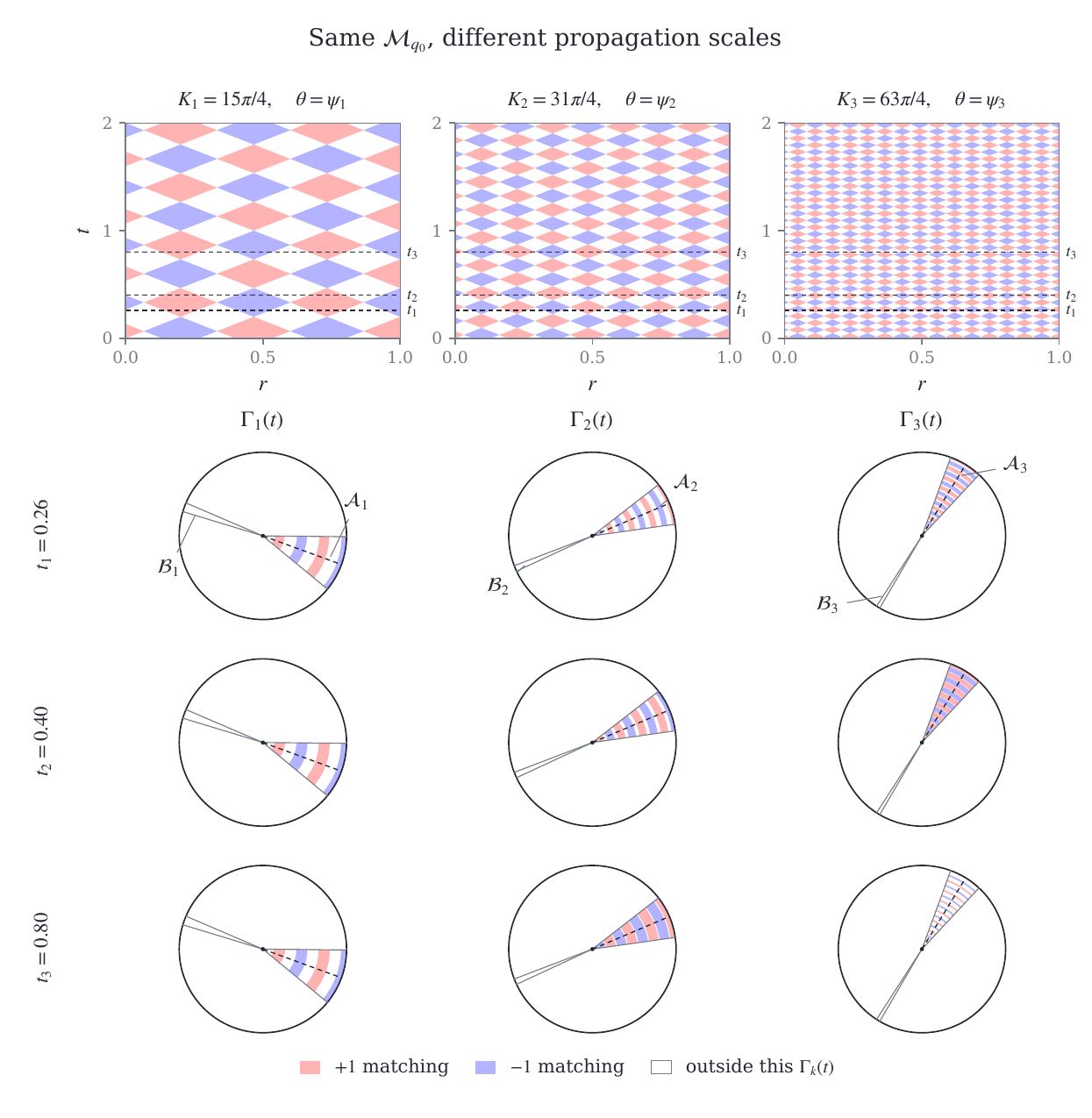}
\caption{Matching regions $\Gamma_k$ for $q=q_0$, with illustrative radial indices $n_{0,k}=4,8,16$. The horizontal lines in the top $(r,t)$-sections mark the three time slices below; the dashed radii locate these sections in the disc. The panels depict $\Gamma_k(t)$, with its complement in white; the observation region $G$ is defined in Section~\ref{subsubsec:disc-observation}.}\label{fig:disc-gamma-scales}
\end{figure}

\subsubsection{A fixed observation region}\label{subsubsec:disc-observation}

We now combine the scales into one observation region $G$. The new assumption is that different sector pairs are disjoint; it allows the matching rule for each $S_k$ to persist through all later stages.

\begin{prop}\label{prop:disc-observation}
Fix $T>0$ and $q\in X$ as in Proposition~\ref{prop:disc-support}. Let $S_k,\Gamma_k,\widehat{\calA}_k,\widehat{\calB}_k$ be given by Proposition~\ref{prop:disc-support}. Suppose, in addition, that the angular intervals $\calA_k,\calB_k$, $k\ge1$, are pairwise disjoint. Define
\begin{equation}\label{eq:disc-G}
\begin{split}
G^{(N)}&:=\Bigl(\bigl((0,T)\times\D\bigr)\setminus\textstyle\bigcup_{k\le N}(\widehat{\calA}_k\cup\widehat{\calB}_k)\Bigr)\cup\bigcup_{k\le N}\Gamma_k,\\
G&:=\bigcap_{N\ge1}G^{(N)}.
\end{split}
\end{equation}
These are Borel sets satisfying
\begin{gather*}
G^{(N+1)}\subset G^{(N)},\qquad G^{(N)}\downarrow G,\\
G^{(N)}\cap(\widehat{\calA}_k\cup\widehat{\calB}_k)=\Gamma_k,
\qquad \forall k\in\{1,\ldots,N\}.
\end{gather*}
Consequently,
\begin{equation}\label{eq:disc-zero}
\begin{aligned}
\1_{G^{(N)}}S_k&=0,\qquad &&\forall k\in\{1,\ldots,N\},\\
\1_GS_k&=0,\qquad &&\forall k\ge1.
\end{aligned}
\end{equation}
\end{prop}

\begin{remark}[Parameters of the observation region]\label{rem:disc-params}
To make the dependence on the choices explicit, we write
\begin{equation}\label{eq:disc-G-param}
G=G\!\left(q,T;\{\psi_k\}_{k\ge1};\{K_k,a_k,b_k\}_{k\ge1}\right).
\end{equation}
For fixed $q$ and $T$, choosing the admissible families $\{\psi_k\}_{k\ge1}$ and $\{K_k,a_k,b_k\}_{k\ge1}$ determines $G$. The further choices serve two purposes:
\begin{itemize}[leftmargin=*,itemsep=3pt]
\item \emph{Non-observability.} In Section~\ref{sec:disc}, we choose the profile and angular cutoffs $P_k,Q_k$, and then take $K_k$ sufficiently large for the approximation estimates; see \eqref{eq:disc-real-params}. Proposition~\ref{prop:disc-real} then yields non-observability on the resulting fixed set $G$.
\item \emph{GCC.} For a finite-step profile and $T=T_\D$, Section~\ref{sec:disc-class} further restricts the sector widths by \eqref{eq:disc-widths}, makes their total width small, and chooses the frequencies so that $\sum_kK_k^{-1}$ is sufficiently small. These choices give a uniform positive observed time for the same set $G$.
\end{itemize}
\end{remark}

\begin{proof}[Proof of Proposition~\ref{prop:disc-observation}]
Recall that $\mathcal U_k=\widehat{\calA}_k\cup\widehat{\calB}_k$ and $\Gamma_k\subset\widehat{\calA}_k$. By \eqref{eq:disc-G}, outside the chosen sector pairs we observe everything. Inside each first sector we retain only $\Gamma_k$, and inside each opposite sector we observe nothing. Pairwise disjointness gives
\begin{equation*}
\Gamma_{N+1}\subset\mathcal U_{N+1}
\subset\bigl((0,T)\times\D\bigr)\setminus\bigcup_{j\le N}\mathcal U_j
\subset G^{(N)}.
\end{equation*}
Thus the next stage changes only the new pair:
\begin{equation*}
G^{(N+1)}=
\bigl(G^{(N)}\setminus(\widehat{\calA}_{N+1}\cup\widehat{\calB}_{N+1})\bigr)
\cup\Gamma_{N+1}\subset G^{(N)}.
\end{equation*}
The definition of $G$ now gives $G^{(N)}\downarrow G$.
\vspace{2mm}

For $k\le N$, the unmodified complement in \eqref{eq:disc-G} does not meet $\mathcal U_k$, and $\Gamma_j\cap\mathcal U_k=\varnothing$ for $j\ne k$. Hence
\begin{equation*}
G^{(N)}\cap\mathcal U_k=\Gamma_k,\qquad
G^{(N)}\cap\widehat{\calA}_k=\Gamma_k,\qquad
G^{(N)}\cap\widehat{\calB}_k=\varnothing.
\end{equation*}
These are the finite-stage relations of Section~\ref{subsec:aos-real}, with $\calE_k=\varnothing$. They persist at every later stage.

Passing to $G$ gives the observation rule in polar coordinates, for $t\in(0,T)$ and $r\in(0,1)$:
\begin{equation}\label{eq:disc-G-rule}
\1_G(t,r,\theta)=
\begin{cases}
\1_{\Gamma_k}(t,r,\theta),&\theta\in\calA_k,\\
0,&\theta\in\calB_k,\\
1,&\theta\notin\bigcup_j(\calA_j\cup\calB_j).
\end{cases}
\end{equation}
This formula separates the three contributions to observation that will be used in Section~\ref{sec:disc-class}.

Proposition~\ref{prop:disc-support}(i) gives $S_k=0$ outside $\mathcal U_k$, while Proposition~\ref{prop:disc-support}(ii) gives $\1_{\Gamma_k}S_k=0$. Combining these two conclusions with the intersection identity above, for $k\le N$ we obtain
\begin{equation*}
\1_{G^{(N)}}S_k=\1_{G^{(N)}\cap\mathcal U_k}S_k=\1_{\Gamma_k}S_k=0.
\end{equation*}
Since $G\subset G^{(N)}$, this proves \eqref{eq:disc-zero}. The Borel property follows from the fixed Borel profile, the continuous propagation phases, and countable set operations.
\end{proof}

For the later \hyperref[eq:gcc]{GCC} proof, we will place all $\calA_k$ in one short arc and take $b_k\ll a_k\to0^+$. 
The wider first sectors allow observed time there to compensate for time spent in the unobserved opposite sectors, while Figure~\ref{fig:disc-stages} illustrates the multiscale assembly of the observation region.


\begin{figure}[htbp]
\centering
\includegraphics[width=\textwidth,
  trim={0mm 6mm 0mm 2mm},
  clip]{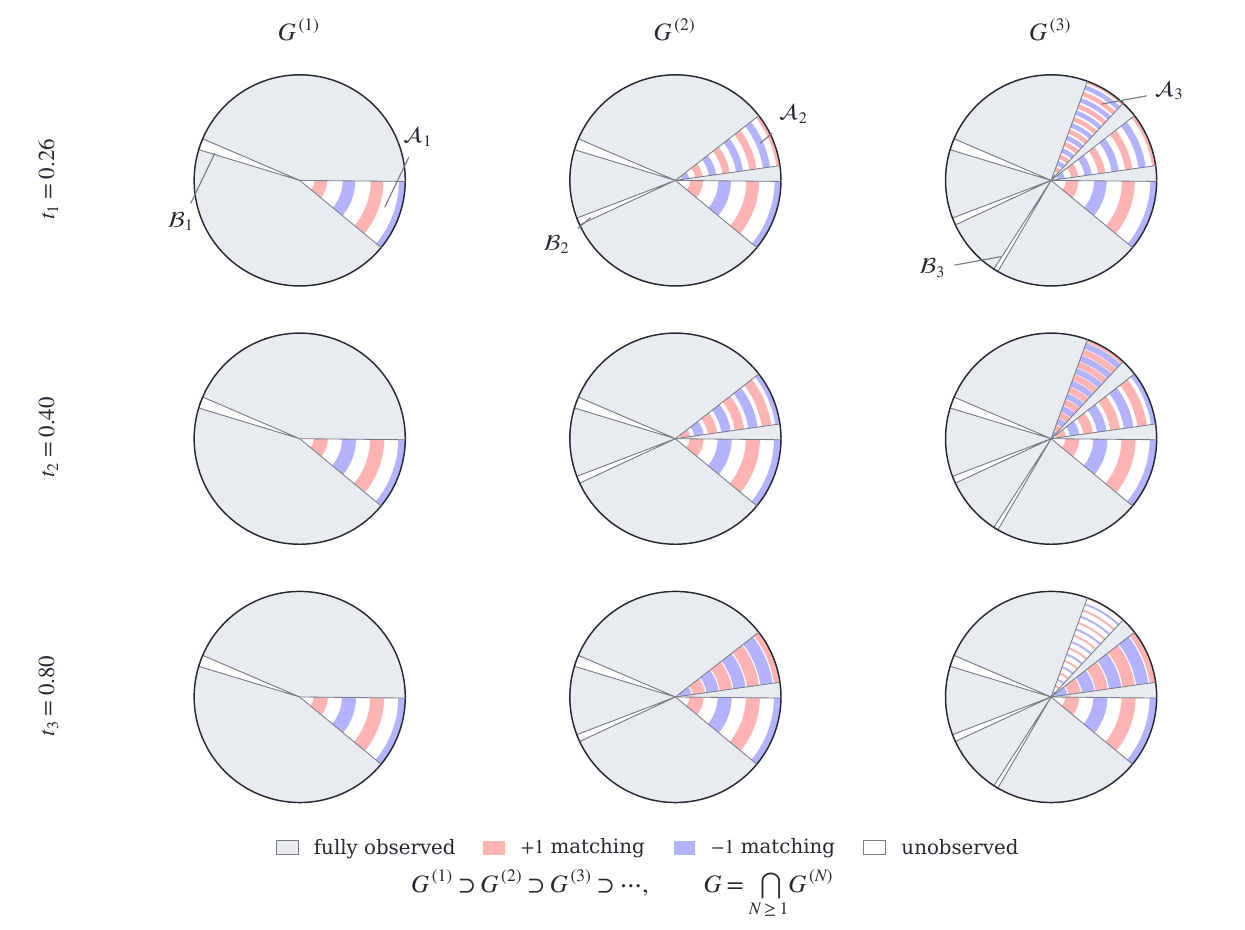}
\caption{Finite-stage observation regions $G^{(1)},G^{(2)},G^{(3)}$ at the same times and scales as Figure~\ref{fig:disc-gamma-scales}. Each column fixes the stage, and each row fixes the physical time. White is unobserved, including every opposite sector $\calB_k$ already introduced at that stage.}\label{fig:disc-stages}
\end{figure}

\subsection{Exact waves and the asymptotic realization formula}\label{subsec:disc-real}

We now construct solutions of \eqref{eq:wave} whose time derivatives approximate $S_k$ in \eqref{eq:disc-S}. The cancellation in \eqref{eq:disc-zero} already holds for every admissible choice of the parameters. It remains to choose them so that exact-wave realization holds as well. See also Appendix~\ref{subsec:supp-local} for a discussion of strict localization and velocity approximation.

We use the two-step procedure of Section~\ref{subsec:aos-real}: construct a finite spectral wave and then normalize its energy.

\Needspace{10\baselineskip}
\begin{prop}\label{prop:disc-real}
Fix a nonzero $q\in X$. Let $T>0$. Fix angular centers $\psi_k$ and widths $0<a_k,b_k<\pi/4$ such that $a_k,b_k\to0^+$ and the sectors $\calA_k,\calB_k$ in \eqref{eq:sectors} are pairwise disjoint. One can choose Fourier cutoffs $P_k$, angular Fourier cutoffs $Q_k$, as well as base radial indices $n_{0,k}\in\Nat$ and scalar factors $C_k>0$, such that
\begin{gather*}
P_k,Q_k,n_{0,k}\longrightarrow\infty,\qquad
Q_k\min\{a_k,b_k\}\longrightarrow\infty,\\
K_k=\pi(n_{0,k}-\tfrac14),\qquad
0<\inf_k C_k\le\sup_k C_k<\infty,
\end{gather*}

Use these choices in Propositions~\ref{prop:disc-support} and \ref{prop:disc-observation} to define $S_k,\Gamma_k$ and $G^{(N)}\downarrow G$. There are real solutions $u_k$ of \eqref{eq:wave}, each given by a finite sum of Dirichlet modes indexed by $(p,m)\in\mathcal I_{P_k,Q_k}$, with radial indices prescribed by Definition~\ref{def:disc-indices}, such that:
\begin{enumerate}[label=(\roman*),leftmargin=*]
\item \emph{(Asymptotic realization).} $E[u_k](0)=1$ and
\begin{equation}\label{eq:disc-real}
\partial_tu_k=S_k+r_k \textrm{ \; with \; } \|r_k\|_{L^2((0,T)\times\D)}\to0.
\end{equation}
\item \emph{(Vanishing observation on the fixed set).}
\begin{gather*}
\1_{G^{(N)}}\partial_tu_k=\1_{G^{(N)}}r_k,
\quad \forall k\in\{1,\ldots,N\},\\
\|\1_G\partial_tu_k\|_{L^2((0,T)\times\D)}
\le\|r_k\|_{L^2((0,T)\times\D)}\longrightarrow0.
\end{gather*}
Thus observability fails on this one fixed Borel set $G$.
\end{enumerate}
\end{prop}

\noindent{\it Idea of approximation.}
Lemma~\ref{lem:disc-leading} approximates the derivative of an exact wave by the leading expression $\widetilde S[q^{(P_k)}]$ in \eqref{eq:disc-leading-packets}. Lemma~\ref{lem:angular} then controls the angular replacements and tails that give the sector-supported expression \eqref{eq:disc-S}. Finally, a uniform $L^2$ estimate restores the full profile $q$. The three comparisons in spacetime $L^2$ are
\begin{equation*}
\partial_tu_k
\underset{\substack{\text{Bessel}\\\text{asymptotics}}}{\approx}
\widetilde S[q^{(P_k)}]
\underset{\substack{\text{angular}\\\text{localization}}}{\approx}
\calT_{\varepsilon_k}q^{(P_k)}
\underset{\substack{\text{profile}\\\text{tail}}}{\approx}
S_k=\calT_{\varepsilon_k}q.
\end{equation*}

\vspace{2mm}
\begin{proof}
\noindent{\it Step 1. Finite-profile realization and normalization.}
Keep the prescribed sectors $\calA_k,\calB_k$ fixed. We choose $P_k$ to control the Fourier tail of $q$ and $Q_k$ to make the angular concentration finer than both sector widths. The energy normalization then determines $C_k$. With these choices fixed, we take the base radial index $n_{0,k}$ sufficiently large. The order of construction is
\begin{equation*}
(q,\psi_k,a_k,b_k)
\ \longrightarrow\ (P_k,Q_k,C_k)
\ \longrightarrow\ (n_{0,k},K_k)
\ \longrightarrow\ (u_k,S_k,\Gamma_k,G).
\end{equation*}

For $P_k\in\Nat$, define
\begin{equation*}
q^{(P_k)}:=2\Real P_{+,P_k}q
\qquad\text{and}\qquad
P_{+,P_k}q(s):=\sum_{\substack{0<p\le P_k\\p\ \mathrm{odd}}}\widehat q_p e^{-ips}.
\end{equation*}
Since $q\ne0$, choose a positive odd frequency $p_*$ with $\widehat q_{p_*}\ne0$. Choose integers $P_k\ge p_*$ with $P_k\to\infty$. Then $q^{(P_k)}\ne0$ for every $k$, and $q^{(P_k)}\to q$ in $L^2(0,2\pi)$. Choose angular cutoffs $Q_k\in\Nat$ such that
\begin{equation*}
Q_k\min\{a_k,b_k\}\longrightarrow\infty.
\end{equation*}
The energy identity \eqref{eq:disc-leading-energy} in the proof of Lemma~\ref{lem:disc-leading} prescribes the normalization
\begin{equation*}
C_k:=\left(2\sum_{\substack{0<p\le P_k\\p\ \mathrm{odd}}}|\widehat q_p|^2\right)^{-1/2}.
\end{equation*}
Moreover,
\begin{equation*}
0<|\widehat q_{p_*}|^2
\le\sum_{\substack{0<p\le P_k\\p\ \mathrm{odd}}}|\widehat q_p|^2
\le\sum_{\substack{p>0\\p\ \mathrm{odd}}}|\widehat q_p|^2<\infty.
\end{equation*}
Thus $C_k$ stays uniformly bounded above and below. For each $k$, we use this same value of $C_k$ in $\calT_{\varepsilon_k}h$ for every input $h\in X$.

\vspace{2mm}
Fix $k$, so that $P_k,Q_k,C_k$ and $\psi_k$ are now fixed. Recall from \eqref{eq:disc-mode-scales} that, for every $n_0\in\Nat$,
\begin{equation*}
K(n_0)=\pi(n_0-\tfrac14),\qquad
n_{p,|m|;n_0}=pn_0-\frac{|m|}{2}-\frac{p-1}{4}.
\end{equation*}
For every sufficiently large $n_0$, Lemma~\ref{lem:disc-leading} gives a real wave $u_{k,n_0}$ by \eqref{eq:disc-waves}, using $q^{(P_k)}$, $Q_k$, $\psi_k$ and $C_k$. By \eqref{eq:disc-leading-energy}, its energy is
\begin{equation*}
E[u_{k,n_0}](0)=2C_k^2\sum_{\substack{0<p\le P_k\\p\ \mathrm{odd}}}|\widehat q_p|^2=1.
\end{equation*}
Let $\widetilde S_{k,n_0}$ be the leading expression \eqref{eq:disc-leading-packets} for the profile $q^{(P_k)}$, with $K=K(n_0)$, $Q=Q_k$, $\psi=\psi_k$ and $C=C_k$. By \eqref{eq:disc-leading-real},
\begin{equation*}
\|\partial_tu_{k,n_0}-\widetilde S_{k,n_0}\|_{L^2((0,T)\times\D)}\longrightarrow0
\qquad\text{as }n_0\to\infty.
\end{equation*}
This limit is taken with $k$ fixed. Moreover, Proposition~\ref{prop:disc-modes}(i) gives, for the same finite set of pairs,
\begin{equation*}
\omega_{p,m;k,n_0}:=j_{|m|,n_{p,|m|;n_0}} = pK(n_0)+O_{P_k,Q_k}(n_0^{-1}),
\quad (p,m)\in\mathcal I_{P_k,Q_k}.
\end{equation*}

Since $p\ge1$, all these frequencies are at least $K(n_0)/2$ for sufficiently large $n_0$. For the fixed $P_k,Q_k,C_k$ and $\psi_k$, choose a threshold $N_k\in\Nat$ such that, for every $n_0\ge N_k$, all prescribed radial indices are positive and the realization error is at most $2^{-k}$, while every active frequency is at least $K(n_0)/2$. We then take $n_{0,k}\ge\max\{k,N_k\}$ and set
\begin{equation}\label{eq:disc-selected}
K_k:=K(n_{0,k}),\quad
u_k:=u_{k,n_{0,k}}, \textrm{ \; and \; }
\widetilde S_k:=\widetilde S_{k,n_{0,k}},
\end{equation}
so that
\begin{equation}\label{eq:disc-spectral-choice}
\begin{gathered}
\|\partial_tu_k-\widetilde S_k\|_{L^2((0,T)\times\D)}\le2^{-k},\\
\omega_{p,m,k}:=\omega_{p,m;k,n_{0,k}}\ge K_k/2,
\qquad (p,m)\in\mathcal I_{P_k,Q_k}.
\end{gathered}
\end{equation}
When choosing $n_{0,k}$, we may also impose the further lower bounds required for the \hyperref[eq:gcc]{GCC} estimate in Section~\ref{sec:disc-class}. The final parameter choices determine the fixed observation region $G$ through \eqref{eq:disc-G}.

\vspace{2mm}
\noindent{\it Step 2. Angular replacements and sector support.}
Write $\varphi=\theta-\psi_k$ and evaluate the packet expressions \eqref{eq:two-prof}, computed from $q^{(P_k)}$, at the propagation phases:
\begin{equation*}
V_{\pm,k}(t,r):=V_\pm\bigl(K_k(t-r),K_k(t+r)\bigr).
\end{equation*}
Recall the leading expression \eqref{eq:disc-leading-packets} and the sector-supported expression \eqref{eq:disc-S}:
\begin{align*}
\widetilde S_k
&=\frac{C_k}{2\sqrt{\pi r}}\bigl[E_{Q_k}^+(\varphi)V_{+,k}(t,r)+E_{Q_k}^-(\varphi)V_{-,k}(t,r)\bigr],\\
\calT_{\varepsilon_k}q^{(P_k)}
&=\frac{C_kE_{Q_k}(\varphi)}{2\sqrt{\pi r}}
\bigl[\1_{\calA_k}(\theta)V_{+,k}(t,r)+\1_{\calB_k}(\theta)V_{-,k}(t,r)\bigr].
\end{align*}
On $\calA_k$, we replace $(E_{Q_k}^+,E_{Q_k}^-)$ by $(E_{Q_k},0)$; on $\calB_k$, by $(0,E_{Q_k})$; outside both sectors, by $(0,0)$.
\vspace{2mm}

To compute the error on $\calA_k$, use \eqref{eq:disc-packet-amp}:
\begin{equation*}
E_{Q_k}^{+}-E_{Q_k}=\frac{O_{Q_k}-E_{Q_k}}2,
\qquad
E_{Q_k}^{-}=-\frac{O_{Q_k}-E_{Q_k}}2.
\end{equation*}
Thus the two terms combine into the factor $(O_{Q_k}-E_{Q_k})\bigl(V_{+,k}(t,r)-V_{-,k}(t,r)\bigr)$. The corresponding identities on $\calB_k$ give
\begin{equation}\label{eq:disc-angular-error}
\widetilde S_k-\calT_{\varepsilon_k}q^{(P_k)}
=\frac{C_k\bigl(V_{+,k}(t,r)-V_{-,k}(t,r)\bigr)}{4\sqrt{\pi r}}
\begin{cases}
O_{Q_k}(\varphi)-E_{Q_k}(\varphi),&\theta\in\calA_k,\\
E_{Q_k}(\varphi)+O_{Q_k}(\varphi),&\theta\in\calB_k.
\end{cases}
\end{equation}
Outside their union, the error is the discarded tail of $\widetilde S_k$. Thus Lemma~\ref{lem:angular} controls both the replacements and the tails. We quantify these errors together with the profile truncation below.

\vspace{2mm}
\noindent{\it Step 3. Control of the realization error.}
Recall that the full profile defines $S_k=\calT_{\varepsilon_k}q$. All parameters have now been fixed in Step~1. Since the same scalar $C_k$ is used for every input of the linear operator $\calT_{\varepsilon_k}$, we have
\begin{equation*}
\calT_{\varepsilon_k}q^{(P_k)}-\calT_{\varepsilon_k}q
=\calT_{\varepsilon_k}(q^{(P_k)}-q).
\end{equation*}
We therefore separate the realization error as
\begin{equation}\label{eq:disc-error-split}
\begin{aligned}
r_k&=\partial_tu_k-\calT_{\varepsilon_k}q\\
&=(\partial_tu_k-\widetilde S_k)
+(\widetilde S_k-\calT_{\varepsilon_k}q^{(P_k)})
+\calT_{\varepsilon_k}(q^{(P_k)}-q).
\end{aligned}
\end{equation}
The first term is bounded by $2^{-k}$ by \eqref{eq:disc-spectral-choice}. To estimate the second term, we first control the packet factors in \eqref{eq:disc-angular-error}.
\vspace{2mm}

\noindent{\it From profile bounds to spacetime bounds.} Recall the periodic integration used in the proof of Proposition~\ref{prop:disc-support}. For every $2\pi$-periodic $h\in L^2(0,2\pi)$, $K\ge1$ and $\beta\in\R$,
\begin{align*}
\int_0^T|h\bigl(K(t\pm r)+\beta\bigr)|^2\dd t
&=\frac1K\int_{\pm Kr+\beta}^{KT\pm Kr+\beta}|h(s)|^2\dd s\\
&\le\frac{\lfloor KT/(2\pi)\rfloor+1}{K}\|h\|_{L^2(0,2\pi)}^2\\
&\le\left(\frac{T}{2\pi}+\frac1K\right)\|h\|_{L^2(0,2\pi)}^2,
\qquad \forall r\in(0,1).
\end{align*}
Integrating in $r\in(0,1)$ gives
\begin{equation*}
\int_0^T\int_0^1|h\bigl(K(t\pm r)+\beta\bigr)|^2\dd r\dd t
\le\left(\frac{T}{2\pi}+1\right)\|h\|_{L^2(0,2\pi)}^2,
\qquad \forall K\ge1,\quad\forall\beta\in\R.
\end{equation*}

We apply this estimate to the packet expressions \eqref{eq:two-prof}, with profile $q^{(P_k)}$ as in Step~2. For the first packet,
\begin{align*}
&\int_0^T\int_0^1
|V_+\bigl(K_k(t-r),K_k(t+r)\bigr)|^2\dd r\dd t\\
&\quad\le2\int_0^T\int_0^1
|q^{(P_k)}\bigl(K_k(t-r)\bigr)|^2\dd r\dd t\\
&\qquad+2\int_0^T\int_0^1
|q^{(P_k)}\bigl(K_k(t+r)+\pi/2\bigr)|^2\dd r\dd t\\
&\quad\le4\left(\frac{T}{2\pi}+1\right)
\|q^{(P_k)}\|_{L^2(0,2\pi)}^2
\le C_T\|q\|_{L^2(0,2\pi)}^2.
\end{align*}
For the opposite packet, the two profiles are translates of $\calH q^{(P_k)}$. Since $q^{(P_k)}$ has zero mean, the periodic Hilbert transform satisfies
\begin{equation*}
\|\calH q^{(P_k)}\|_{L^2(0,2\pi)}
=\|q^{(P_k)}\|_{L^2(0,2\pi)}.
\end{equation*}
The same calculation therefore gives the same bound for $V_-$. Recall the composed packet expressions $V_{\pm,k}$ from Step~2, still with profile $q^{(P_k)}$. Using $|V_{+,k}-V_{-,k}|^2\le2|V_{+,k}|^2+2|V_{-,k}|^2$, we obtain
\begin{equation}\label{eq:disc-packet-bound}
\int_0^T\int_0^1
\bigl(|V_{+,k}|^2+|V_{-,k}|^2+|V_{+,k}-V_{-,k}|^2\bigr)\dd r\dd t
\le C_T\|q\|_{L^2(0,2\pi)}^2,
\qquad\forall k\ge1.
\end{equation}

\noindent{\it Angular replacements and tails.} Recall that $\widehat{\calA}_k$ and $\widehat{\calB}_k$ are the spacetime sectors in \eqref{eq:disc-spacetime-sectors}. In \eqref{eq:disc-angular-error}, the angular factors depend only on $\theta$, whereas $V_{\pm,k}$ depend only on $(t,r)$. Moreover, squaring the radial factor $r^{-1/2}$ cancels the polar area density:
\begin{equation*}
\int_{\D}|r^{-1/2}F(r,\theta)|^2\dd x
=\int_0^{2\pi}\int_0^1|F(r,\theta)|^2\dd r\dd\theta.
\end{equation*}
Thus the error on the first sector satisfies
\begin{align*}
&\|\1_{\widehat{\calA}_k}(\widetilde S_k-\calT_{\varepsilon_k}q^{(P_k)})\|_{L^2((0,T)\times\D)}^2\\
&\quad=\frac{C_k^2}{16\pi}
\|O_{Q_k}-E_{Q_k}\|_{L^2(\calA_k-\psi_k)}^2
\int_0^T\int_0^1|V_{+,k}-V_{-,k}|^2\dd r\dd t\\
&\quad\le C_{T,q}\|O_{Q_k}-E_{Q_k}\|_{L^2(\calA_k-\psi_k)}^2\\
&\quad\le C_{T,q}\left(\frac1{Q_k^2}+\frac{a_k}{Q_k}\right).
\end{align*}
The first inequality uses \eqref{eq:disc-packet-bound} and the uniform upper bound for $C_k$ from Step~1. The second uses Lemma~\ref{lem:angular} on $\calA_k-\psi_k=(-a_k,a_k)$, interpreted modulo $2\pi$.
\vspace{2mm}

On the opposite sector, \eqref{eq:disc-angular-error} gives the angular factor $E_{Q_k}+O_{Q_k}$. Applying \eqref{eq:disc-packet-bound} and the estimate near $\pi$ in Lemma~\ref{lem:angular}, we similarly obtain
\begin{align*}
&\|\1_{\widehat{\calB}_k}(\widetilde S_k-\calT_{\varepsilon_k}q^{(P_k)})\|_{L^2((0,T)\times\D)}^2\\
&\quad\le C_{T,q}\|E_{Q_k}+O_{Q_k}\|_{L^2(\calB_k-\psi_k)}^2
\le C_{T,q}\left(\frac1{Q_k^2}+\frac{b_k}{Q_k}\right).
\end{align*}

It remains to estimate the error outside $\mathcal U_k:=\widehat{\calA}_k\cup\widehat{\calB}_k$. There $\calT_{\varepsilon_k}q^{(P_k)}=0$, so the error is $\widetilde S_k$. The angular factors in \eqref{eq:disc-packet-amp} satisfy
\begin{equation*}
|E_{Q_k}^+|^2+|E_{Q_k}^-|^2
=\frac12\bigl(|E_{Q_k}|^2+|O_{Q_k}|^2\bigr).
\end{equation*}
Using this identity, the leading expression in Step~2, and \eqref{eq:disc-packet-bound}, we deduce
\begin{align*}
\|\1_{\mathcal U_k^c}\widetilde S_k\|_{L^2((0,T)\times\D)}^2
&\le C_{T,q}\int_{\theta\notin\calA_k\cup\calB_k}
\bigl(|E_{Q_k}(\theta-\psi_k)|^2+|O_{Q_k}(\theta-\psi_k)|^2\bigr)\dd\theta\\
&\le\frac{C_{T,q}}{Q_k\min\{a_k,b_k\}}.
\end{align*}
For the last inequality, notice that $\dist(\theta-\psi_k,\{0,\pi\})\ge\min\{a_k,b_k\}$ outside the two angular sectors, and apply \eqref{eq:angular-tail}. Adding the squared norms on the two disjoint sectors and their complement gives
\begin{equation}\label{eq:disc-angular-bound}
\|\widetilde S_k-\calT_{\varepsilon_k}q^{(P_k)}\|_{L^2((0,T)\times\D)}^2
\le C_{T,q}\left(\frac1{Q_k^2}+\frac{a_k+b_k}{Q_k}
+\frac1{Q_k\min\{a_k,b_k\}}\right).
\end{equation}
Here $C_{T,q}$ is independent of $k$, the spectral cutoffs, and the sector widths.

\noindent{\it Restoring the full profile.} Keep the parameters and the scalar $C_k$ fixed, and let $h\in X$. Repeat the periodic calculation above for the two packets with profile $h$, using $\|\calH h\|_{L^2(0,2\pi)}=\|h\|_{L^2(0,2\pi)}$. In \eqref{eq:disc-S}, the two sector indicators have disjoint supports. Integrating the squared expression and canceling the radial factor as above, we obtain
\begin{align*}
\|\calT_{\varepsilon_k}h\|_{L^2((0,T)\times\D)}^2
&\le C_TC_k^2\left(\int_{\calA_k\cup\calB_k}|E_{Q_k}(\theta-\psi_k)|^2\dd\theta\right)\|h\|_{L^2(0,2\pi)}^2\\
&\le C_TC_k^2\|E_{Q_k}\|_{L^2(0,2\pi)}^2\|h\|_{L^2(0,2\pi)}^2\\
&\le C_{T,q}\|h\|_{L^2(0,2\pi)}^2,
\qquad \forall k\ge1,\quad\forall h\in X.
\end{align*}
The last inequality uses $\|E_{Q_k}\|_{L^2(0,2\pi)}^2=2\pi$ from Lemma~\ref{lem:angular} and the uniform bound for $C_k$ from Step~1. Taking $h=q^{(P_k)}-q$ gives
\begin{equation*}
\|\calT_{\varepsilon_k}(q^{(P_k)}-q)\|_{L^2((0,T)\times\D)}^2
\le C_{T,q}\|q^{(P_k)}-q\|_{L^2(0,2\pi)}^2\longrightarrow0.
\end{equation*}
Finally, the decomposition \eqref{eq:disc-error-split} gives
\begin{align*}
\|r_k\|_{L^2((0,T)\times\D)}^2
&\le3\|\partial_tu_k-\widetilde S_k\|_{L^2((0,T)\times\D)}^2\\
&\quad+3\|\widetilde S_k-\calT_{\varepsilon_k}q^{(P_k)}\|_{L^2((0,T)\times\D)}^2\\
&\quad+3\|\calT_{\varepsilon_k}(q^{(P_k)}-q)\|_{L^2((0,T)\times\D)}^2.
\end{align*}
Combining the profile-tail bound with \eqref{eq:disc-angular-bound} and \eqref{eq:disc-spectral-choice}, we obtain
\begin{equation}\label{eq:disc-real-error}
\begin{split}
\|r_k\|_{L^2((0,T)\times\D)}^2\le C_{T,q}\Bigl[
&\|q^{(P_k)}-q\|_{L^2(0,2\pi)}^2
+\frac1{Q_k^2}+\frac{a_k+b_k}{Q_k}\\
&+\frac1{Q_k\min\{a_k,b_k\}}+2^{-2k}\Bigr]\longrightarrow0.
\end{split}
\end{equation}
The profile tail tends to zero because $P_k\to\infty$. The next two terms tend to zero because $Q_k\to\infty$ and $a_k,b_k<\pi/4$. The angular tail tends to zero by $Q_k\min\{a_k,b_k\}\to\infty$. The last term comes from the choice of $n_{0,k}$ in \eqref{eq:disc-spectral-choice}. All these choices were made in Step~1, and the constants above are uniform along the resulting sequence. This proves (i).
\vspace{2mm}

\noindent{\it Vanishing observation.} Combining (i) with Proposition~\ref{prop:disc-observation}, we obtain
\begin{align*}
\1_{G^{(N)}}\partial_tu_k
&=\1_{G^{(N)}}S_k+\1_{G^{(N)}}r_k\\
&=\1_{G^{(N)}}r_k,
\qquad \forall k\in\{1,\ldots,N\}.
\end{align*}
Since $G\subset G^{(N)}$, it follows that
\begin{equation*}
\|\1_G\partial_tu_k\|_{L^2((0,T)\times\D)}
\le\|r_k\|_{L^2((0,T)\times\D)}\longrightarrow0.
\end{equation*}
This proves (ii).

\noindent\emph{Conclusion on the choice of parameters.}
For the prescribed sectors and a fixed positive odd $p_*$ with $\widehat q_{p_*}\ne0$, sufficient parameter conditions for the conclusions of this proposition are
\begin{equation}\label{eq:disc-real-params}
\begin{cases}
P_k\ge p_*,\qquad P_k\longrightarrow\infty,\\[1mm]
Q_k\min\{a_k,b_k\}\longrightarrow\infty,\\[1mm]
C_k=\displaystyle\left(2\sum_{\substack{0<p\le P_k\\p\ \mathrm{odd}}}|\widehat q_p|^2\right)^{-1/2},\\[3mm]
n_{0,k}\ge\max\{k,N_k\},\qquad K_k=\pi(n_{0,k}-\tfrac14).
\end{cases}
\end{equation}
Here $N_k$ is chosen in Step~1 after fixing the finite cutoffs $P_k,Q_k$ and the normalization factor $C_k$, so that the modes are well defined and \eqref{eq:disc-spectral-choice} holds for every $n_{0,k}\ge N_k$. Section~\ref{subsec:disc-uniform} will use this freedom to impose an additional bound on $\sum_kK_k^{-1}$. This finishes the proof.
\end{proof}

We finally verify that the generated velocities retain a nonvanishing spacetime $L^2$ norm, as required in Section~\ref{sec:aos}. The frequency bound in \eqref{eq:disc-spectral-choice} and $E[u_k](t)=1$ give $\|u_k(t)\|_{L^2(\D)}\le2\sqrt2/K_k$ and $\|\partial_tu_k(t)\|_{L^2(\D)}\le\sqrt2$ for every $t\in[0,T]$. The wave equation and the energy identity imply
\begin{equation*}
\frac{\dd}{\dd t}\int_{\D}u_k\partial_tu_k\dd x
=\|\partial_tu_k\|_{L^2(\D)}^2-\|\nabla u_k\|_{L^2(\D)}^2
=2\|\partial_tu_k\|_{L^2(\D)}^2-2.
\end{equation*}
Integrating in time and using $\|r_k\|_{L^2((0,T)\times\D)}\to0$, we obtain
\begin{equation}\label{eq:disc-mass}
\begin{aligned}
\|\partial_tu_k\|_{L^2((0,T)\times\D)}^2
&=T+\frac12\left[\int_{\D}u_k(t,x)\partial_tu_k(t,x)\dd x\right]_{t=0}^{t=T}
=T+O(K_k^{-1}),\\
\|S_k\|_{L^2((0,T)\times\D)}^2&=T+o(1).
\end{aligned}
\end{equation}
For a profile with nontrivial matching, such as $q_0$, these objects supply the AOS of Section~\ref{sec:aos}.

\section{GCC realization on the disc}\label{sec:disc-class}
\enlargethispage{2pt}

Section~\ref{sec:disc} gave a construction of fixed observation regions $G$. In this section, we restrict to nonzero real finite-step antiperiodic profiles $q\in X$. Within this construction, we now choose the sector widths $a_k,b_k$ and scaling parameters $K_k$ so that $G$ satisfies \eqref{eq:gcc}.

The proof has three steps.
\begin{itemize}[leftmargin=*,itemsep=3pt]
\item \emph{Step 1.} Section~\ref{subsec:disc-rays} estimates the observed time for rays away from diameters.
\item \emph{Step 2.} Sections~\ref{subsec:disc-matching}--\ref{subsec:disc-pairing} estimate the observed time for nearly diametral rays.
\item \emph{Step 3.} Section~\ref{subsec:disc-uniform} chooses the parameters and combines these estimates to obtain a uniform lower bound for the observed time.
\end{itemize}

\subsection{The non-observable region satisfying GCC}\label{subsec:disc-gcc}

Let $q$ be a nonzero real finite-step profile with\footnote{At every partition boundary point choose an adjacent essential value, consistently with antiperiodicity.}
\begin{equation}\label{eq:disc-step}
\begin{gathered}
q(s)=c_j\quad\text{for every }s\in(s_{j-1},s_j),\\
0=s_0<s_1<\cdots<s_{J_q}=2\pi\qquad\text{and}\qquad q(s+\pi)=-q(s).
\end{gathered}
\end{equation}
The branch profiles in the first packet are $f_1=q$ and $f_2=\tau_{\pi/2}q$. Their signed sum is the packet expression $V_+$ in \eqref{eq:two-prof}. Thus the matching region remains
\begin{equation}\label{eq:disc-step-match}
\Gamma_k=\{(t,r,\theta):\theta\in\calA_k,\ q\bigl(K_k(t-r)\bigr)=q\bigl(K_k(t+r)+\pi/2\bigr)\}.
\end{equation}
We use the pairwise disjoint sectors in \eqref{eq:sectors} and construct $G$ by \eqref{eq:disc-G}, with the parameter dependence in \eqref{eq:disc-G-param}.

\begin{theorem}\label{thm:disc-gcc}
Let $q$ be a nonzero real profile as in \eqref{eq:disc-step}. Let $T_\D=64$.
There are sector centers $\psi_k$, widths $a_k,b_k$, Fourier cutoffs $P_k$, angular Fourier cutoffs $Q_k$, base radial indices $n_{0,k}$ and $C_k>0$ satisfying the hypotheses and parameter conditions of Proposition~\ref{prop:disc-real}, with $K_k=\pi(n_{0,k}-\tfrac14)$, such that the following conclusions hold for the same choices.

Define $S_k,\Gamma_k$ by Proposition~\ref{prop:disc-support}, define $G^{(N)}\downarrow G$ by Proposition~\ref{prop:disc-observation}, and take the exact waves $u_k$ from Proposition~\ref{prop:disc-real}.
\begin{enumerate}[label=(\roman*),leftmargin=*,itemsep=3pt]
\item \emph{(Non-observability).} By Proposition~\ref{prop:disc-real},
\begin{equation*}
E[u_k](0)=1\qquad\text{and}\qquad
\|\1_G\partial_tu_k\|_{L^2((0,T_\D)\times\D)}\longrightarrow0.
\end{equation*}
Thus observability fails on this fixed Borel set $G$.
\item \emph{(GCC).} Moreover, the same set $G\subset(0,T_\D)\times\D$ satisfies the essential spacetime measurable GCC: there is $\eta_q>0$ such that
\begin{equation*}
\essinf_{\rho\in S^*\D}\int_0^{T_\D}\1_G(t,\pi\Phi_t\rho)\dd t\ge\eta_q,
\end{equation*}
where the essential infimum is relative to Liouville measure.
\end{enumerate}
\end{theorem}

To see what will be estimated, fix a regular billiard ray $x(t)$.\footnote{For the essential condition \eqref{eq:gcc}, it suffices to consider regular billiard rays, which form a full-measure set in the interior phase space; see Sections~\ref{subsec:meas-flow}--\ref{subsec:meas-def}. Boundary-gliding rays are treated separately in Remark~\ref{rem:disc-completion}.} Away from the origin, write it in polar coordinates as
\begin{equation}\label{eq:disc-polar-ray}
x(t)=\bigl(r(t)\cos\theta(t),r(t)\sin\theta(t)\bigr).
\end{equation}
Recall its observed time $D_G=D_G(x(\cdot))$ from \eqref{eq:gcc} and split the time interval according to the angular sectors:
\begin{equation}\label{eq:disc-times}
\begin{cases}
T_A:=|\{t\in(0,T):\theta(t)\in\textstyle\bigcup_k\calA_k\}|,\\
T_B:=|\{t\in(0,T):\theta(t)\in\textstyle\bigcup_k\calB_k\}|,\\
T_{\rm out}:=T-T_A-T_B.
\end{cases}
\end{equation}
The definition of $G$ gives the exact observation-time identity
\begin{equation}\label{eq:disc-obs-time}
D_G=T_{\rm out}+\sum_k\int_{\{t\in(0,T):\theta(t)\in\calA_k\}}\1_{\Gamma_k}\bigl(t,x(t)\bigr)\dd t.
\end{equation}
Thus all time outside the sectors is observed. Inside $\calA_k$, only profile matching contributes; time in $\calB_k$ does not contribute directly.

\begin{figure}[htbp]
\centering
\includegraphics[width=\textwidth,
  trim={0mm 3mm 0mm 5mm},
  clip]{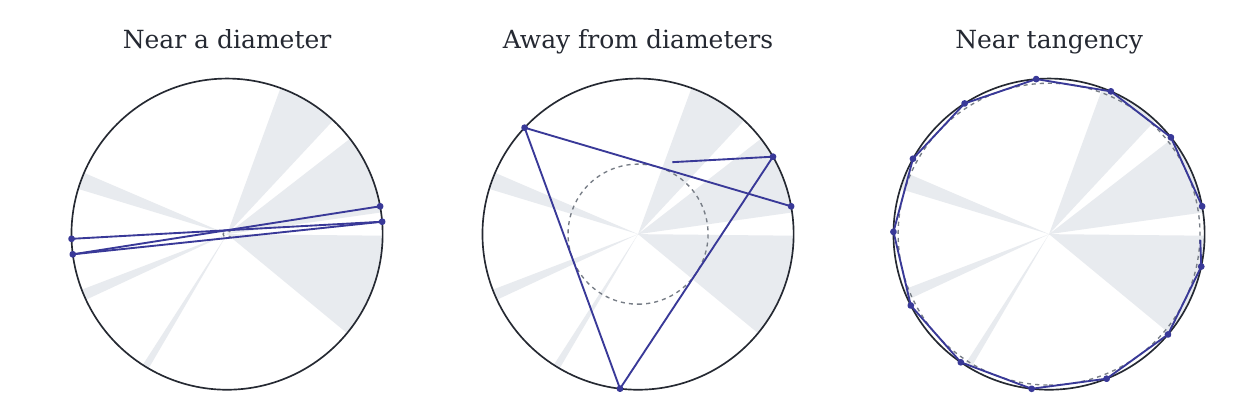}
\caption{Disc rays with $|L|=0.025,0.45,0.97$, each followed for six units of physical time. The dashed circles are their invariant caustics. Gray marks the angular sectors.}\label{fig:disc-gcc-overview}
\end{figure}

\subsection{Angular momentum and time spent in sectors}\label{subsec:disc-rays}

We estimate the time spent in angular sectors uniformly in the number of reflections.

\begin{defi}\label{def:disc-momentum}
For $(x,\xi)\in S^*\D$, the signed angular momentum is
\begin{equation}\label{eq:disc-momentum}
L:=x\wedge\xi=x_1\xi_2-x_2\xi_1.
\end{equation}
\end{defi}

Angular momentum is preserved by free motion and reflection. Indeed, along free motion, $\dot x=\xi$, $\dot\xi=0$ and
\begin{equation*}
\frac{\dd}{\dd t}(x\wedge\xi)=\xi\wedge\xi=0.
\end{equation*}
At a reflection point $|x|=1$, the incoming and outgoing velocities satisfy
\begin{equation*}
\xi^+=\xi^- -2(\xi^-\cdot x)x,
\qquad x\wedge\xi^+=x\wedge\xi^-.
\end{equation*}
For $0<|L|<1$, the distance from the line of each unit-speed ray segment to the origin is $|x\wedge\xi|=|L|$. Thus every complete segment\footnote{A complete ray segment joins two consecutive reflection points and is a chord of the disc. The time interval $[0,T]$ may contain only part of the initial and final segments.} is tangent to the same circle $r=|L|$, the invariant caustic; see Figure~\ref{fig:disc-gcc-overview}. When $L=0$, the ray follows a diameter.

\begin{lemma}\label{lem:disc-sector-time}
Fix $T>0$. There is a constant $C_T>0$, depending only on $T$, such that for every angular interval $I\subset\R/2\pi\Z$ and every regular billiard ray $x(t)$ with $0<|L|<1$,
\begin{equation}\label{eq:disc-interval-time}
|\{t\in(0,T):\theta(t)\in I\}|\le\frac{C_T|I|}{|L|}.
\end{equation}
\end{lemma}

\begin{proof}
We first bound the angular variation. On a ray segment between consecutive reflections, choose $y=0$ at the closest point to the origin, with $y$ increasing at unit speed. Along the segment,
\begin{equation}\label{eq:disc-ray}
\left\{\begin{aligned}
&r^2=L^2+y^2,\qquad -\sqrt{1-L^2}\le y\le\sqrt{1-L^2},\\
&\dot r=y/r,\qquad \dot\theta=L/r^2\qquad\text{for a.e. }t.
\end{aligned}\right.
\end{equation}
The travel time and the absolute polar-angle change along a complete segment are
\begin{equation*}
\tau(L)=2\sqrt{1-L^2}\qquad\text{and}\qquad
|\Delta\theta|=2\arccos|L|.
\end{equation*}
Their ratio satisfies
\begin{equation*}
\frac{|\Delta\theta|}{\tau(L)}
=\frac{\arccos|L|}{\sqrt{1-L^2}}\le\frac\pi2.
\end{equation*}
The complete segments in $[0,T]$ therefore contribute at most $\pi T/2$. The initial and final partial segments each contribute at most $\pi$. Hence a continuous lift of the angle satisfies
\begin{equation}\label{eq:ang-var}
\Var_{[0,T]}\theta\le\frac\pi2T+2\pi.
\end{equation}

Since $\dot\theta=L/r^2$ has a fixed sign, the lifted angle is monotone. Choose a representative interval $\widetilde I\subset\R$ of $I$. Each translate $\widetilde I+2\pi m$ is visited at most once, and \eqref{eq:ang-var} gives
\begin{equation*}
\#\{\text{visits to }I\}\le\frac{\Var_{[0,T]}\theta}{2\pi}+2
\le\frac T4+3.
\end{equation*}
On each visit, \eqref{eq:disc-ray} gives
\begin{equation*}
\left|\frac{\dd t}{\dd\theta}\right|=\frac{r^2}{|L|}\le\frac1{|L|}.
\end{equation*}
Adding the visits, we obtain
\begin{equation*}
|\{t\in(0,T):\theta(t)\in I\}|
\le C_T\int_I\frac{\dd\theta}{|L|}
=\frac{C_T|I|}{|L|}.
\end{equation*}
This proves \eqref{eq:disc-interval-time}.
\end{proof}

Recall that the sectors $\calA_k,\calB_k$ in \eqref{eq:sectors} are pairwise disjoint angular intervals with $0<a_k,b_k<\pi/4$, $|\calA_k|=2a_k$ and $|\calB_k|=2b_k$. Applying Lemma~\ref{lem:disc-sector-time} to each interval and summing, we obtain, for $0<|L|<1$,
\begin{equation}\label{eq:disc-sector-time}
T_A\le\frac{C_T}{|L|}\sum_k|\calA_k| \textrm{ \; and \; }
T_B\le\frac{C_T}{|L|}\sum_k|\calB_k|.
\end{equation}

\subsection{Symmetry matching time in $\calA_k$ near diameters}\label{subsec:disc-matching}

The bound in Lemma~\ref{lem:disc-sector-time} deteriorates as $L\to0$. We obtain a positive observed fraction of $T_A$ through profile matching.

\begin{prop}\label{prop:disc-matching}
Let $q$ be as in \eqref{eq:disc-step}. There are $L_0=L_0(q)\in(0,1/2]$ and $c=c(q)\in(0,1)$ such that, for every $T>0$, there is $C_1=C_1(T,q)>0$ with the following property.

For any pairwise disjoint sectors $\calA_k,\calB_k$ in \eqref{eq:sectors}, with $0<a_k,b_k<\pi/4$, and any $K_k>0$ satisfying $\sum_kK_k^{-1}<\infty$, define $G$ by \eqref{eq:disc-G}. Then every regular billiard ray $x(t)$ with $|L|\le L_0$ satisfies
\begin{equation}\label{eq:disc-obs-lower}
D_G(x(\cdot))\ge T_{\rm out}+cT_A-C_1\sqrt{|L|}-C_1\sum_kK_k^{-1}.
\end{equation}
\end{prop}

Recall from \eqref{eq:disc-step-match} and \eqref{eq:disc-G-rule} that, while $\theta(t)\in\calA_k$, the ray is observed precisely when
\begin{equation}\label{eq:disc-ray-match}
q\bigl(K_k(t-r(t))\bigr)=q\bigl(K_k(t+r(t))+\pi/2\bigr).
\end{equation}
We seek a lower bound for the time during which these two profile values agree. The inputs of $q$ are the two radial phases. Along the ray, they satisfy, for almost every $t$,
\begin{equation*}
\begin{cases}
\dfrac{\dd}{\dd t}\bigl[K_k(t-r(t))\bigr]=K_k(1-\dot r(t)),\\[2mm]
\dfrac{\dd}{\dd t}\bigl[K_k(t+r(t))+\pi/2\bigr]=K_k(1+\dot r(t)).
\end{cases}
\end{equation*}
On the nearly radial parts of a ray segment, $\dot r$ is close to $-1$ or $1$. Thus one phase moves at speed close to $2K_k$, while the other moves slowly. As long as the slow phase remains in one constancy interval of $q$, the slow-profile value is a constant $v$. Matching occurs whenever the fast profile also takes the value $v$.
\vspace{2mm}

For the finite essential range $\mathcal V(q)$, set
\begin{equation*}
L_*:=\min_{v\in\mathcal V(q)}\int_0^{2\pi}\1_{\{q(s)=v\}}\dd s>0.
\end{equation*}
The positivity follows from the finite-step assumption. Each complete fast period during which the slow-profile value remains fixed contributes at least $L_*$ of matching phase length. The following lemma converts this phase length into matching time.

\begin{lemma}[Fast and slow profile matching]\label{lem:fastslow}
Let $q$ be as in \eqref{eq:disc-step}. There are $\delta_q\in(0,1)$ and $C_q>0$ such that, for every $K>0$ and absolutely continuous monotone phases $\Theta_{\rm f},\Theta_{\rm s}:I\to\R$ on an interval $I$, the conditions
\begin{equation*}
K\le|\Theta_{\rm f}'|\le2K\qquad\text{and}\qquad |\Theta_{\rm s}'|\le\delta_q|\Theta_{\rm f}'|
\end{equation*}
imply
\begin{equation*}
|\{t\in I:q\bigl(\Theta_{\rm f}(t)\bigr)=q\bigl(\Theta_{\rm s}(t)+\alpha\bigr)\}|
\ge\delta_q|I|-C_q/K,
\end{equation*}
uniformly in $\alpha\in\R$.
\end{lemma}

\begin{proof}
We use $\Theta_{\rm f}$ as coordinate and divide its range into complete periods. On a time interval $J$ corresponding to one such period, suppose $q(\Theta_{\rm s}+\alpha)$ keeps the value $v$. Then
\begin{align*}
\int_J\1_{\{q(\Theta_{\rm f}(t))=v\}}\dd t
&=\int_{\Theta_{\rm f}(J)}\1_{\{q(\xi)=v\}}
\frac{\dd\xi}{|\Theta_{\rm f}'\bigl(\Theta_{\rm f}^{-1}(\xi)\bigr)|}\\
&\ge\frac1{2K}\int_0^{2\pi}\1_{\{q(\xi)=v\}}\dd\xi
\ge\frac{L_*}{2K}.
\end{align*}
Thus each complete period without a slow-profile change supplies a definite amount of matching time.

Write $V:=|\Theta_{\rm f}(I)|$. There are at least $V/(2\pi)-2$ complete periods. Since the slow phase is continuous and monotone, the inverse image of each jump level is an interval, a point, or the empty set. If it pauses at a jump, the slow profile can change value only when entering or leaving that level; throughout the pause it keeps the chosen essential value. Complete fast periods contained in the pause therefore contribute matching time as above.
\vspace{2mm}

There are only finitely many jump levels per period of $q$. Hence there is a counting constant $A_q>0$, depending only on $q$, such that the number $N_{\rm bad}$ of complete fast periods containing a slow-profile change satisfies
\begin{equation*}
N_{\rm bad}\le A_q(1+\Var_I\Theta_{\rm s})\le A_q(1+\delta_qV).
\end{equation*}
Discarding these periods, we obtain
\begin{equation*}
|\{t\in I:q\bigl(\Theta_{\rm f}(t)\bigr)=q\bigl(\Theta_{\rm s}(t)+\alpha\bigr)\}|
\ge\frac{L_*}{2K}\left(\frac V{2\pi}-2-A_q(1+\delta_qV)\right).
\end{equation*}

Choose $\delta_q\in(0,1)$ so that $A_q\delta_q\le1/(4\pi)$ and $\delta_q\le L_*/(8\pi)$. The preceding lower bound is at least
\begin{align*}
\frac{L_*}{2K}\left(\frac{V}{4\pi}-2-A_q\right)
&=\frac{L_*V}{8\pi K}-\frac{L_*(2+A_q)}{2K}\\
&\ge\delta_q|I|-\frac{C_q}K,
\end{align*}
where $C_q:=L_*(2+A_q)/2$ and we used $V\ge K|I|$. This proves the lemma.
\end{proof}

\begin{figure}[htbp]
\centering
\includegraphics[width=.9\textwidth,
  trim={0mm 3mm 0mm 2mm},
  clip]{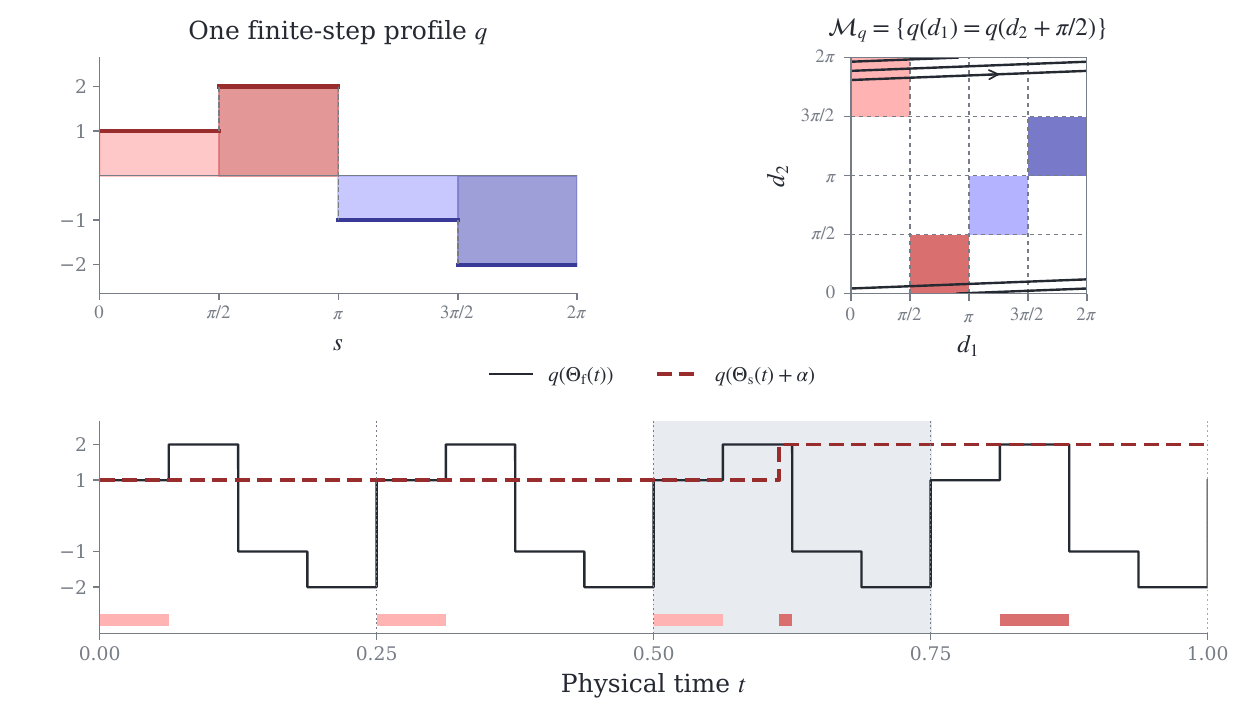}
\caption{Profile matching for a finite-step profile. The profile takes the values $1,2,-1,-2$ on consecutive quarter-periods. The phase trace is nearly horizontal in the reference variables, while the lower panel shows the corresponding profile values against physical time; here $\alpha=\pi/2$. Colored intervals mark equal values.}\label{fig:disc-fast-slow}
\end{figure}

\begin{proof}[Proof of Proposition~\ref{prop:disc-matching}]
Take $\delta_q$ from Lemma~\ref{lem:fastslow} and set
\begin{equation*}
L_0:=\min\{1/2,4\delta_q\},\qquad c=c(q):=\delta_q.
\end{equation*}
Suppose first that $0<|L|\le L_0$. We apply Lemma~\ref{lem:fastslow} to the matching condition \eqref{eq:disc-ray-match}. On the incoming and outgoing halves of a ray segment, respectively, the phases and shift in the lemma are
\begin{equation*}
(\Theta_{\rm f},\Theta_{\rm s},\alpha)=
\begin{cases}
\bigl(K_k(t-r),K_k(t+r),\pi/2\bigr),&y<0,\\
\bigl(K_k(t+r)+\pi/2,K_k(t-r),0\bigr),&y>0.
\end{cases}
\end{equation*}
Thus the incoming and outgoing halves exchange the roles of the two profiles. On either half-segment, \eqref{eq:disc-ray} gives
\begin{equation*}
|\Theta_{\rm f}'|=K_k(1+|\dot r|)\qquad\text{and}\qquad
|\Theta_{\rm s}'|=K_k(1-|\dot r|).
\end{equation*}
Thus $K_k\le|\Theta_{\rm f}'|\le2K_k$. The identity $r^2-y^2=L^2$ also gives
\begin{equation}\label{eq:speed-ratio}
\frac{|\Theta_{\rm s}'|}{|\Theta_{\rm f}'|}
=\frac{1-|\dot r|}{1+|\dot r|}
=\frac{L^2}{(r+|y|)^2}\le\frac{L^2}{4y^2}.
\end{equation}
For $0<|L|\le1/2$, omit $|y|<\sqrt{|L|}$ on each segment. On the remaining pieces, \eqref{eq:speed-ratio} gives
\begin{equation*}
\frac{|\Theta_{\rm s}'|}{|\Theta_{\rm f}'|}\le\frac{|L|}{4}\le\delta_q
\qquad\text{if }|L|\le4\delta_q.
\end{equation*}
This verifies the slow-phase hypothesis of Lemma~\ref{lem:fastslow} for $0<|L|\le L_0$, where $L_0$ depends only on $q$.
\vspace{2mm}

For $|L|\le1/2$, each complete segment has length
\begin{equation*}
\tau(L)=2\sqrt{1-L^2}\ge\sqrt3.
\end{equation*}
There are therefore at most $T/\sqrt3+2$ segments meeting $[0,T]$. Since $\dot y=1$, their omitted central pieces have total time at most
\begin{equation}\label{eq:disc-central-time}
2\sqrt{|L|}\left(T/\sqrt3+2\right)\le C_T\sqrt{|L|}.
\end{equation}

Cut each visit to $\calA_k$ at reflections and remove these central pieces. Denote the resulting intervals by $\mathcal J_k$. On each retained half-segment, the lifted angle is monotone and varies by less than $\pi$, so it meets at most two translates of a fixed angular interval. The segment count therefore gives
\begin{equation*}
\#\mathcal J_k\le4\left(T/\sqrt3+2\right)\le C_T,
\qquad\forall k\ge1.
\end{equation*}
Applying Lemma~\ref{lem:fastslow} on each $J\in\mathcal J_k$ and summing, we obtain
\begin{equation}\label{eq:disc-match-per-k}
\sum_{J\in\mathcal J_k}\int_J\1_{\Gamma_k}\bigl(t,x(t)\bigr)\dd t
\ge c\sum_{J\in\mathcal J_k}|J|-\frac{C_q\#\mathcal J_k}{K_k}
\ge c\sum_{J\in\mathcal J_k}|J|-\frac{C_{T,q}}{K_k}.
\end{equation}

Since the sectors are pairwise disjoint, their omitted central time is counted only once. By \eqref{eq:disc-central-time},
\begin{equation*}
\sum_k\sum_{J\in\mathcal J_k}|J|\ge T_A-C_T\sqrt{|L|}.
\end{equation*}
Sum \eqref{eq:disc-match-per-k} first over $k\le N$, and then let $N\to\infty$. Monotone convergence applies to the nonnegative time integrals and interval lengths, while the additive errors sum to $C_{T,q}\sum_kK_k^{-1}$. Choosing $C_1=C_1(T,q)$ to bound both error constants, we conclude that
\begin{equation}\label{eq:disc-count}
\sum_k\int_{\{t:\theta(t)\in\calA_k\}}\1_{\Gamma_k}\bigl(t,x(t)\bigr)\dd t
\ge cT_A-C_1\sqrt{|L|}-C_1\sum_kK_k^{-1}.
\end{equation}
The choice of $\delta_q$ in Lemma~\ref{lem:fastslow} absorbs the slow-phase variation into the positive fraction $c$. The series $C_1\sum_kK_k^{-1}$ sums the remaining additive errors over all cut intervals, including when a ray meets infinitely many sectors. When $L=0$, apply the lemma directly on the incoming and outgoing halves of each diameter: the fast phase has speed $2K_k$ and the slow phase is constant. The chosen values at jumps ensure that this constant slow-profile value is essential.

Combining \eqref{eq:disc-obs-time} and \eqref{eq:disc-count} proves \eqref{eq:disc-obs-lower}. This finishes the proof of Proposition~\ref{prop:disc-matching}.
\end{proof}

\subsection{Time spent in opposite sectors}\label{subsec:disc-pairing}

Proposition~\ref{prop:disc-matching} extracts an observed fraction of $T_A$. We now compare the time spent in $\calB_k$ with the time spent in $\calA_k$. On a complete ray segment, the points with coordinates $y$ and $-y$ have nearly opposite angles away from the center, and the map $y\mapsto-y$ preserves time. Choosing $\calB_k$ much narrower than $\calA_k$ lets us place the counterpart of most points of $\calB_k$ in $\calA_k$. This gives a bound for $T_B$ in terms of $T_A$, to which we will apply the preceding matching estimate.

For $a_*\in(0,1)$, consider widths satisfying
\begin{equation}\label{eq:disc-widths}
a_k=a_*2^{-k},\qquad 0<b_k\le a_k^4,\qquad\forall k\ge1.
\end{equation}
The following estimate is uniform in $a_*$. We will choose $a_*$ in Section~\ref{subsec:disc-uniform} together with the remaining parameters.
\begin{prop}\label{prop:disc-pairing}
Let $T>0$. There is $C_2=C_2(T)>0$ with the following property. For any $a_*\in(0,1)$ and any pairwise disjoint sectors $\calA_k,\calB_k$ in \eqref{eq:sectors} with widths satisfying \eqref{eq:disc-widths}, every regular billiard ray $x(\cdot)$ with $|L|\le1/2$ satisfies
\begin{equation}\label{eq:opposite-time}
T_B\le T_A+4+C_2\sqrt{|L|}.
\end{equation}
\end{prop}

\begin{proof}
Suppose first that $L\ne0$. We use time pairing on each complete ray segment. Let $t_*$ be its midpoint time, so that $y(t)=t-t_*$. The paired time is
\begin{equation*}
t'=2t_*-t,\qquad y(t')=-y(t),\qquad |\dd t'|=\dd t.
\end{equation*}
This map is a measure-preserving involution of the segment's time interval. The paired points have nearly opposite polar angles: the deviation of their angular difference from $\pi$ is
\begin{equation*}
\dist_{\R/2\pi\Z}\bigl(\theta(t)-\theta(t'),\pi\bigr)
=2\arctan(|L|/|y|)\le2|L|/|y|,\qquad y\ne0.
\end{equation*}
We absorb this angular error in the wider sectors $\calA_k$ and use Lemma~\ref{lem:disc-sector-time} for the smaller sectors.
\vspace{2mm}

For $|y|\ge\sqrt{|L|}$, the angular error is at most $2\sqrt{|L|}$. If $a_k\ge8\sqrt{|L|}$ and $\theta(t)\in\calB_k$, then
\begin{equation*}
\dist_{\R/2\pi\Z}\bigl(\theta(t'),\psi_k\bigr)
\le b_k+2\sqrt{|L|}\le a_k/2+a_k/4<a_k.
\end{equation*}
Here $b_k\le a_k/2$ follows from $a_k<1/2$ and $b_k\le a_k^4$. Hence the paired point lies in $\calA_k$. On each complete segment, the same map $t\mapsto t'$ is used for all sector pairs. It is one-to-one, and different complete segments have disjoint time intervals. Thus, on the union of all retained times with $\theta(t)\in\calB_k$ and $a_k\ge8\sqrt{|L|}$, the pairing is a measure-preserving injection into the times counted by $T_A$. Their total duration is therefore at most $T_A$.

By \eqref{eq:disc-central-time}, the central pieces omitted from all complete segments have total duration at most $C_T\sqrt{|L|}$.

For the remaining sectors, let $k_0$ be the first index with $a_{k_0}<8\sqrt{|L|}$. Since $a_k=a_*2^{-k}$, geometric summation gives
\begin{equation*}
\sum_{a_k<8\sqrt{|L|}}b_k\le\sum_{k\ge k_0}a_k^4=\frac{a_{k_0}^4}{1-2^{-4}}\le\frac{(8\sqrt{|L|})^4}{1-2^{-4}}\le C|L|^2.
\end{equation*}
This estimate also covers $k_0=1$, with a constant independent of $a_*$. Applying \eqref{eq:disc-sector-time}, we obtain
\begin{equation*}
\bigl|\{t:\theta(t)\in\textstyle\bigcup_{a_k<8\sqrt{|L|}}\calB_k\}\bigr|
\le\frac{2C_T}{|L|}\sum_{a_k<8\sqrt{|L|}}b_k
\le C_T|L|.
\end{equation*}

The initial and final incomplete segments each last at most $2$. Combining the paired time and the three losses, we obtain
\begin{align*}
T_B&\le T_A+\underbrace{C_T\sqrt{|L|}}_{\text{central pieces}}
+\underbrace{C_T|L|}_{\text{small sectors}}
+\underbrace{4}_{\text{incomplete segments}}\\
&\le T_A+4+C_2\sqrt{|L|}.
\end{align*}
This proves \eqref{eq:opposite-time} for $L\ne0$. When $L=0$, the paired points on a complete diameter have exactly opposite angles, apart from the single time at the origin. Since $b_k<a_k$, every paired point from $\calB_k$ lies in $\calA_k$, and only the incomplete segments are lost. Thus $T_B\le T_A+4$, which proves the remaining case.
\end{proof}

\begin{figure}[tbp]
\centering
\includegraphics[width=.9\textwidth]{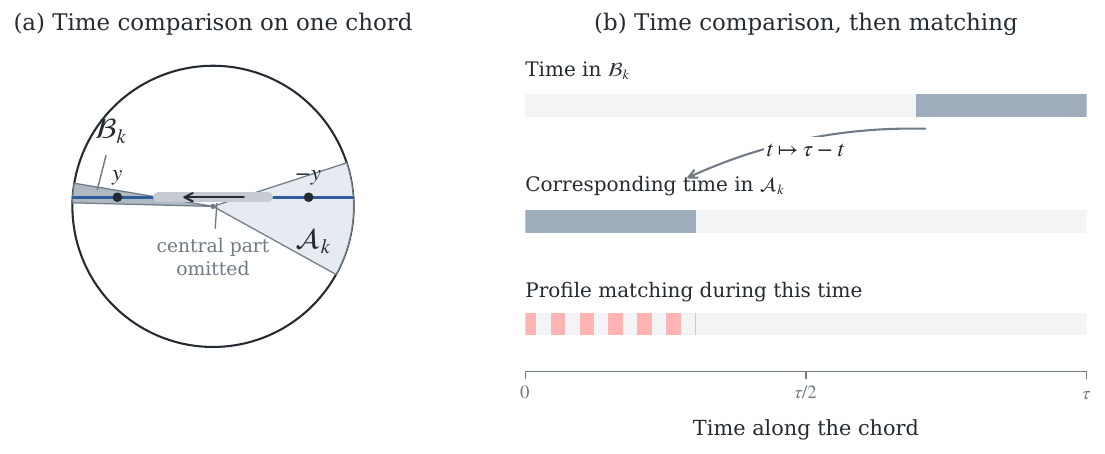}
\caption{Time comparison and profile matching on a nearly diametral ray segment. The map $y\mapsto-y$ relates the first two time strips; the colored part of the third strip is matching for $q_0$ at an illustrative radial scale.}\label{fig:disc-pairing}
\end{figure}

\subsection{Choosing the parameters and obtaining a uniform bound}\label{subsec:disc-uniform}

We combine the estimates of Sections~\ref{subsec:disc-rays}--\ref{subsec:disc-pairing} to obtain a uniform observed-time bound. The parameters should also satisfy Proposition~\ref{prop:disc-real}, so that GCC and non-observability hold on the same fixed set $G$. For $T=T_\D$, we choose one threshold $\ell_\varepsilon\in(0,L_0)$ and distinguish two ranges:
\begin{itemize}[leftmargin=*,itemsep=3pt]
\item \emph{Near diameters: $|L|<\ell_\varepsilon$.} We choose $\ell_\varepsilon$ so that $C_2\sqrt{\ell_\varepsilon}\le1$. Proposition~\ref{prop:disc-pairing} then gives $T_B\le T_A+5$, hence $2T_A+T_{\rm out}\ge T-5$. Proposition~\ref{prop:disc-matching} supplies the observed fraction $cT_A$; we make its errors small by the same choice of $\ell_\varepsilon$ and a small sum $\sum_kK_k^{-1}$.
\item \emph{Away from diameters: $|L|\ge\ell_\varepsilon$.} With $\ell_\varepsilon$ fixed, we choose $a_*$ small enough that the widths $a_k,b_k$ in \eqref{eq:disc-widths} give $T_A+T_B\le T/4$ by \eqref{eq:disc-sector-time}. Hence $D_G\ge T_{\rm out}\ge3T/4$.
\end{itemize}

\begin{proof}[Proof of Theorem~\ref{thm:disc-gcc}]
Fix $T=T_\D=64$. Recall $c\in(0,1)$ and $C_1=C_1(T,q)$ from Proposition~\ref{prop:disc-matching}, and $C_2=C_2(T)$ from Proposition~\ref{prop:disc-pairing}. Equations \eqref{eq:disc-times} and \eqref{eq:opposite-time} imply, for $|L|\le1/2$,
\begin{equation*}
T\le2T_A+T_{\rm out}+4+C_2\sqrt{|L|}.
\end{equation*}
Thus $T_A$ and $T_{\rm out}$ cannot both be small when $T$ is large. The positive terms in \eqref{eq:disc-obs-lower} satisfy
\begin{align*}
T_{\rm out}+cT_A
&\ge\frac c2(T_{\rm out}+2T_A)\\
&\ge\frac c2\bigl(T-4-C_2\sqrt{|L|}\bigr).
\end{align*}

We first choose the angular-momentum threshold $\ell_\varepsilon$ and the geometric parameters $a_*,a_k,b_k,\psi_k$. We then choose the spectral cutoffs $P_k,Q_k$, the normalization $C_k$ and the radial indices $n_{0,k}$ according to \eqref{eq:disc-real-params}, with the additional requirement that $\sum_kK_k^{-1}$ be small.

\noindent\emph{Step 1. Choice of $\ell_\varepsilon,a_*,a_k,b_k,\psi_k$.}
Choose $\ell_\varepsilon\in(0,L_0)$ small enough that
\begin{equation}\label{eq:disc-threshold}
C_1\sqrt{\ell_\varepsilon}\le cT/8\qquad\text{and}\qquad C_2\sqrt{\ell_\varepsilon}\le1.
\end{equation}
For the widths $a_k,b_k$ in \eqref{eq:disc-widths}, we have
\begin{equation*}
\sum_ka_k=a_*\qquad\text{and}\qquad
\sum_kb_k\le\sum_ka_k^4=\frac{a_*^4}{15}.
\end{equation*}
Recall $C_T$ from Lemma~\ref{lem:disc-sector-time}. Choose $a_*\in(0,\pi/4)$ small enough that
\begin{equation}\label{eq:disc-width-choice}
\frac{2C_T}{\ell_\varepsilon}\left(a_*+\frac{a_*^4}{15}\right)\le\frac T4.
\end{equation}
Fix $a_k=a_*2^{-k}$ and any $0<b_k\le a_k^4$. The factor $2$ in \eqref{eq:disc-width-choice} accounts for $|\calA_k|=2a_k$ and $|\calB_k|=2b_k$. Thus \eqref{eq:disc-sector-time} gives
\begin{equation*}
T_A+T_B\le\frac{2C_T}{|L|}\sum_k(a_k+b_k)
\le\frac{2C_T}{\ell_\varepsilon}\left(a_*+\frac{a_*^4}{15}\right)
\le\frac T4,
\qquad |L|\ge\ell_\varepsilon.
\end{equation*}
Since $\sum_k|\calA_k|=2a_*<\pi/2$, choose centers $\psi_k$ so that all $\calA_k$ are pairwise disjoint in one arc of length less than $\pi/2$. The opposite sectors $\calB_k$ are then also disjoint, since $b_k<a_k$, and lie in the opposite arc. This gives the sector placement described in Section~\ref{sec:disc}. The same $\ell_\varepsilon$ gives $T_B\le T_A+5$ whenever $|L|<\ell_\varepsilon$ by Proposition~\ref{prop:disc-pairing}.

\vspace{2mm}
\noindent\emph{Step 2. Choice of $P_k,Q_k,C_k,n_{0,k},K_k$ and verification of Proposition~\ref{prop:disc-real}.}
Fix these sectors. They are pairwise disjoint and satisfy
\begin{equation*}
0<a_k,b_k<\pi/4\qquad\text{and}\qquad a_k,b_k\longrightarrow0.
\end{equation*}
Thus they satisfy the geometric hypotheses of Proposition~\ref{prop:disc-real}. We now choose the remaining parameters in \eqref{eq:disc-real-params}.
\vspace{2mm}

Fix a positive odd $p_*$ with $\widehat q_{p_*}\ne0$. Choose integers $P_k\ge p_*$ with $P_k\to\infty$, and then angular cutoffs $Q_k$, such that
\begin{equation}\label{eq:disc-cutoff-choice}
\|q^{(P_k)}-q\|_{L^2(0,2\pi)}\le2^{-k} \textrm{ \; and \; }
 Q_k\min\{a_k,b_k\}\ge2^k.
\end{equation}
Set $C_k$ by the normalization in \eqref{eq:disc-real-params}. These choices provide the profile and angular approximation required there.
\vspace{2mm}

For each fixed $P_k,Q_k,C_k$ and $\psi_k$, take the threshold $N_k$ in \eqref{eq:disc-real-params}. The modes and the bounds \eqref{eq:disc-spectral-choice} are valid for every radial index beyond this threshold. We may therefore choose $n_{0,k}$ so that simultaneously
\begin{equation}\label{eq:disc-radial-choice}
n_{0,k}\ge\max\{k,N_k\},\qquad
K_k=\pi(n_{0,k}-\tfrac14), \textrm{ \; and \; }
C_1K_k^{-1}\le\frac{cT}{16}\,2^{-k}.
\end{equation}
Consequently,
\begin{equation}\label{eq:disc-freq-loss}
C_1\sum_kK_k^{-1}\le cT/16.
\end{equation}

Use these same parameters in Propositions~\ref{prop:disc-support} and \ref{prop:disc-observation} to define $S_k,\Gamma_k$ and $G^{(N)}\downarrow G$. Proposition~\ref{prop:disc-real} supplies exact waves satisfying
\begin{equation*}
E[u_k](0)=1\qquad\text{and}\qquad
\|\1_G\partial_tu_k\|_{L^2((0,T)\times\D)}
\le\|r_k\|_{L^2((0,T)\times\D)}\longrightarrow0.
\end{equation*}
Thus non-observability already holds on this fixed set $G$.

\vspace{2mm}
\noindent\emph{Step 3. Verification of GCC.}

\noindent\emph{Case 1. $|L|<\ell_\varepsilon$.} Proposition~\ref{prop:disc-pairing} and \eqref{eq:disc-threshold} give $2T_A+T_{\rm out}\ge T-5$. Combining Proposition~\ref{prop:disc-matching} with \eqref{eq:disc-threshold} and \eqref{eq:disc-freq-loss}, we obtain
\begin{align*}
D_G&\ge T_{\rm out}+cT_A-C_1\sqrt{\ell_\varepsilon}-C_1\sum_kK_k^{-1}\\
&\ge\frac c2(2T_A+T_{\rm out})-cT/8-cT/16\\
&\ge c\left(\frac{5T}{16}-\frac52\right)\ge cT/4,
\end{align*}
where the last inequality uses $T=64$.
\vspace{2mm}

\noindent\emph{Case 2. $|L|\ge\ell_\varepsilon$.} Equations \eqref{eq:disc-obs-time} and \eqref{eq:disc-width-choice} give
\begin{equation*}
D_G\ge T_{\rm out}\ge3T/4.
\end{equation*}
This estimate is uniform up to tangency because \eqref{eq:disc-sector-time} uses total angular variation.

Thus the same constant $\eta_q:=cT/4>0$ works in both ranges. In particular,
\begin{equation*}
\essinf_{\rho\in S^*\D}D_G(\rho)\ge\eta_q.
\end{equation*}
The essential infimum is relative to Liouville measure; Appendix~\ref{sec:meas} ensures that this statement is unchanged by modifying $G$ on a spacetime null set. This proves \eqref{eq:gcc}. Together with the vanishing observation established above, it finishes the proof.
\end{proof}

Taking $q=q_0$ proves Theorem~\ref{thm:disc}. The radial disc construction is now complete: the fixed profile symmetry is realized by exact waves asymptotically, and its fixed observation region satisfies \hyperref[eq:gcc]{GCC}.

\begin{remark}[Boundary completion and every generalized ray]\label{rem:disc-completion}
The disc counterexample does not rely on omitting boundary-gliding rays. Keep the Borel set $G$ in \eqref{eq:disc-G}, with the values of $q$ at jumps fixed as above, and set
\begin{equation*}
G^\sharp:=G\cup\bigl((0,T_\D)\times\partial\D\bigr).
\end{equation*}
The estimates in the proof give $D_G\ge\eta_q$ for every regular ray. In particular, \eqref{eq:disc-count} and \eqref{eq:opposite-time} include diameters; a constant slow phase at a profile jump is covered by the chosen jump value.
The open angular sectors already leave their side edges observed. The same estimates allow a transverse initial point on the circle, since the initial and final incomplete segments were included.

All glancing points of the disc are of second-order gliding type, since its boundary curvature is everywhere strictly positive. Thus a unit-speed generalized ray either follows these reflected segments or glides along the circle. The latter spends the whole time $T_\D$ in $G^\sharp$. Hence, writing $x_\gamma$ for its spatial projection, we obtain
\begin{equation*}
\inf_{\gamma\ \mathrm{generalized}}\int_0^{T_\D}
\1_{G^\sharp}\bigl(t,x_\gamma(t)\bigr)\dd t
\ge\min\{\eta_q,T_\D\}>0.
\end{equation*}
Since $G^\sharp\cap\bigl((0,T_\D)\times\D\bigr)=G$, the interior wave observation is unchanged and the same normalized waves still violate observability.

However, the completed set is not relatively open in $(0,T_\D)\times\overline\D$: adding the circle does not add an interior neighborhood; see Section~\ref{subsec:meas-classical}.
\end{remark}

\appendix
\numberwithin{figure}{section}
\section[Supplementary discussion on AOS]{Supplementary discussion on AOS: compatible profiles, velocity approximation, and localization}\label{sec:supp-aos}

In this appendix, we discuss three aspects of AOS:
\begin{itemize}[leftmargin=*,itemsep=2pt]
\item velocity approximation by profile-generated functions;
\item strict localization of the generated velocities;
\item the compatibility relations between branch profiles imposed by propagation.
\end{itemize}
The geometric-optics and energy arguments below are standard tools, organized here around these questions.

Sections~\ref{subsec:supp-free}--\ref{subsec:supp-local} treat velocity approximation and localization. Sections~\ref{subsec:supp-geometry}--\ref{subsec:supp-compat} connect boundary reflection and caustics with compatible profiles.

Throughout this appendix, we write
\begin{equation*}
\Box:=\partial_t^2-\Delta,\qquad Q_T:=(0,T)\times\Omega.
\end{equation*}
Here $\Omega=\R^d$ or a smooth bounded Euclidean domain, and the bounded-domain wave equation has the homogeneous Dirichlet condition. We use $\hbar>0$ as a small wavelength parameter.

\Needspace{18\baselineskip}
\subsection{Velocity approximation without a boundary}\label{subsec:supp-free}

A fixed periodic profile can be carried by a localized amplitude moving in its propagation direction. The following proposition gives both the approximation error and the relation between the two scales.

\begin{prop}\label{prop:supp-free}
Fix $T>0$, $\omega\in\R^d$ with $|\omega|=1$, and a nonzero real function $A\in C_c^\infty(\R^d)$. Let $q\in L^2(\R/2\pi\Z;\R)$ be nonzero and have zero mean. For $\hbar,\lambda\in(0,1]$, define
\begin{equation}\label{eq:supp-free-S}
\begin{split}
A_\lambda(x)&:=\lambda^{-d/2}A(x/\lambda),\\
S_{\hbar,\lambda}(t,x)&:=A_\lambda(x-t\omega)q\left(\frac{x\cdot\omega-t}{\hbar}\right).
\end{split}
\end{equation}
There is an exact real solution $u_{\hbar,\lambda}$ of $\Box u=0$ on $\R^d$ such that
\begin{equation}\label{eq:supp-free-err}
\begin{gathered}
\|\partial_tu_{\hbar,\lambda}-S_{\hbar,\lambda}\|_{L^2((0,T)\times\R^d)}
\le C_{T,A,q}\left(\frac{\hbar}{\lambda}+\frac{\hbar}{\lambda^2}\right), \;\;
\forall\hbar,\lambda\in(0,1].
\end{gathered}
\end{equation}
Moreover, the generated velocity is strictly localized:
\begin{equation}\label{eq:supp-free-support}
\supp S_{\hbar,\lambda}(t,\cdot)\subseteq
\bigl\{x=t\omega+\lambda y:\ y\in\supp A\bigr\},
\quad\forall t\in[0,T].
\end{equation}
As $\hbar/\lambda\to0$, the initial energies satisfy
\begin{equation}\label{eq:supp-free-ene}
E[u_{\hbar,\lambda}](0)\longrightarrow
\|A\|_{L^2(\R^d)}^2\frac1{2\pi}\int_0^{2\pi}|q(\theta)|^2\dd\theta>0.
\end{equation}
Consequently, if $\lambda\to0$ and $\hbar/\lambda^2\to0$, the velocities are approximated while their containing regions shrink.
\end{prop}

\begin{proof}
We first construct a primitive of the generated velocity, then correct its wave-equation residual. Let $F\in H^1(\R/2\pi\Z)$ be the zero-mean periodic primitive of $q$. The one-dimensional Sobolev inequality gives
\begin{equation*}
F'=q\qquad\text{and}\qquad
\|F\|_{L^\infty(0,2\pi)}\le C\|q\|_{L^2(0,2\pi)}.
\end{equation*}
Set
\begin{equation}\label{eq:supp-free-v}
v_{\hbar,\lambda}(t,x):=-\hbar A_\lambda(x-t\omega)F\left(\frac{x\cdot\omega-t}{\hbar}\right).
\end{equation}
The minus sign compensates for the time derivative of the phase.
\vspace{2mm}

\noindent\textit{Step 1. Velocity and residual.}
Write $D_\omega:=\omega\cdot\nabla$ and $\Delta_\perp:=\Delta-D_\omega^2$. Differentiating \eqref{eq:supp-free-v}, we obtain
\begin{equation}\label{eq:supp-free-diff}
\partial_tv_{\hbar,\lambda}-S_{\hbar,\lambda}
=\hbar(D_\omega A_\lambda)(x-t\omega)F\left(\frac{x\cdot\omega-t}{\hbar}\right).
\end{equation}
Both factors in \eqref{eq:supp-free-v} are functions of $x-t\omega$. Thus $\Box$ acts on their product as $-\Delta_\perp$. Since $F\bigl((x\cdot\omega-t)/\hbar\bigr)$ has no transverse variation, this gives
\begin{equation}\label{eq:supp-free-res}
\Box v_{\hbar,\lambda}
=\hbar(\Delta_\perp A_\lambda)(x-t\omega)F\left(\frac{x\cdot\omega-t}{\hbar}\right).
\end{equation}
In particular, derivatives of the rapidly oscillating profile do not appear in this residual. This cancellation uses the relation between the phase and the motion of the amplitude.

The identities hold distributionally for $q\in L^2$. Indeed, approximate $q$ in periodic $L^2$ by smooth zero-mean profiles. Their zero-mean primitives converge uniformly, and their compositions with this linear phase converge in $L^2$ on bounded sets. We may therefore pass to the limit in both identities.
\vspace{2mm}

\noindent\textit{Step 2. Correction to an exact wave.}
Let $u_{\hbar,\lambda}$ be the exact wave with Cauchy data $v_{\hbar,\lambda}[0]$. Its difference from $v_{\hbar,\lambda}$ has zero Cauchy data and forcing $-\Box v_{\hbar,\lambda}$. The elementary energy estimate, proved in Lemma~\ref{lem:supp-corr} below, implies
\begin{equation*}
\|\partial_tu_{\hbar,\lambda}-\partial_tv_{\hbar,\lambda}\|_{L^2((0,T)\times\R^d)}
\le\sqrt T\|\Box v_{\hbar,\lambda}\|_{L^1(0,T;L^2(\R^d))}.
\end{equation*}
Using \eqref{eq:supp-free-diff}--\eqref{eq:supp-free-res} and the scaling of $A_\lambda$, we obtain the more explicit estimate
\begin{equation*}
\begin{split}
\|\partial_tu_{\hbar,\lambda}-S_{\hbar,\lambda}\|_{L^2((0,T)\times\R^d)}
\le\hbar\|F\|_\infty\left(
\frac{\sqrt T}{\lambda}\|D_\omega A\|_2
+\frac{T^{3/2}}{\lambda^2}\|\Delta_\perp A\|_2\right).
\end{split}
\end{equation*}
This proves \eqref{eq:supp-free-err}. The support inclusion \eqref{eq:supp-free-support} follows directly from the amplitude in \eqref{eq:supp-free-S}.
\vspace{2mm}

\noindent\textit{Step 3. Nondegenerate energy.}
At time zero, the leading terms in the Cauchy data are
\begin{equation*}
\partial_tv_{\hbar,\lambda}(0,x)=A_\lambda(x)q(x\cdot\omega/\hbar)+O_{L^2}(\hbar/\lambda),
\end{equation*}
and
\begin{equation*}
\nabla v_{\hbar,\lambda}(0,x)=-\omega A_\lambda(x)q(x\cdot\omega/\hbar)+O_{L^2}(\hbar/\lambda).
\end{equation*}
To average the squared leading term, put $\varepsilon:=\hbar/\lambda$. After $x=\lambda y$ and integration over $\omega^\perp$, it becomes
\begin{equation*}
\int_\R W(s)|q(s/\varepsilon)|^2\dd s \textrm{ \; with \; }
\qquad
W(s):=\int_{\omega^\perp}|A(s\omega+y')|^2\dd y'.
\end{equation*}
For $d=1$, this means $W(s)=|A(s\omega)|^2$. In every dimension, $W$ is smooth and compactly supported. Let $m_q:=(2\pi)^{-1}\int_0^{2\pi}|q|^2$. The periodic $L^1$ function $|q|^2-m_q$ has a bounded periodic primitive $G$. Integration by parts gives
\begin{equation*}
\int_\R W(s)\bigl(|q(s/\varepsilon)|^2-m_q\bigr)\dd s
=-\varepsilon\int_\R W'(s)G(s/\varepsilon)\dd s\longrightarrow0.
\end{equation*}
This also bounds the leading $L^2$ terms uniformly as $\varepsilon\to0$. Since $u_{\hbar,\lambda}[0]=v_{\hbar,\lambda}[0]$, combining the two Cauchy-data expansions proves \eqref{eq:supp-free-ene}. This finishes the proof of Proposition~\ref{prop:supp-free}.
\end{proof}

The longitudinal period in this example is $2\pi\hbar$. The amplitude $A$ determines the shape of the containing region, $\lambda$ determines its size, and $\omega$ determines its motion. Localization of the profile $q$ can make the actual support even smaller. The condition $\hbar/\lambda^2\to0$ is sufficient for this construction. In dimension one, $\Delta_\perp=0$, and the residual in \eqref{eq:supp-free-res} vanishes identically.

\subsection{Velocity approximation with a general phase}\label{subsec:supp-phase}

The preceding construction uses the linear oscillatory variable $(x\cdot\omega-t)/\hbar$. A smooth nonlinear phase gives the same approximation mechanism when the phase follows the wave equation's characteristic geometry and the amplitude is transported accordingly.

\begin{prop}\label{prop:supp-phase}
Fix $T>0$ and a nonzero real zero-mean $q\in L^2(\R/2\pi\Z)$. Let $\calA$ and $\varphi$ be smooth real functions on $[0,T]\times\Omega$. Assume that the spatial support of $\calA$ is contained in one compact set $K\Subset\Omega$, uniformly in time. On a neighborhood of the spacetime support of $\calA$, suppose that
\begin{equation}\label{eq:supp-eik}
(\partial_t\varphi)^2=|\nabla\varphi|^2
\qquad\text{and}\qquad |\partial_t\varphi|\ge c>0,
\end{equation}
and that
\begin{equation}\label{eq:supp-trans}
2\bigl(\partial_t\varphi\,\partial_t\calA-\nabla\varphi\cdot\nabla\calA\bigr)+\calA\Box\varphi=0.
\end{equation}
Define
\begin{equation}\label{eq:supp-phase-S}
\begin{aligned}
A(t,x)&:=\calA(t,x)\,\partial_t\varphi(t,x),\\
S_\hbar(t,x)&:=A(t,x)\,q\left(\frac{\varphi(t,x)}{\hbar}\right).
\end{aligned}
\end{equation}
There are exact real energy solutions $u_\hbar$ of $\Box u_\hbar=0$, with the Dirichlet condition when $\partial\Omega\ne\varnothing$, such that
\begin{equation}\label{eq:supp-phase-err}
\|\partial_tu_\hbar-S_\hbar\|_{L^2(Q_T)}\le C\hbar,
\quad\forall\hbar\in(0,1].
\end{equation}
The constant depends on the fixed data but not on $\hbar$. If $A(0,\cdot)\not\equiv0$, the waves can be chosen so that
\begin{equation}\label{eq:supp-phase-ene}
E[u_\hbar](0)=
\frac1{2\pi}\int_0^{2\pi}|q(\theta)|^2\dd\theta
\int_\Omega|A(0,x)|^2\dd x+O(\hbar),
\quad\hbar\to0^+.
\end{equation}
In particular, their initial energies have a positive limit.
\end{prop}

The eikonal equation \eqref{eq:supp-eik} cancels the highest-order oscillatory term. The transport equation \eqref{eq:supp-trans} cancels the next term. The amplitude of the approximate wave is $\calA$, whereas the amplitude of its leading velocity is $A=\calA\partial_t\varphi$.

\begin{lemma}\label{lem:supp-corr}
Let $v\in C([0,T];H_0^1(\Omega))\cap C^1([0,T];L^2(\Omega))$ have distributional residual $R:=\Box v\in L^1(0,T;L^2(\Omega))$. For $\Omega=\R^d$, interpret $H_0^1(\Omega)$ as $H^1(\R^d)$. Let $u$ be the exact wave with Cauchy data $u[0]=v[0]$ and the same homogeneous boundary condition. Then, for every $S\in L^2(Q_T)$,
\begin{equation}\label{eq:supp-corr}
\|\partial_tu-S\|_{L^2(Q_T)}
\le\|\partial_tv-S\|_{L^2(Q_T)}
+\sqrt T\|R\|_{L^1(0,T;L^2(\Omega))}.
\end{equation}
\end{lemma}

\begin{proof}
Set $e:=u-v$. Then $\Box e=-R$, $e[0]=0$, and $e$ has zero boundary trace. Multiplication by $\partial_te$ and integration by parts give
\begin{equation*}
\frac{\dd}{\dd t}\sqrt{2E[e](t)}\le\|R(t)\|_{L^2(\Omega)}.
\end{equation*}
Hence
\begin{equation*}
\sup_{t\in[0,T]}\|\partial_te(t)\|_{L^2(\Omega)}
\le\int_0^T\|R(t)\|_{L^2(\Omega)}\dd t.
\end{equation*}
The triangle inequality proves \eqref{eq:supp-corr}. The argument for energy-class functions follows by approximation.
\end{proof}

\begin{proof}[Proof of Proposition~\ref{prop:supp-phase}]
Let $F$ be the zero-mean periodic primitive of $q$. We construct
\begin{equation}\label{eq:supp-phase-v}
v_\hbar(t,x):=\hbar\calA(t,x)F\left(\frac{\varphi(t,x)}{\hbar}\right).
\end{equation}
For a smooth profile, direct differentiation gives
\begin{equation}\label{eq:supp-phase-full}
\begin{split}
\Box v_\hbar={}&\hbar(\Box\calA)F(\varphi/\hbar)\\
&+\bigl[2(\varphi_t\calA_t-\nabla\varphi\cdot\nabla\calA)+\calA\Box\varphi\bigr]q(\varphi/\hbar)\\
&+\hbar^{-1}\calA\bigl[(\varphi_t)^2-|\nabla\varphi|^2\bigr]q'(\varphi/\hbar).
\end{split}
\end{equation}
The eikonal and transport equations cancel the last two lines. Thus
\begin{equation}\label{eq:supp-phase-res}
\begin{aligned}
\partial_tv_\hbar-S_\hbar&=\hbar\calA_tF(\varphi/\hbar),\\
\Box v_\hbar&=\hbar(\Box\calA)F(\varphi/\hbar).
\end{aligned}
\end{equation}
These identities also hold for $q\in L^2$: since $|\nabla\varphi|\ge c$ on the amplitude support, local phase coordinates allow smooth periodic approximation in $L^2$, while the corresponding primitives converge uniformly.
\vspace{2mm}

The amplitude vanishes near the boundary, so $v_\hbar$ has zero Dirichlet trace. Let $u_\hbar$ be the exact wave with Cauchy data $v_\hbar[0]$. The energy correction in Lemma~\ref{lem:supp-corr} gives
\begin{equation}\label{eq:supp-phase-bound}
\begin{split}
\|\partial_tu_\hbar-S_\hbar\|_{L^2(Q_T)}
\le\hbar\|F\|_\infty\Bigl(
\|\calA_t\|_{L^2(Q_T)} +\sqrt T\|\Box\calA\|_{L^1(0,T;L^2(\Omega))}\Bigr).
\end{split}
\end{equation}
This proves \eqref{eq:supp-phase-err}.

For the energy expansion, write $\Phi(x):=\varphi(0,x)$ and
\begin{equation*}
m_q:=\frac1{2\pi}\int_0^{2\pi}|q|^2 \textrm{ \; and \; }
Y:=\frac{\nabla\Phi}{|\nabla\Phi|^2}.
\end{equation*}
Let $G$ be a bounded periodic primitive of $|q|^2-m_q$. For a smooth compactly supported weight $a$ in the region where $\nabla\Phi\ne0$, integration by parts gives
\begin{align*}
\int_\Omega a(x)\Bigl(|q\bigl(\Phi(x)/\hbar\bigr)|^2-m_q\Bigr)\dd x
&=-\hbar\int_\Omega\operatorname{div}(aY)G(\Phi/\hbar)\dd x\\
&=O(\hbar).
\end{align*}
At time zero, \eqref{eq:supp-phase-v} also gives
\begin{align*}
\partial_tv_\hbar(0,\cdot)&=A(0,\cdot)q(\Phi/\hbar)+O_{L^2(\Omega)}(\hbar),\\
\nabla v_\hbar(0,\cdot)&=\calA(0,\cdot)\nabla\Phi\,q(\Phi/\hbar)+O_{L^2(\Omega)}(\hbar).
\end{align*}
Using the eikonal equation and then the averaging identity with $a=|A(0,\cdot)|^2$, we obtain
\begin{align*}
E[u_\hbar](0)
&=\int_\Omega|A(0,x)|^2|q\bigl(\Phi(x)/\hbar\bigr)|^2\dd x+O(\hbar)\\
&=m_q\|A(0,\cdot)\|_{L^2(\Omega)}^2+O(\hbar).
\end{align*}
This proves \eqref{eq:supp-phase-ene} and finishes the proof.
\end{proof}

\subsection{Strict localization}\label{subsec:supp-local}

Strict localization here concerns the generated velocity. For a measurable spacetime region $\calU_\hbar$, it means
\begin{equation}\label{eq:supp-strict}
S_\hbar=\1_{\calU_\hbar}S_\hbar
\qquad\text{and}\qquad
\|\partial_tu_\hbar-S_\hbar\|_{L^2(Q_T)}\longrightarrow0.
\end{equation}
All support statements for measurable functions are understood up to null sets. In particular, \eqref{eq:supp-strict} implies
\begin{equation*}
\|\1_{Q_T\setminus\calU_\hbar}\partial_tu_\hbar\|_{L^2(Q_T)}
\le\|\partial_tu_\hbar-S_\hbar\|_{L^2(Q_T)}\longrightarrow0.
\end{equation*}
We first localize through the amplitude, then explain how to remove a remaining small tail.

\subsubsection{Localization through the amplitude}\label{subsubsec:supp-amplitude}

In Proposition~\ref{prop:supp-free}, the amplitude $A$ determines a containing shape, $\lambda$ determines its size, and $\omega$ determines its motion. For a general phase, the corresponding motion is given by the ray flow. We first fix the geometry in which the localized amplitude will be transported.

\begin{defi}[Ray flow and an interior ray tube]\label{def:supp-ray-flow}
Fix $T>0$ and a smooth real phase $\varphi$ satisfying \eqref{eq:supp-eik} on $[0,T]\times U$, where $U\Subset\Omega$ is open. The associated ray velocity and ray flow are
\begin{equation}\label{eq:supp-ray-flow}
\begin{cases}
    V_\varphi:=-\frac{\nabla\varphi}{\partial_t\varphi}, \\
\frac{\dd}{\dd t}X_t(a)=V_\varphi\bigl(t,X_t(a)\bigr),
\qquad X_0(a)=a.
\end{cases}
\end{equation}
Choose $x_0\in U$ and $\lambda_0\in(0,1]$ so that this flow is defined for $0\le t\le T$ on $x_0+\lambda_0\overline{B(0,1)}$ and its trajectories form a compact interior ray tube:
\begin{equation*}
\bigcup_{0\le t\le T}X_t\bigl(x_0+\lambda_0\overline{B(0,1)}\bigr)\Subset U.
\end{equation*}
\end{defi}

\begin{prop}\label{prop:supp-local}
Fix the phase $\varphi$, flow $X_t$ and parameters $T,U,x_0,\lambda_0$ of Definition~\ref{def:supp-ray-flow}. Let $A_0\in C_c^\infty(B(0,1))$ be nonzero and real, and let $q\in L^2(\R/2\pi\Z;\R)$ be nonzero with zero mean.

For $\lambda\in(0,\lambda_0]$, define $\calA_\lambda$ by transporting the initial data
\begin{equation}\label{eq:supp-local-init}
\calA_\lambda(0,x):=
\frac{\lambda^{-d/2}A_0\bigl((x-x_0)/\lambda\bigr)}{\partial_t\varphi(0,x)},
\end{equation}
along $X_t$ according to \eqref{eq:supp-trans}, and extend it by zero outside their transported support. Define the velocity amplitude and generated velocity by
\begin{equation}\label{eq:supp-local-def}
\begin{aligned}
A_\lambda(t,x)&:=\calA_\lambda(t,x)\,\partial_t\varphi(t,x),\\
S_{\hbar,\lambda}(t,x)&:=A_\lambda(t,x)\,q\left(\frac{\varphi(t,x)}{\hbar}\right),
\end{aligned}
\end{equation}
with the same zero extension. Then
\begin{equation}\label{eq:supp-local-support}
\begin{gathered}
\supp S_{\hbar,\lambda}(t,\cdot)\subseteq
\calU_\lambda(t):=X_t\bigl(x_0+\lambda\supp A_0\bigr),\\
\diam\calU_\lambda(t)\le C\lambda,\qquad\forall t\in[0,T].
\end{gathered}
\end{equation}
Moreover, there are exact real waves $u_{\hbar,\lambda}$, with the homogeneous Dirichlet condition when $\partial\Omega\ne\varnothing$, such that
\begin{equation}\label{eq:supp-local-bound}
\|\partial_tu_{\hbar,\lambda}-S_{\hbar,\lambda}\|_{L^2(Q_T)}
\le C\frac{\hbar}{\lambda^2},
\quad\forall\hbar\in(0,1],\ \lambda\in(0,\lambda_0].
\end{equation}
The constants are independent of $\hbar$ and $\lambda$. Thus $\lambda\to0$ and $\hbar/\lambda^2\to0$ give velocity approximation in shrinking moving regions.
\end{prop}

\begin{proof}
We first write the amplitude along the ray flow. Dividing \eqref{eq:supp-trans} by $2\partial_t\varphi$, we obtain
\begin{equation*}
(\partial_t+V_\varphi\cdot\nabla)\calA_\lambda+c_\varphi\calA_\lambda=0  \textrm{ \; where \; }
c_\varphi:=\frac{\Box\varphi}{2\partial_t\varphi}.
\end{equation*}
Consequently,
\begin{equation}\label{eq:supp-local-transport}
\calA_\lambda\bigl(t,X_t(a)\bigr) =\frac{\lambda^{-d/2}A_0\bigl((a-x_0)/\lambda\bigr)}{\partial_t\varphi(0,a)}
\exp\left(-\int_0^t c_\varphi\bigl(\tau,X_\tau(a)\bigr)\dd\tau\right).
\end{equation}
The exponential never vanishes. Thus the amplitude support is precisely transported by $X_t$, and the profile can only make the support of $S_{\hbar,\lambda}$ smaller. The uniform Lipschitz bound for $X_t$ gives the diameter estimate in \eqref{eq:supp-local-support}.

The flow, its inverse and their derivatives are bounded on the fixed compact ray tube. Differentiating \eqref{eq:supp-local-transport} and changing variables from $x$ to $a$, we therefore obtain
\begin{equation}\label{eq:supp-local-der}
\sup_{t\in[0,T]} \|\partial_{t,x}^{\beta}\calA_\lambda(t,\cdot)\|_{L^2(\Omega)}
\le C_\beta\lambda^{-|\beta|}, \qquad |\beta|\le2.
\end{equation}
The zero extension is smooth and vanishes near $\partial\Omega$. Let $F$ be the zero-mean periodic primitive of $q$, put $v_{\hbar,\lambda}:=\hbar\calA_\lambda F(\varphi/\hbar)$, and choose the exact wave with Cauchy data $v_{\hbar,\lambda}[0]$. Applying the explicit energy-correction estimate \eqref{eq:supp-phase-bound}, we deduce
\begin{align*}
\|\partial_tu_{\hbar,\lambda}-S_{\hbar,\lambda}\|_{L^2(Q_T)}
&\le C\hbar\bigl(\lambda^{-1}+\lambda^{-2}\bigr)\\
&\le C'\hbar\lambda^{-2}.
\end{align*}
This proves \eqref{eq:supp-local-bound} and finishes the proof.
\end{proof}

For $\varphi=x\cdot\omega-t$, the flow is $X_t(a)=a+t\omega$ and the velocity amplitude is
\begin{equation*}
A_\lambda(t,x)=\lambda^{-d/2}A_0\left(\frac{x-t\omega-x_0}{\lambda}\right).
\end{equation*}
Thus Proposition~\ref{prop:supp-local} recovers the moving region in Proposition~\ref{prop:supp-free}.

\Needspace{17\baselineskip}
\subsubsection{Localization by removing small tails}\label{subsubsec:supp-tail}

A second mechanism begins with a generated velocity that already concentrates in a region. We then delete its small tail.

\begin{cor}\label{cor:supp-tail}
Let $u_\hbar$ be exact waves and let $\widetilde S_\hbar\in L^2(Q_T)$ satisfy
\begin{equation*}
\|\partial_tu_\hbar-\widetilde S_\hbar\|_{L^2(Q_T)}\longrightarrow0.
\end{equation*}
Suppose measurable regions $\calU_\hbar\subset Q_T$ satisfy
\begin{equation*}
\|\1_{Q_T\setminus\calU_\hbar}\widetilde S_\hbar\|_{L^2(Q_T)}\longrightarrow0.
\end{equation*}
Then $S_\hbar:=\1_{\calU_\hbar}\widetilde S_\hbar$ is strictly localized in $\calU_\hbar$ and is approximated by the same exact waves:
\begin{equation}\label{eq:supp-tail}
\begin{split}
\|\partial_tu_\hbar-S_\hbar\|_{L^2(Q_T)}
\le{}&\|\partial_tu_\hbar-\widetilde S_\hbar\|_{L^2(Q_T)}\\
&+\|\1_{Q_T\setminus\calU_\hbar}\widetilde S_\hbar\|_{L^2(Q_T)}\longrightarrow0.
\end{split}
\end{equation}
\end{cor}

\begin{proof}
Since $\widetilde S_\hbar-S_\hbar=\1_{Q_T\setminus\calU_\hbar}\widetilde S_\hbar$, the triangle inequality gives \eqref{eq:supp-tail}.
\end{proof}

The cutoff acts on the generated velocity in $L^2(Q_T)$, while the exact wave remains unchanged. More generally, the same argument applies whenever $\|S_\hbar-\widetilde S_\hbar\|_{L^2(Q_T)}\to0$, allowing both removal of tails and replacement by a simpler localized expression.

\subsection{Boundary reflection and caustics}\label{subsec:supp-geometry}

The previous sections treat a smooth phase in an interior ray tube. We now examine two mechanisms that connect different phase branches: \emph{reflection} at the physical boundary and turning at a \emph{caustic}. The disc makes the distinction explicit.

\subsubsection{Boundary reflection}\label{subsubsec:supp-reflection}

At a non-glancing specular reflection, the tangential momentum is unchanged and the normal momentum reverses. If $\mathbf n$ is the unit normal at the reflection point, then
\begin{equation*}
\xi_{\rm ref}=\xi_{\rm in}-2(\xi_{\rm in}\cdot\mathbf n)\mathbf n.
\end{equation*}
This local relation applies to both flat and curved boundaries; see Figure~\ref{fig:supp-rays}. The boundary condition links the incident and reflected waves.

\begin{figure}[htbp]
\centering
\begin{tikzpicture}[x=1cm,y=1cm,>=Stealth,font=\small]
\definecolor{rayblue}{RGB}{48,96,154}
\definecolor{rayred}{RGB}{171,70,56}
\foreach \shift/\curved/\heading in {0/0/{(a) Flat boundary},7.6/1/{(b) Curved boundary}}{
\begin{scope}[xshift=\shift cm]
\node[font=\bfseries] at (0,3.25) {\heading};
\ifnum\curved=0
  \fill[gray!10] (-3,-.68) rectangle (3,0);
  \draw[very thick] (-3,0)--(3,0);
\else
  \fill[gray!10] plot[domain=-3:3,samples=65] (\x,{.08*\x*\x}) -- (3,-.68)--(-3,-.68)--cycle;
  \draw[very thick] plot[domain=-3:3,samples=65] (\x,{.08*\x*\x});
  \draw[gray,densely dashed] (-2.4,0)--(2.4,0);
  \node[gray,below] at (2.0,0) {tangent};
\fi
\draw[densely dashed,gray] (0,0)--(0,2.65);
\node[above] at (0,2.65) {normal};
\draw[rayblue,very thick,-{Stealth[length=2.6mm]}] (-2.05,2.45)--(-.035,.042);
\draw[rayred,very thick,-{Stealth[length=2.6mm]}] (0,0)--(2.05,2.45);
\fill (0,0) circle (1.5pt);
\node[below left] at (0,0) {$p$};
\node[rayblue] at (-2.05,1.65) {incident};
\node[rayred] at (2.08,1.65) {reflected};
\draw (50.1:.82) arc[start angle=50.1,end angle=90,radius=.82];
\draw (90:.82) arc[start angle=90,end angle=129.9,radius=.82];
\node at (-.36,1.03) {$\theta$};
\node at (.36,1.03) {$\theta$};
\node at (0,-.52) {boundary};
\end{scope}
}
\end{tikzpicture}

\caption{Specular reflection at a flat boundary (left) and a curved boundary (right). The incoming and outgoing directions make equal angles with the normal at $p$. The boundary is solid black; the gray region lies outside the domain.}\label{fig:supp-rays}
\end{figure}
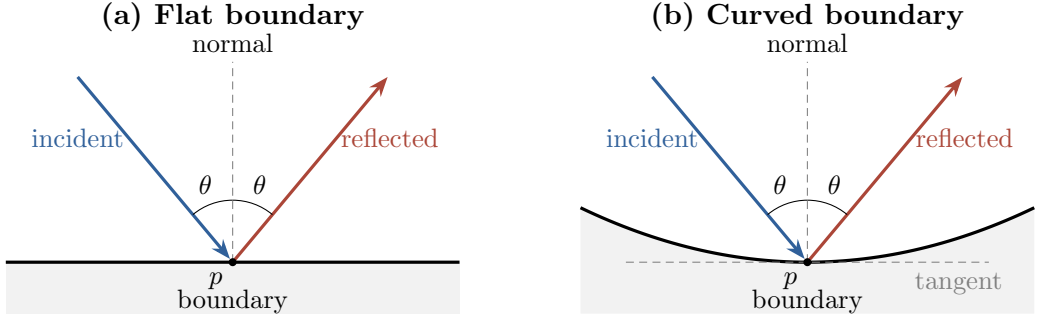

The half-space $\{x_d>0\}$ gives a direct example. Choose $\nu>0$ and $\omega'\in\R^{d-1}$ with $|\omega'|^2+\nu^2=1$, and set
\begin{equation}\label{eq:supp-ref-phases}
\begin{aligned}
\varphi_{\rm in}&:=x'\cdot\omega'-\nu x_d-t,\\
\varphi_{\rm ref}&:=x'\cdot\omega'+\nu x_d-t.
\end{aligned}
\end{equation}
Both phases satisfy the eikonal equation, and they agree on the boundary. For a smooth periodic profile $q$ with primitive $F$, define
\begin{equation}\label{eq:supp-ref-wave}
v_\hbar:=-\hbar\bigl[F(\varphi_{\rm in}/\hbar)-F(\varphi_{\rm ref}/\hbar)\bigr].
\end{equation}
Direct differentiation gives
\begin{equation*}
\begin{gathered}
\Box v_\hbar=0\qquad\text{and}\qquad v_\hbar\big|_{x_d=0}=0,\\
\partial_tv_\hbar=q(\varphi_{\rm in}/\hbar)-q(\varphi_{\rm ref}/\hbar).
\end{gathered}
\end{equation*}
Thus the same profile appears in both branches, with opposite signs already included in their difference.

On a curved boundary, the reflected phase and amplitude follow the local geometry, while the boundary condition again relates the traces of the connected branches. A shift between their normalized boundary phases appears as a translation between the profiles. We use this relation in Section~\ref{subsubsec:supp-profiles}.

\subsubsection{Caustics}\label{subsubsec:supp-caustic}

A caustic is an envelope of a family of rays. We use the disc to connect this geometry to the wave equation: the squared radial momentum of the rays is also the leading coefficient in the radial wave equation. Its simple zero gives the Airy model.

Glancing propagation and its Airy description are classical parts of microlocal wave analysis; see Friedlander \cite{Friedlander1976} and Melrose--Taylor \cite{MelroseTaylor1986}. Their role in boundary concentration, dispersion and Strichartz estimates has been studied extensively; see, for instance, \cite{SmithSogge2007,Ivanovici2010,ILP2014}. Here we give an explicit disc calculation to explain the relation between the two branches and the resulting constraints on compatible profiles. We refer to \cite[Section~1.5]{AOSII} for further background and its connection with the glancing construction.

\noindent\textit{The two radial branches.}
Consider unit-speed rays in the unit disc, with one fixed counterclockwise angular momentum
\begin{equation*}
\ell:=x_1\xi_2-x_2\xi_1,\qquad 0<\ell<1.
\end{equation*}
This quantity is preserved both along each straight segment and at reflection on the circle. In polar coordinates $(r,\vartheta)$, the tangential velocity is $\ell/r$. Thus the radial momentum satisfies
\begin{equation}\label{eq:supp-disc-mom}
p_r^2+\frac{\ell^2}{r^2}=1,
\qquad
p_r^\pm=\pm\sqrt{1-\frac{\ell^2}{r^2}},
\qquad r\in(\ell,1).
\end{equation}
At each point with $r>\ell$, the two signs describe inward and outward radial motion with the same counterclockwise orientation. They correspond to the two tangents drawn from $P$ in Figure~\ref{fig:supp-caustic}(a).
\vspace{2mm}

At $r=\ell$, the radial momentum vanishes and the two directions coincide with the tangent to the inner circle. No ray in this family enters $r<\ell$. Hence $r=\ell$ is the caustic. A ray continues along a straight line through its closest point to the center; its radial motion changes from inward to outward there. Reflection occurs instead at the physical boundary $r=1$. If $\theta$ is the incidence angle measured from the boundary normal, then $\ell=\sin\theta$. Glancing corresponds to $\ell\to1$, when the caustic approaches the boundary.

\begin{figure}[htbp]
\centering
\begin{tikzpicture}[x=1cm,y=1cm,>=Stealth,font=\small]
\definecolor{rayblue}{RGB}{48,96,154}
\definecolor{rayred}{RGB}{171,70,56}
\begin{scope}[xshift=.3cm,yshift=3.0cm]
\node[font=\bfseries] at (0,2.8) {(a) A caustic in the disc};
\def\outerrad{2.2}
\def\innerrad{1.45}
\pgfmathsetmacro{\halflength}{sqrt(\outerrad*\outerrad-\innerrad*\innerrad)}
\fill[gray!6] (0,0) circle (\innerrad);
\foreach \ang in {0,30,...,330}{
  \draw[rayblue!30,thin]
  ({\innerrad*cos(\ang)+\halflength*sin(\ang)},{\innerrad*sin(\ang)-\halflength*cos(\ang)})
  --({\innerrad*cos(\ang)-\halflength*sin(\ang)},{\innerrad*sin(\ang)+\halflength*cos(\ang)});
}
\pgfmathsetmacro{\advanceangle}{2*acos(\innerrad/\outerrad)}
\foreach \j in {0,1}{
  \pgfmathsetmacro{\startangle}{115+\j*\advanceangle}
  \pgfmathsetmacro{\endangle}{\startangle+\advanceangle}
  \draw[rayblue!70,thick,postaction={decorate},decoration={markings,mark=at position .65 with {\arrow{Stealth}}}]
    (\startangle:\outerrad)--(\endangle:\outerrad);
}
\pgfmathsetmacro{\contactangle}{acos(\innerrad/1.85)}
\foreach \sign in {-1,1}{
  \pgfmathsetmacro{\ang}{\sign*\contactangle}
  \draw[rayblue,very thick,postaction={decorate},decoration={markings,mark=at position .68 with {\arrow{Stealth}}}]
    ({\innerrad*cos(\ang)+\halflength*sin(\ang)},{\innerrad*sin(\ang)-\halflength*cos(\ang)})
    --({\innerrad*cos(\ang)-\halflength*sin(\ang)},{\innerrad*sin(\ang)+\halflength*cos(\ang)});
  \fill[rayred] (\ang:\innerrad) circle (1.4pt);
}
\draw[very thick] (0,0) circle (\outerrad);
\draw[rayred,thick,dashed] (0,0) circle (\innerrad);
\fill (1.85,0) circle (1.6pt);
\node[anchor=west,fill=white,inner sep=1.2pt] at (1.94,0) {$P$};
\node[text=rayred,fill=white,inner sep=2pt,align=center] at (0,.18) {caustic\\$r=\ell$};
\draw[rayred,thin] (-.1,.66)--(-.55,1.32);
\node[anchor=east,align=center] at (-1.9,1.8) {boundary\\$r=1$};
\draw[thin] (-1.91,1.58)--(-1.55,1.56);
\draw[rayblue,thin] (2.03,.23)--(2.57,.68);
\node[anchor=west,text=rayblue,fill=white,inner sep=1pt] at (2.58,.68) {$p_r>0$};
\draw[rayblue,thin] (2.03,-.23)--(2.57,-.68);
\node[anchor=west,text=rayblue,fill=white,inner sep=1pt] at (2.58,-.68) {$p_r<0$};
\node[align=center,text=rayblue] at (0,-2.62) {two directions at $P$; one tangent at the caustic};
\end{scope}
\begin{scope}[xshift=8.0cm,yshift=2.3cm]
\node[font=\bfseries] at (0,3.5) {(b) Boundary coordinates};
\fill[gray!6] (-3,2.55)--plot[domain=-3:3,samples=65] (\x,{8-sqrt(6.55*6.55-\x*\x)})--(3,3.05)--(-3,3.05)--cycle;
\fill[gray!10] plot[domain=-3:3,samples=65] (\x,{8-sqrt(64-\x*\x)})--(3,-.62)--(-3,-.62)--cycle;
\begin{scope}
\clip (-3,0) rectangle (3,2.8);
\foreach \tangentx in {-1.3,0,1.3}{
  \pgfmathsetmacro{\slope}{\tangentx/sqrt(6.55*6.55-\tangentx*\tangentx)}
  \pgfmathsetmacro{\tangenty}{8-sqrt(6.55*6.55-\tangentx*\tangentx)}
  \draw[rayblue!65,thick,postaction={decorate},decoration={markings,mark=at position .62 with {\arrow{Stealth}}}]
    plot[domain=-3:3,samples=2] (\x,{\tangenty+\slope*(\x-\tangentx)});
}
\end{scope}
\draw[very thick] plot[domain=-3:3,samples=65] (\x,{8-sqrt(64-\x*\x)});
\draw[rayred,thick,dashed] plot[domain=-3:3,samples=65] (\x,{8-sqrt(6.55*6.55-\x*\x)});
\draw[->,gray!80] (-2.4,.3685)--(-1.65,2.7534);
\node[left] at (-1.65,2.75) {$y$};
\draw[->,gray!80] plot[domain=-.7:1.4,samples=30] (\x,{8-sqrt(64-\x*\x)-.22});
\node[below] at (1.4,-.0965) {$s$};
\node[fill=white,inner sep=1.5pt] at (.55,.57) {$\zeta>0$: oscillatory};
\node[fill=white,inner sep=1.5pt] at (.55,2.72) {$\zeta<0$: decaying};
\node[text=rayred,fill=white,inner sep=1.5pt,align=center] at (.45,2.07) {$\zeta=0$\\$y=y_c(s)$};
\draw[rayred,thin] (.45,1.83)--(.45,1.47);
\node at (0,-.68) {boundary: $y=0$};
\end{scope}
\node[anchor=north,inner sep=0pt] at (4.1,-.45) {\includegraphics[width=14.4cm]{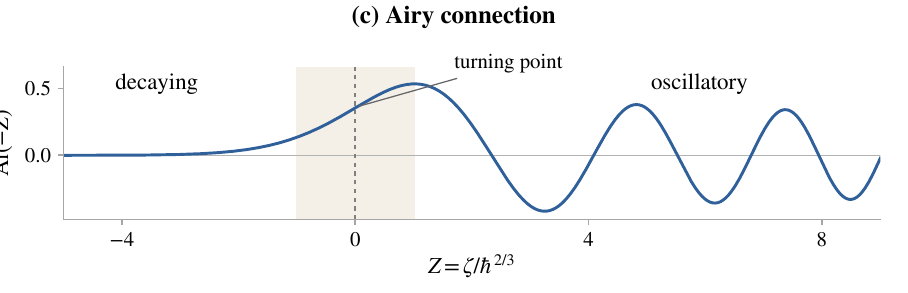}};
\end{tikzpicture}

\caption{From a ray caustic to the Airy connection. (a) Unit-disc rays with angular momentum $\ell$ are tangent to $r=\ell$; the two highlighted directions at $P$ have opposite radial momenta and the same counterclockwise orientation. (b) Inward boundary coordinates separate the physical boundary $y=0$ from the turning curve $y=y_c(s)$; for the unit disc, $y_c=1-\ell$. (c) The Airy function $\Ai(-Z)$, with $Z=\zeta/\hbar^{2/3}$, connects the decaying and oscillatory sides smoothly.}\label{fig:supp-caustic}
\end{figure}
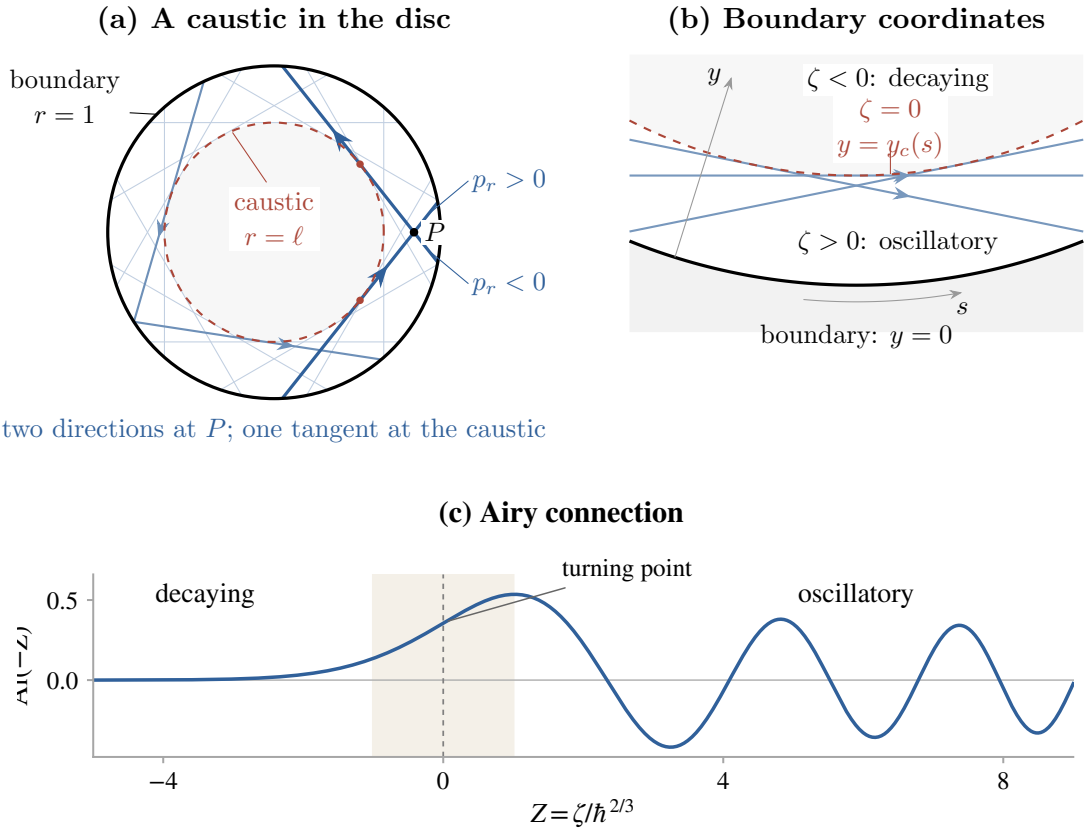

\noindent\textit{The same coefficient in the wave equation.}
Separate time and angular variables in a complex mode,
\begin{equation*}
u(t,r,\vartheta)=e^{-it/\hbar}e^{im\vartheta}v(r),
\qquad m\in\Nat\setminus\{0\},\qquad \ell=m\hbar.
\end{equation*}
Substitution into $\Box u=0$ gives the exact radial equation
\begin{equation}\label{eq:supp-disc-rad}
\hbar^2\left(v''+\frac1r v'\right)+ \left(1-\frac{\ell^2}{r^2}\right)v=0.
\end{equation}
The coefficient in parentheses is precisely the squared radial momentum in \eqref{eq:supp-disc-mom}. The Dirichlet condition is imposed at the outer boundary through $v(1)=0$.

To remove the first derivative, set $v(r)=r^{-1/2}W(r)$. Equation~\eqref{eq:supp-disc-rad} becomes
\begin{equation}\label{eq:supp-disc-normal}
\hbar^2W''+ \left(1-\frac{\ell^2}{r^2}+\frac{\hbar^2}{4r^2}\right)W=0.
\end{equation}
Thus the same function controls both the ray geometry and the leading radial oscillations of the wave.
\vspace{2mm}

\noindent\textit{A simple turning point.}
Near the caustic,
\begin{equation}\label{eq:supp-disc-turn}
1-\frac{\ell^2}{r^2} = \frac2\ell(r-\ell)+O\bigl((r-\ell)^2\bigr).
\end{equation}
The coefficient has a simple zero at $r=\ell$. On the oscillatory side, define the radial action and the two propagation phases by
\begin{equation}\label{eq:supp-disc-action}
\begin{aligned}
R_\ell(r)&:=\int_\ell^r\sqrt{1-\frac{\ell^2}{\rho^2}}\dd\rho,\\
\varphi_\pm(t,r,\vartheta)&:=\ell\vartheta-t\pm R_\ell(r).
\end{aligned}
\end{equation}
Indeed, $\partial_r\varphi_\pm=p_r^\pm$ and
\begin{equation*}
|\nabla\varphi_\pm|^2 = \bigl(R_\ell'(r)\bigr)^2+\frac{\ell^2}{r^2}=1
=(\partial_t\varphi_\pm)^2.
\end{equation*}
Away from $r=\ell$, the two phases therefore fit the eikonal equation used in Section~\ref{subsec:supp-phase}.
\vspace{2mm}

For the leading radial equation, the separate geometric-optics expressions have the form
\begin{equation*}
W_\pm(r)\approx
\left(1-\frac{\ell^2}{r^2}\right)^{-1/4}e^{\pm iR_\ell(r)/\hbar}.
\end{equation*}
Their amplitude factor becomes singular as $r\to\ell^+$. The two-branch representation breaks down at the turning point, even though the radial equation \eqref{eq:supp-disc-normal} has smooth coefficients there. We need one expression that describes both branches and their transition.

\noindent\textit{The Airy normal coordinate.}
The simple zero in \eqref{eq:supp-disc-turn} allows a smooth increasing coordinate $\zeta=\zeta(r)$ near $r=\ell$ satisfying
\begin{equation}\label{eq:supp-disc-zeta}
\zeta(\ell)=0,\quad \zeta(\zeta')^2=1-\frac{\ell^2}{r^2},
\quad \zeta'>0.
\end{equation}
On the oscillatory side, its definition is equivalently
\begin{equation*}
R_\ell(r)=\frac23\zeta(r)^{3/2},
\quad \zeta(r)=\left(\frac2\ell\right)^{1/3}(r-\ell)
+O\bigl((r-\ell)^2\bigr).
\end{equation*}
Thus the action in the new coordinate is the original radial action.

Set
\begin{equation*}
W(r)=\bigl(\zeta'(r)\bigr)^{-1/2}V\bigl(\zeta(r)\bigr).
\end{equation*}
The first-derivative terms in $V$ cancel, and \eqref{eq:supp-disc-normal} becomes
\begin{equation}\label{eq:supp-airy-reduced}
\hbar^2V''(\zeta)+\zeta V(\zeta)+ \hbar^2c_\ell(\zeta)V(\zeta)=0,
\end{equation}
where $c_\ell$ is smooth near $\zeta=0$. The leading normal equation is therefore the Airy model \cite{Olver},
\begin{equation}\label{eq:supp-airy-model}
\hbar^2V''(\zeta)+\zeta V(\zeta)=0.
\end{equation}
Equation~\eqref{eq:supp-airy-reduced} records the lower-order term separately from this model.

\noindent\textit{One expression across the turning point.}
Balancing the two terms in \eqref{eq:supp-airy-model} gives the scale $\zeta=\hbar^{2/3}Z$. More precisely, with $\mathcal V(Z):=V(\hbar^{2/3}Z)$, that equation becomes
\begin{equation*}
\mathcal V''(Z)+Z\mathcal V(Z)=0.
\end{equation*}
The solution that decays as $Z\to-\infty$ is, up to a constant factor,
\begin{equation*}
V_\hbar(\zeta)=\Ai\left(-\frac\zeta{\hbar^{2/3}}\right).
\end{equation*}
Figure~\ref{fig:supp-caustic}(c) shows its oscillations for $\zeta>0$, decay for $\zeta<0$ and smooth transition at $\zeta=0$. The transition width $O(\hbar^{2/3})$ follows from the equation, rather than from an additional choice of scale.

The same simple-turning-point scale occurs in the Airy description of whispering-gallery concentration near a curved boundary; see \cite{SmithSogge2007}. This normal scale alone does not determine the spacetime estimates, which also depend on how the waves propagate and concentrate over time.

On the oscillatory side, write $R:=\frac23\zeta^{3/2}=R_\ell(r)$. The Airy expansion is
\begin{equation}\label{eq:supp-airy}
\begin{gathered}
\Ai(-\zeta/\hbar^{2/3})
=\pi^{-1/2}\hbar^{1/6}\zeta^{-1/4}
\left[\sin\left(\frac{R}{\hbar}+\frac\pi4\right)
+O\left(\frac{\hbar}{\zeta^{3/2}}\right)\right],\\
\zeta/\hbar^{2/3}\longrightarrow+\infty.
\end{gathered}
\end{equation}
After multiplication by $(\zeta')^{-1/2}$, its oscillatory amplitude recovers the earlier factor:
\begin{equation*}
(\zeta')^{-1/2}\zeta^{-1/4}
=\left(1-\frac{\ell^2}{r^2}\right)^{-1/4}.
\end{equation*}
Thus one smooth Airy expression generates the two phase branches away from the caustic. Their relative coefficients are fixed by the sine in \eqref{eq:supp-airy}; they are not independently chosen.
\vspace{2mm}

\noindent\textit{Boundary coordinates and variable curvature.}
In the unit disc, inward boundary distance gives
\begin{equation*}
y=1-r,\qquad y_c=1-\ell,\qquad r-\ell=y_c-y.
\end{equation*}
Consequently, near the turning curve, the signs of the coordinate in Figure~\ref{fig:supp-caustic}(b) are
\begin{equation*}
\begin{cases}
\zeta>0 &\text{on the oscillatory side, }0<y<y_c,\\
\zeta=0 &\text{on the turning curve, }y=y_c,\\
\zeta<0 &\text{on the decaying side, }y>y_c.
\end{cases}
\end{equation*}
For a variable-curvature boundary, we use a chart $x=x(s,y)$, with $s$ the arclength and $y_c(s)$ the turning depth. The next section combines the Airy connection at this turning curve with the Dirichlet relation at $y=0$ to determine compatible profiles.

\subsection{Compatible profiles and multi-branch velocity approximation}\label{subsec:supp-compat}

In free space, separate waves may carry independently chosen profiles. Branches belonging to one prescribed reflection or caustic construction are instead linked by that construction. We call the resulting requirements \emph{profile compatibility}. We first read the relations from the Airy and boundary expressions, then combine the profiles into an approximate wave whose time derivative has the required leading velocity.

\subsubsection{Compatibility relations}\label{subsubsec:supp-profiles}

To read the relation between their profiles, let $C(x)-t$ be the common tangential propagation phase and let $n>0$ be a profile Fourier frequency. The two branches have phases $C-t+R$ and $C-t-R$. Writing the sine as two exponentials gives
\begin{equation}\label{eq:supp-airy-factors}
\begin{split}
2i e^{in(C-t)/\hbar}\sin\left(\frac{nR}{\hbar}+\frac\pi4\right)
={}&e^{i\pi/4}e^{in(C-t+R)/\hbar}\\
&-e^{-i\pi/4}e^{in(C-t-R)/\hbar}.
\end{split}
\end{equation}
The same Fourier coefficient generates both branches, with relative factors $e^{i\pi/4}$ and $e^{-i\pi/4}$ fixed by the Airy connection. The factor $2i$ arises from time differentiation in the wave construction of Section~\ref{subsubsec:supp-assembly}.
\vspace{2mm}

We now use period-one reference variables, as in the glancing construction; a period-$2\pi$ profile is written as $q_{\rm ref}(\xi):=q(2\pi\xi)$, so that $q(\varphi/\hbar)=q_{\rm ref}\bigl(\varphi/(2\pi\hbar)\bigr)$.

Let $H(\xi)=\sum_{n>0}d_ne^{2\pi in\xi}$ be a positive-frequency $L^2$ function, and define the real profiles
\begin{equation}\label{eq:supp-fg}
f:=\Real(e^{i\pi/4}H)\quad \text{and} \quad
g:=\Real(e^{-i\pi/4}H).
\end{equation}
With the periodic Hilbert transform convention
\begin{equation*}
\calH(e^{2\pi in\xi}):=-i\sgn(n)e^{2\pi in\xi},\qquad n\ne0,
\end{equation*}
the ratio $e^{-i\pi/4}/e^{i\pi/4}=-i$ of the positive-frequency factors in \eqref{eq:supp-fg} gives
\begin{equation}\label{eq:supp-hilbert}
g=\calH f.
\end{equation}
The negative frequencies are their conjugates because the profiles are real. The boundary condition imposes a second relation on the same pair. Suppose the chosen normalized boundary phases satisfy $\psi=\phi+\beta$ modulo $1$, and the leading branch difference is $f(\phi)-g(\psi)$ with a common amplitude. Cancellation of that boundary expression for all phase values requires
\begin{equation*}
f=\tau_\beta g, \quad \text{ where }\tau_\beta g(\xi):=g(\xi+\beta).
\end{equation*}
Together with \eqref{eq:supp-hilbert}, this gives $f=\tau_\beta\calH f$. In the glancing construction, the scale is coordinated so that $\beta=1/4$. Thus the compatible profile space is
\begin{equation}\label{eq:supp-X}
X:=\left\{f\in L^2(\R/\Z;\R):\int_0^1f=0,\quad f=\tau_{1/4}\calH f\right\}.
\end{equation}
For $n>0$, its Fourier condition is
\begin{equation*}
\widehat f(n)=-ie^{\pi in/2}\widehat f(n),
\quad\text{ hence }\quad
\widehat f(n)=0\quad\text{unless }n\equiv1\pmod4.
\end{equation*}
The quarter-period shift is selected by the boundary phase coordination; the Hilbert-transform relation comes from the Airy connection.

\begin{remark}\label{rem:supp-match}
Compatibility and matching answer different questions. For example,
\begin{equation*}
f(\xi)=\cos(2\pi\xi) \quad\text{ and } \quad g(\xi)=\sin(2\pi\xi)
\end{equation*}
satisfy $g=\calH f$ and $f=\tau_{1/4}g$. However,
\begin{equation*}
\bigl|\{(\xi,\eta)\in(\R/\Z)^2:f(\xi)=g(\eta)\}\bigr|=0.
\end{equation*}
Indeed, for each fixed $\xi$, the equality has only finitely many solutions in $\eta$, and Fubini's theorem applies. Thus compatibility alone does not produce a positive-measure matching region. The diamond construction in \cite[Section~2]{AOSII} addresses the additional problem of finding compatible profiles with the nonzero plateaus needed for such matching.
\end{remark}

\subsubsection{Multi-branch velocity approximation}\label{subsubsec:supp-assembly}

We now explain how the branch profiles enter a common wave expression. Away from a caustic, this is visible directly through primitives. Across a caustic, the same role is played by a uniform Airy expression rather than by its singular two-branch expansion.

\noindent\textit{Regular phase branches.}
Let $f_\alpha$ be finitely many real zero-mean period-$2\pi$ profiles, and let $F_\alpha'=f_\alpha$ be their zero-mean primitives. For smooth branch data $(\calA_\alpha,\varphi_\alpha)$, set
\begin{equation}\label{eq:supp-assembly}
\begin{split}
v_\hbar&:=\hbar\sum_{\alpha=1}^J \calA_\alpha F_\alpha(\varphi_\alpha/\hbar),\\
S_\hbar&:=\sum_{\alpha=1}^J A_\alpha f_\alpha(\varphi_\alpha/\hbar),
\qquad A_\alpha:=\calA_\alpha\partial_t\varphi_\alpha.
\end{split}
\end{equation}
If each branch satisfies the eikonal and transport equations on its support, the calculation in Proposition~\ref{prop:supp-phase} gives
\begin{equation}\label{eq:supp-assembly-res}
\begin{split}
\partial_tv_\hbar-S_\hbar
&=\hbar\sum_{\alpha=1}^J (\partial_t\calA_\alpha)F_\alpha(\varphi_\alpha/\hbar),\\
\Box v_\hbar
&=\hbar\sum_{\alpha=1}^J (\Box \calA_\alpha)F_\alpha(\varphi_\alpha/\hbar).
\end{split}
\end{equation}
For fixed smooth branch data with uniformly compact spatial support and finite $J$, both errors are $O(\hbar)$ in the norms of Lemma~\ref{lem:supp-corr}. When the assembled $v_\hbar$ has zero boundary trace, that lemma produces an exact Dirichlet wave approximating $S_\hbar$.

The profile maps $f_\alpha=L_\alpha q$ in AOS encode the boundary and caustic relations between these branches. The wave construction combines the compatible profiles with the phases and amplitudes in \eqref{eq:supp-assembly}.
\vspace{2mm}

\noindent\textit{Airy assembly across a turning curve.}
There is an equally explicit way to read the two profiles from a common expression. Use the boundary chart $x=x(s,y)$ from Figure~\ref{fig:supp-caustic}(b), with $s$ the boundary arclength and $y$ the inward normal distance. The smooth turning coordinate $\zeta$ vanishes on $y=y_c(s)$ and is positive between that curve and the boundary. Let $C$ and $b$ be smooth real spatial functions, and take a finite positive-frequency profile
\begin{equation*}
H(\xi)=\sum_{n\in\mathcal N}d_ne^{2\pi in\xi},
\qquad\mathcal N\subset\Nat\setminus\{0\}\text{ finite}.
\end{equation*}
Define $f,g$ by \eqref{eq:supp-fg}. Before imposing the geometric equations and correcting the residual, consider the Airy ansatz
\begin{equation}\label{eq:supp-airy-v}
v_\hbar^{\rm Ai}:= -2\sqrt\pi\,\Real\sum_{n\in\mathcal N} d_n\left(\frac\hbar n\right)^{5/6}b\,
\Ai\left(-\frac{\zeta}{(\hbar/n)^{2/3}}\right)e^{in(C-t)/\hbar}.
\end{equation}
This expression is smooth through $\zeta=0$ for every fixed $\hbar>0$. On the oscillatory side, put
\begin{equation}\label{eq:supp-airy-phases}
R:=\tfrac23\zeta^{3/2},\quad \phi:=\frac{C-t+R}{2\pi\hbar}
\quad\text{ and }\quad \psi:=\frac{C-t-R}{2\pi\hbar}.
\end{equation}
Time differentiation of \eqref{eq:supp-airy-v}, followed by \eqref{eq:supp-airy} with $\hbar$ replaced by $\hbar/n$, gives
\begin{equation}\label{eq:supp-airy-leading}
\partial_tv_\hbar^{\rm Ai}
=b\zeta^{-1/4}\bigl[f(\phi)-g(\psi)\bigr]+O(\hbar).
\end{equation}
Here the error is uniform on every fixed compact set where $\zeta\ge\zeta_0>0$, with a constant depending on that set and the finite Fourier data. Indeed, the time derivative changes the factor $-2\sqrt\pi(\hbar/n)^{5/6}$ into $2i\sqrt\pi(\hbar/n)^{-1/6}$; the leading Airy amplitude cancels this power, leaving exactly \eqref{eq:supp-airy-factors} times $b\zeta^{-1/4}$.

The glancing constructions in \cite[Sections~2--3]{AOSII} choose $C,\zeta,b$ to control the wave residual and use Airy zeros to impose the Dirichlet condition. They use \eqref{eq:supp-airy-leading} in the oscillatory region and \eqref{eq:supp-airy-v} near the turning curve. 
The common limiting tangential direction allows one moving envelope to localize both branches; estimates near and beyond the turning curve control the normal tail, which is then removed as in Corollary~\ref{cor:supp-tail}. Fourier approximation and uniform tail estimates recover the full profiles used in the generated velocities and their matching regions.

\section{Spacetime measurable observation and GCC}\label{sec:meas}

For readers' convenience, we first recall generalized rays and classical open-set GCC. We then define the regular billiard flow and the essential spacetime measurable GCC \eqref{eq:gcc}, including its invariance under spacetime null modifications. We explain the difficulty in extending the open-set observation argument to measurable sets and treat the countable limits of observation regions. Remark~\ref{rem:disc-completion} discusses boundary completion for the disc construction.

\subsection{Generalized rays and classical GCC}\label{subsec:meas-classical}

Throughout this appendix, $\Omega$ is a smooth bounded planar domain, not necessarily simply connected, and $T>0$ is fixed. We recall the boundary propagation of Melrose--Sj\"ostrand \cite{MS1978}; see also \cite[Section~1C1]{LRLTT2017}.

A bicharacteristic is a phase-space curve for the principal symbol $p(t,x;\tau,\xi)=|\xi|^2-\tau^2$. Away from the boundary it is an integral curve of the Hamilton field $H_p$ in $\{p=0\}\setminus0$. Its spatial projection, parametrized by physical time, travels at unit speed. We call the spatial projection of a generalized bicharacteristic a \emph{generalized ray}.
\vspace{2mm}

Near the boundary, let $y$ be inward normal distance, $s$ boundary arclength, and $(\tau,\sigma,\eta)$ the covariables of $(t,s,y)$. Write
\begin{equation*}
p=\eta^2+r(t,s,y;\tau,\sigma),\qquad\text{where }r_0:=r|_{y=0}.
\end{equation*}
At a boundary covector, $r_0<0$ gives two distinct normal roots and defines the hyperbolic region; $r_0=0$ gives the glancing region, where $\eta=0$. The elliptic region $r_0>0$ has no real characteristic lift.
\vspace{2mm}

At the boundary, the compressed cotangent variables $(t,s,y;\tau,\sigma,y\eta)$ identify the two normal roots at a reflection. Generalized bicharacteristics are continuous in this compressed space. Their basic local behaviors are as follows.

\begin{enumerate}[label=(\roman*),leftmargin=*]
\item \emph{Interior propagation.} The curve follows $H_p$; for the Euclidean wave equation its spatial projection is a straight line.
\item \emph{Reflection at hyperbolic points.} The ray reaches the boundary transversely. The incoming and outgoing unit velocities satisfy
\begin{equation}\label{eq:meas-reflection}
\xi^+=\xi^--2\bigl(\xi^-\cdot\nu_{\rm out}(x)\bigr)\nu_{\rm out}(x),
\end{equation}
where $\nu_{\rm out}$ is the outward unit normal. The position and tangential velocity are continuous, while the normal velocity changes sign.
\item \emph{Diffractive glancing.} At a second-order glancing point with $H_p^2y>0$, the ordinary bicharacteristic touches the boundary and lies in the interior on both sides of the contact. It continues along $H_p$, with no boundary-gliding interval. A line tangent to the inner circle of an annulus illustrates this case. Here \emph{diffractive} refers to smooth-boundary tangency, not diffraction at a corner.
\item \emph{Gliding.} At a second-order glancing point with $H_p^2y<0$, the free tangent bicharacteristic lies outside the domain on both sides of the contact. The generalized ray instead remains on the boundary locally. Its boundary cotangent variables follow $H_{r_0}$; in the planar Euclidean setting its spatial projection travels along the boundary at unit arclength speed. The circular boundary of the disc supports such trajectories.
\end{enumerate}
A generalized ray may exhibit several of these local behaviors.

\begin{remark}[Higher-order contact and uniqueness]\label{rem:meas-contact}
At higher-order glancing points, where $H_p^2y=0$, interior and boundary propagation may join according to the generalized-bicharacteristic rules; see \cite{MS1978,LRLTT2017} for the classification. In the classical smooth theory, a compressed generalized bicharacteristic with no infinite-order boundary contact is uniquely determined by any one of its points \cite[Proposition~1.5]{LRLTT2017}.

Burq--Dehman--Le Rousseau \cite{BDLR2024} formulate geometric control on \emph{all} generalized trajectories for $C^1$ metrics and prove observability under their condition. The essential condition below uses the full-measure regular billiard flow on arbitrary smooth bounded planar domains.
\end{remark}

\noindent\emph{Classical open-set GCC.} Classical GCC requires every unit-speed generalized ray to reach the observation region within the prescribed time \cite{BLR1992}. For an interior observation cylinder, one may use a relatively open set $\widehat\omega\subset\overline\Omega$, with actual interior observation region $\omega:=\widehat\omega\cap\Omega$. The geometric condition is imposed on $\widehat\omega$.

A boundary point in $\widehat\omega$ therefore comes with an observed interior neighborhood, not just an assigned boundary value. For example, in the disc the collar
\begin{equation*}
\widehat\omega=\{x\in\overline\D:|x|>1-\delta\},\qquad\text{with }0<\delta<1,
\end{equation*}
includes the circle and an interior neighborhood of it. The observation integral is still taken over $(0,T)\times\omega$, with no boundary integral. The analogous spacetime formulation uses a relatively open subset of $(0,T)\times\overline\Omega$; see \cite[Definition~1.6 and Remark~1.7(4)]{LRLTT2017}.

\subsection{Regular billiards and observed time}\label{subsec:meas-flow}

Initial states are taken at time zero in the interior phase space
\begin{equation*}
\begin{gathered}
S^*\Omega=\Omega\times S^1,\qquad
\rho=(x_0,\xi_0),\qquad |\xi_0|=1,\\
\dd\mu_L=\dd x\,\dd\sigma(\xi),\quad\text{where }\sigma(S^1)=2\pi.
\end{gathered}
\end{equation*}
Between impacts, the ray moves with constant unit velocity and obeys \eqref{eq:meas-reflection} at each transverse impact. We call a trajectory regular when its impacts are transverse and do not accumulate in finite time.

Write $\Phi_t\rho$ for the regular billiard flow and $\pi$ for its position projection. At a reflection instant, we use the outgoing velocity. The following lemma collects the properties needed to define and compare observed times.

\begin{lemma}[Regular billiard flow]\label{lem:meas-flow}
Let $\Omega$ be a smooth bounded planar domain. The regular billiard flow has the following properties.
\begin{enumerate}[label=(\roman*),leftmargin=*]
\item For $\mu_L$-almost every initial state in $S^*\Omega$, the trajectory is defined for all $t\in\R$ and has only finitely many impacts, all transverse, on every bounded time interval.
\item The map $(t,\rho)\mapsto\Phi_t\rho$ is jointly measurable on the full-measure set of regular initial states.
\item For every fixed $t\in\R$, the map $\Phi_t$, defined $\mu_L$-almost everywhere in $S^*\Omega$, preserves $\mu_L$.
\end{enumerate}
\end{lemma}

\begin{proof}
\noindent\emph{Almost-everywhere regularity.} Use the outgoing boundary section
\begin{equation*}
\begin{gathered}
\Sigma:=\{(x,\xi):x\in\partial\Omega,\ |\xi|=1,\ \xi\cdot\mathbf n(x)>0\},\\
\text{ with } \quad \dd\nu=\bigl(\xi\cdot\mathbf n(x)\bigr)\,\dd\sigma_{\partial\Omega}(x)\,\dd\sigma(\xi).
\end{gathered}
\end{equation*}
Here $\mathbf n$ is the inward unit normal and $\nu(\Sigma)<\infty$. Grazing segments form a null set: they are parametrized by a boundary point, either tangent direction and a parameter along the tangent line, a two-dimensional family in the three-dimensional interior phase space. Flow boxes at transverse impacts transfer this null property to the boundary flux.

Outside this set, the next-collision map $\mathcal B$ is invertible and preserves $\nu$, because straight-line propagation between reflections and specular reflection preserve the boundary flux. Removing the countably many forward and backward preimages of the grazing set leaves all its iterates defined.

Let $\tau(\rho)>0$ be the time to the next reflection from $\rho\in\Sigma$. Poincar\'e recurrence on the finite-measure sets $\{\tau\ge1/m\}$ shows that for $\nu$-almost every $\rho$, some such set is revisited infinitely often, both forwards and backwards. For this $m$,
\begin{equation*}
\sum_{j=0}^{N-1}\tau(\mathcal B^j\rho)
\ge\frac1m\#\{j\in\{0,\ldots,N-1\}:\tau(\mathcal B^j\rho)\ge1/m\}
\longrightarrow\infty.
\end{equation*}
Applying the same argument to $\mathcal B^{-1}$ gives divergence backwards in time. Thus impacts do not accumulate in finite time for almost every boundary state. The flow-box measure $\dd\nu\,\dd t$ gives the corresponding full-measure assertion in $S^*\Omega$.

\smallskip\noindent\emph{Measurability and invariance.} We separate trajectories according to their finitely many transverse reflections between time zero and a given time. The straight-line propagation maps between reflections and the successive transverse hitting times are smooth on each such local piece. These pieces, together with the measurable graphs of the impact times, give a countable measurable decomposition. This proves joint measurability. For each fixed time, the initial states having an impact at that time form a null set by the flow-box representation. Invariance follows from volume preservation during straight-line propagation and the invariant specular reflection flux.
\end{proof}

Null exceptional initial states may be assigned arbitrary measurable trajectories; their assignment will not enter the essential condition below. Values at reflection instants likewise do not affect time integrals.

For a Borel set $G\subset(0,T)\times\Omega$, extend $\1_G$ by zero to the boundary and define the observed time by
\begin{equation*}
x_\rho(t):=\pi\Phi_t\rho\qquad\text{and}\qquad
D_G(\rho):=\int_0^T\1_G\bigl(t,x_\rho(t)\bigr)\dd t
=|\{t\in(0,T):\bigl(t,x_\rho(t)\bigr)\in G\}|.
\end{equation*}
Lemma~\ref{lem:meas-flow} and Tonelli's theorem imply that $D_G$ is measurable, with $0\le D_G\le T$.

\subsection{Essential GCC and invariance under null modifications}\label{subsec:meas-def}

Every Lebesgue measurable $G\subset(0,T)\times\Omega$ differs from a Borel set by a spacetime null set. Choose such a Borel set to define $D_G$; this is what is meant by choosing a Borel representative. The proposition below shows that the resulting essential condition does not depend on this choice.

\begin{defi*}[Essential spacetime measurable GCC]
The set $G$ satisfies \eqref{eq:gcc} on $[0,T]$ if there exists $\eta>0$ such that
\begin{equation*}
\essinf_{\rho\in S^*\Omega}D_G(\rho)\ge\eta,
\end{equation*}
where the essential infimum is taken with respect to $\mu_L$. Equivalently, there is a measurable exceptional set $\mathcal N\subset S^*\Omega$ such that
\begin{equation*}
\mu_L(\mathcal N)=0\qquad\text{and}\qquad
D_G(\rho)\ge\eta\quad\text{for every }\rho\in S^*\Omega\setminus\mathcal N.
\end{equation*}
\end{defi*}

The constant $\eta$ may depend on $\Omega,T,G$, but is uniform over initial states outside $\mathcal N$. Changing $G$ on a spacetime null set leaves the wave observation integral unchanged. The following proposition gives the corresponding invariance of \eqref{eq:gcc}.

\begin{prop}[Invariance under null modifications]\label{prop:representative}
For Borel sets $G,\widetilde G\subset(0,T)\times\Omega$,
\begin{equation}\label{eq:rep-bound}
\|D_G-D_{\widetilde G}\|_{L^1(S^*\Omega,\mu_L)} \le 2\pi|G\mathbin\triangle\widetilde G|.
\end{equation}
Consequently, if $|G\mathbin\triangle\widetilde G|=0$, then $D_G=D_{\widetilde G}$ for Liouville-almost every initial state, and their essential infima agree.
\end{prop}

\begin{proof}
The absolute difference of the two time integrals is bounded by the integral of $\1_{G\triangle\widetilde G}$. Integrating in $\rho$ and applying Tonelli's theorem and invariance of Liouville measure, we obtain
\begin{align*}
\int_{S^*\Omega}|D_G-D_{\widetilde G}|\dd\mu_L& \le\int_0^T\int_{S^*\Omega}\1_{G\triangle\widetilde G}(t,\pi\Phi_t\rho)\dd\mu_L\dd t\\
&=\int_0^T\int_\Omega\int_{S^1}\1_{G\triangle\widetilde G}(t,x)\dd\sigma\dd x\dd t
=2\pi|G\triangle\widetilde G|.
\end{align*}
If $|G\triangle\widetilde G|=0$, then $D_G=D_{\widetilde G}$ almost everywhere, proving the assertion.
\end{proof}

A spacetime null set can contain an entire ray graph. Indeed, for a regular initial state $\rho$, its spacetime graph satisfies
\begin{equation*}
\mathcal N_\rho:=\{\bigl(t,x_\rho(t)\bigr):0<t<T,\ x_\rho(t)\in\Omega\},\qquad
|\mathcal N_\rho|=0\qquad\text{and}\qquad D_{\mathcal N_\rho}(\rho)=T.
\end{equation*}
Thus a null modification can change the whole time integral on one ray. This explains the essential infimum in \eqref{eq:gcc}.

\subsection{Measurable observation and the classical argument}\label{subsec:meas-argument}

The classical high-frequency argument combines localization of the limiting energy with propagation along generalized rays; see \cite{BLR1992,LRLTT2017}. We isolate the interior localization step. Set $Q_T:=(0,T)\times\Omega$, and let $v_k=\partial_tu_k$ be a bounded, weakly null sequence in $L^2(Q_T)$. After passing to a subsequence, there is a nonnegative Radon measure $\mu$ on $Q_T$ such that
\begin{equation*}
|v_k|^2\dd x\dd t\rightharpoonup\mu,\qquad
\int_{Q_T}a|v_k|^2\dd x\dd t\longrightarrow\int_{Q_T}a\dd\mu
\quad\text{for every }a\in C_c(Q_T).
\end{equation*}
This measure records the limiting spacetime energy density; the microlocal defect measure also records the propagation directions. If $G\subset Q_T$ is open and the observation tends to zero, then every nonnegative $a\in C_c^\infty(G)$ satisfies
\begin{equation*}
\int a\dd\mu
=\lim_{k\to\infty}\int_{Q_T}a|v_k|^2\dd x\dd t
\le\|a\|_{L^\infty}\lim_{k\to\infty}\int_G|v_k|^2\dd x\dd t=0.
\end{equation*}
Thus $\mu(G)=0$: smooth weights supported inside the open observation region detect all limiting energy there. This supplies the local vanishing used in the propagation argument.

For a merely measurable set, such interior weights need not detect its observation; a positive-measure set may have empty interior. Weak convergence alone does not give
\begin{equation*}
\int_G|v_k|^2\dd x\dd t\longrightarrow\mu(G).
\end{equation*}
To see what smooth approximation would require, choose weights satisfying
\begin{equation*}
b_\varepsilon\in C_c^\infty(Q_T),\qquad
0\le b_\varepsilon\le1,\qquad
\|b_\varepsilon-\1_G\|_{L^1(Q_T)}\longrightarrow0.
\end{equation*}
For each fixed $\varepsilon$, weak convergence applies to $b_\varepsilon$. However, comparison with the actual observation gives
\begin{equation*}
\left|\int_G|v_k|^2\dd x\dd t
-\int_{Q_T}b_\varepsilon|v_k|^2\dd x\dd t\right|
\le\int_{Q_T}|\1_G-b_\varepsilon|\,|v_k|^2\dd x\dd t.
\end{equation*}
Thus this approximation route needs additional control of the weighted error, for example
\begin{equation*}
\lim_{\varepsilon\to0}\limsup_{k\to\infty}
\int_{Q_T}|\1_G-b_\varepsilon|\,|v_k|^2\dd x\dd t=0.
\end{equation*}
Lebesgue approximation and boundedness of $v_k$ in $L^2(Q_T)$ do not imply this estimate. High-frequency energy may concentrate where the approximation error is large, even when that region has small Lebesgue measure.

\begin{remark}\label{rem:meas-argument}
This difficulty already arises for cylindrical observation $G=(0,T)\times\omega$ when $\omega$ is merely measurable. Separately, the Liouville-almost-everywhere ray bound in \eqref{eq:gcc} need not hold almost everywhere for a microlocal defect measure, which may be singular with respect to Liouville measure.
Establishing a sufficient geometric criterion adapted to measurable observation would require further, more refined analysis of the interaction between high-frequency waves and the observation set.
\end{remark}

\subsection{Countable constructions and decreasing sets}\label{subsec:meas-limits}

For decreasing Borel observation regions, bounded convergence on each regular ray gives
\begin{equation*}
G^{(N)}\downarrow G\quad \Longrightarrow \quad
D_{G^{(N)}}(\rho)\downarrow D_G(\rho).
\end{equation*}
If all finite stages have the same lower bound, intersect their countably many full-measure sets. On this intersection,
\begin{equation*}
D_G(\rho)=\lim_{N\to\infty}D_{G^{(N)}}(\rho)\ge\eta.
\end{equation*}
Thus the uniform bound supplied by the geometric proof passes to the fixed observation region.

More generally, the approximation used in Remark~\ref{rem:aos-limit} satisfies
\begin{equation*}
|G^{(N)}\mathbin\triangle G|\to0
\quad\Longrightarrow\quad
\|D_{G^{(N)}}-D_G\|_{L^1(S^*\Omega,\mu_L)}\to0
\end{equation*}
by \eqref{eq:rep-bound}. A common almost-everywhere lower bound then passes to the limit by a subsequence. In both passages, the geometric argument supplies a common lower bound $\eta>0$.

The same countable-intersection convention applies to all null modifications of phases, profile representatives, and observation regions used in the constructions.

\bibliographystyle{abbrv}
\bibliography{references}

@article{Friedlander1976,
  author = {Friedlander, F. G.},
  title = {The wave front set of the solution of a simple initial-boundary value problem with glancing rays},
  journal = {Math. Proc. Cambridge Philos. Soc.}, volume = {79}, number = {1}, pages = {145--159}, year = {1976},
  doi = {10.1017/S0305004100052166}
}

@article{MelroseTaylor1986,
  author = {Melrose, R. B. and Taylor, M. E.},
  title = {The radiation pattern of a diffracted wave near the shadow boundary},
  journal = {Comm. Partial Differential Equations}, volume = {11}, number = {6}, pages = {599--672}, year = {1986},
  doi = {10.1080/03605308608820439}
}

@article{SmithSogge2007,
  author = {Smith, H. F. and Sogge, C. D.},
  title = {On the {$L^p$} norm of spectral clusters for compact manifolds with boundary},
  journal = {Acta Math.}, volume = {198}, pages = {107--153}, year = {2007},
  doi = {10.1007/s11511-007-0014-z}
}

@article{Ivanovici2010,
  author = {Ivanovici, Oana},
  title = {Counterexamples to {Strichartz} estimates for the wave equation in domains},
  journal = {Math. Ann.}, volume = {347}, pages = {627--673}, year = {2010},
  doi = {10.1007/s00208-009-0454-1}
}

@article{ILP2014,
  author = {Ivanovici, Oana and Lebeau, Gilles and Planchon, Fabrice},
  title = {Dispersion for the wave equation inside strictly convex domains {I}: the {Friedlander} model case},
  journal = {Ann. of Math. (2)}, volume = {180}, number = {1}, pages = {323--380}, year = {2014},
  doi = {10.4007/annals.2014.180.1.7}
}

@inproceedings{WangICM2026,
  author = {Wang, Gengsheng},
  title = {Geometric characteristics of observable regions},
  booktitle = {Proceedings of the International Congress of Mathematicians 2026},
  volume = {7}, pages = {249--266},
  publisher = {Society for Industrial and Applied Mathematics}, year = {2026},
  doi = {10.1137/25M1800809}
}

@article{BBZ2013,
  author = {Bourgain, Jean and Burq, Nicolas and Zworski, Maciej},
  title = {Control for {Schr\"odinger} operators on 2-tori: rough potentials},
  journal = {J. Eur. Math. Soc. (JEMS)}, volume = {15}, number = {5}, pages = {1597--1628}, year = {2013},
  doi = {10.4171/JEMS/399}
}

@unpublished{BPP2026,
  author = {Beauchard, Karine and Perrin, Thomas and Pozzoli, Eugenio},
  title = {Approximate controllability of a bilinear wave equation and minimum time},
  note = {arXiv:2601.19544}, year = {2026}
}

@article{BLR1992,
  author = {Bardos, Claude and Lebeau, Gilles and Rauch, Jeffrey},
  title = {Sharp sufficient conditions for the observation, control, and stabilization of waves from the boundary},
  journal = {SIAM J. Control Optim.}, volume = {30}, pages = {1024--1065}, year = {1992}
}

@article{LRLTT2017,
  author = {Le Rousseau, J{\'e}r{\^o}me and Lebeau, Gilles and Terpolilli, Peppino and Tr{\'e}lat, Emmanuel},
  title = {Geometric control condition for the wave equation with a time-dependent observation domain},
  journal = {Anal. PDE}, volume = {10}, pages = {983--1015}, year = {2017}
}

@book{TW2009,
  author = {Tucsnak, Marius and Weiss, George},
  title = {Observation and Control for Operator Semigroups},
  publisher = {Birkh{\"a}user}, address = {Basel}, year = {2009}
}

@book{Lions1988,
  author = {Lions, Jacques-Louis},
  title = {Contr{\^o}labilit{\'e} exacte, perturbations et stabilisation de syst{\`e}mes distribu{\'e}s. {Tome 1}: Contr{\^o}labilit{\'e} exacte},
  publisher = {Masson}, address = {Paris}, year = {1988}
}

@article{BG1997,
  author = {Burq, Nicolas and G{\'e}rard, Patrick},
  title = {Condition n{\'e}cessaire et suffisante pour la contr{\^o}labilit{\'e} exacte des ondes},
  journal = {C. R. Acad. Sci. Paris S{\'e}r. I Math.},
  volume = {325}, number = {7}, pages = {749--752}, year = {1997},
  doi = {10.1016/S0764-4442(97)80053-5}
}

@article{CCM2014,
  author = {Castro, Carlos and C{\^\i}ndea, Nicolae and M{\"u}nch, Arnaud},
  title = {Controllability of the linear one-dimensional wave equation with inner moving forces},
  journal = {SIAM J. Control Optim.},
  volume = {52}, number = {6}, pages = {4027--4056}, year = {2014},
  doi = {10.1137/140956129}
}

@article{Shao2019,
  author = {Shao, Arick},
  title = {On {Carleman} and observability estimates for wave equations on time-dependent domains},
  journal = {Proc. Lond. Math. Soc. (3)},
  volume = {119}, number = {4}, pages = {998--1064}, year = {2019},
  doi = {10.1112/plms.12253}
}

@article{BG2020,
  author = {Burq, Nicolas and G{\'e}rard, Patrick},
  title = {Stabilization of wave equations on the torus with rough dampings},
  journal = {Pure Appl. Anal.},
  volume = {2}, number = {3}, pages = {627--658}, year = {2020},
  doi = {10.2140/paa.2020.2.627}
}

@article{Burq1998,
  author = {Burq, Nicolas},
  title = {D{\'e}croissance de l'{\'e}nergie locale de l'{\'e}quation des ondes pour le probl{\`e}me ext{\'e}rieur et absence de r{\'e}sonance au voisinage du r{\'e}el},
  journal = {Acta Math.}, volume = {180}, number = {1}, pages = {1--29}, year = {1998},
  doi = {10.1007/BF02392877}
}

@article{BZ2004,
  author = {Burq, Nicolas and Zworski, Maciej},
  title = {Geometric control in the presence of a black box},
  journal = {J. Amer. Math. Soc.}, volume = {17}, number = {2}, pages = {443--471}, year = {2004},
  doi = {10.1090/S0894-0347-04-00452-7}
}

@incollection{Lebeau1996,
  author = {Lebeau, Gilles},
  title = {{\'E}quation des ondes amorties},
  booktitle = {Algebraic and Geometric Methods in Mathematical Physics ({Kaciveli}, 1993)},
  series = {Math. Phys. Stud.}, volume = {19}, pages = {73--109},
  publisher = {Kluwer Acad. Publ.}, address = {Dordrecht}, year = {1996}
}

@article{NWXKdV2026,
  author = {Niu, Jingrui and Wang, Ming and Xiang, Shengquan},
  title = {The periodic {KdV} with control on space-time measurable sets},
  journal = {J. Funct. Anal.}, volume = {291}, number = {11}, pages = {111621}, year = {2026},
  doi = {10.1016/j.jfa.2026.111621}
}

@book{Coron2007,
  author = {Coron, Jean-Michel},
  title = {Control and Nonlinearity},
  series = {Mathematical Surveys and Monographs}, volume = {136},
  publisher = {American Mathematical Society}, address = {Providence, RI}, year = {2007}
}

@unpublished{NWX,
  author = {Niu, Jingrui and Wang, Ming and Xiang, Shengquan},
  title = {Symmetry conditions for spacetime observability of wave equations on the torus},
  note = {arXiv:2512.09873v2}, year = {2025}
}

@unpublished{HX,
  author = {Huang, Bin and Xiang, Shengquan},
  title = {Observable symmetry for spacetime observability
           of wave equations on the interval},
  note = {arXiv:2609.27507},
  year = {2026}
}

@book{Watson,
  author = {Watson, G. N.}, title = {A Treatise on the Theory of Bessel Functions},
  edition = {Second}, publisher = {Cambridge University Press}, year = {1944}
}

@book{Olver,
  author = {Olver, F. W. J.}, title = {Asymptotics and Special Functions},
  publisher = {Academic Press}, year = {1974}
}

@article{Russell1971a,
  author = {Russell, David L.},
  title = {Boundary value control of the higher-dimensional wave equation},
  journal = {SIAM J. Control}, volume = {9}, pages = {29--42}, year = {1971}
}

@article{RT1974,
  author = {Rauch, Jeffrey and Taylor, Michael},
  title = {Exponential decay of solutions to hyperbolic equations in bounded domains},
  journal = {Indiana Univ. Math. J.}, volume = {24}, pages = {79--86}, year = {1974},
  doi = {10.1512/iumj.1974.24.24004}
}

@article{Zhang2000,
  author = {Zhang, Xu},
  title = {Explicit observability inequalities for the wave equation with lower order terms by means of {Carleman} inequalities},
  journal = {SIAM J. Control Optim.}, volume = {39}, number = {3}, pages = {812--834}, year = {2000},
  doi = {10.1137/S0363012999350298}
}

@article{DZZ2008,
  author = {Duyckaerts, Thomas and Zhang, Xu and Zuazua, Enrique},
  title = {On the optimality of the observability inequalities for parabolic and hyperbolic systems with potentials},
  journal = {Ann. Inst. H. Poincar\'e C Anal. Non Lin\'eaire},
  volume = {25}, number = {1}, pages = {1--41}, year = {2008},
  doi = {10.1016/j.anihpc.2006.07.005}
}

@article{Tataru1995,
  author = {Tataru, Daniel},
  title = {Unique continuation for solutions to {PDE}'s; between {H\"ormander}'s theorem and {Holmgren}'s theorem},
  journal = {Comm. Partial Differential Equations}, volume = {20}, number = {5-6}, pages = {855--884}, year = {1995},
  doi = {10.1080/03605309508821117}
}

@article{RZ1998,
  author = {Robbiano, Luc and Zuily, Claude},
  title = {Uniqueness in the {Cauchy} problem for operators with partially holomorphic coefficients},
  journal = {Invent. Math.}, volume = {131}, number = {3}, pages = {493--539}, year = {1998},
  doi = {10.1007/s002220050212}
}

@article{LL2019,
  author = {Laurent, Camille and L\'eautaud, Matthieu},
  title = {Quantitative unique continuation for operators with partially analytic coefficients. {Application} to approximate control for waves},
  journal = {J. Eur. Math. Soc. (JEMS)}, volume = {21}, number = {4}, pages = {957--1069}, year = {2019},
  doi = {10.4171/JEMS/854}
}

@article{DE2017,
  author = {Dehman, Belhassen and Ervedoza, Sylvain},
  title = {Observability estimates for the wave equation with rough coefficients},
  journal = {C. R. Math. Acad. Sci. Paris}, volume = {355}, number = {5}, pages = {499--514}, year = {2017},
  doi = {10.1016/j.crma.2017.03.011}
}

@misc{DEZ2025,
  author = {Dehman, Belhassen and Ervedoza, Sylvain and Zuazua, Enrique},
  title = {Regional and partial observability and control of waves},
  year = {2025}, note = {arXiv:2504.18976}
}

@misc{BurqZhu2025,
  author = {Burq, Nicolas and Zhu, Hui},
  title = {Observability of {Schr\"odinger} propagators on tori in rough settings},
  year = {2025}, note = {arXiv:2509.23965}
}

@article{MS1978,
  author = {Melrose, Richard B. and Sj\"ostrand, Johannes},
  title = {Singularities of boundary value problems. {I}},
  journal = {Comm. Pure Appl. Math.}, volume = {31}, number = {5}, pages = {593--617}, year = {1978},
  doi = {10.1002/cpa.3160310504}
}

@misc{BDLR2024,
  author = {Burq, Nicolas and Dehman, Belhassen and Le Rousseau, J{\'e}r{\^o}me},
  title = {Measure and continuous vector field at a boundary {I}: propagation equations and wave observability},
  year = {2024}, note = {arXiv:2407.02255},
  doi = {10.48550/arXiv.2407.02255}
}

@book{Komornik1994,
  author = {Komornik, Vilmos},
  title = {Exact Controllability and Stabilization: The Multiplier Method},
  series = {RAM: Research in Applied Mathematics},
  publisher = {Masson and John Wiley \& Sons}, address = {Paris and Chichester}, year = {1994}
}

@article{Zuazua1990,
  author = {Zuazua, Enrique},
  title = {Exact controllability for the semilinear wave equation},
  journal = {J. Math. Pures Appl. (9)}, volume = {69}, number = {1}, pages = {1--31}, year = {1990}
}

@article{LL2016,
  author = {Laurent, Camille and L\'eautaud, Matthieu},
  title = {Uniform observability estimates for linear waves},
  journal = {ESAIM Control Optim. Calc. Var.}, volume = {22}, number = {4}, pages = {1097--1136}, year = {2016},
  doi = {10.1051/cocv/2016046}
}

@article{JL2013,
  author = {Joly, Romain and Laurent, Camille},
  title = {Stabilization for the semilinear wave equation with geometric control condition},
  journal = {Anal. PDE}, volume = {6}, number = {5}, pages = {1089--1119}, year = {2013},
  doi = {10.2140/apde.2013.6.1089}
}

@article{KX2024,
  author = {Krieger, Joachim and Xiang, Shengquan},
  title = {Semi-global controllability of a geometric wave equation},
  journal = {Pure Appl. Math. Q.}, volume = {20}, number = {4}, pages = {1763--1824}, year = {2024},
  doi = {10.4310/pamq.2024.v20.n4.a9}
}

@article{CKX2025,
  author = {Coron, Jean-Michel and Krieger, Joachim and Xiang, Shengquan},
  title = {Global controllability and stabilization of the wave maps equation from a circle to a sphere},
  journal = {Calc. Var. Partial Differential Equations}, volume = {64}, number = {4}, pages = {Paper No. 108}, year = {2025},
  doi = {10.1007/s00526-025-02961-3}
}

@misc{CKXgeneral2025,
  author = {Coron, Jean-Michel and Krieger, Joachim and Xiang, Shengquan},
  title = {Wave maps from circle to {Riemannian} manifold: global controllability is equivalent to homotopy},
  year = {2025}, note = {arXiv:2509.12779}
}

@article{BHHR2015,
  author = {Beauchard, Karine and Helffer, Bernard and Henry, Rapha\"el and Robbiano, Luc},
  title = {Degenerate parabolic operators of {Kolmogorov} type with a geometric control condition},
  journal = {ESAIM Control Optim. Calc. Var.}, volume = {21}, number = {2}, pages = {487--512}, year = {2015},
  doi = {10.1051/cocv/2014035}
}

@article{WWZ2019,
  author = {Wang, Gengsheng and Wang, Ming and Zhang, Yubiao},
  title = {Observability and unique continuation inequalities for the {Schr\"odinger} equation},
  journal = {J. Eur. Math. Soc. (JEMS)}, volume = {21}, number = {11}, pages = {3513--3572}, year = {2019},
  doi = {10.4171/JEMS/908}
}

@article{BM2023,
  author = {Burq, Nicolas and Moyano, Iv\'an},
  title = {Propagation of smallness and control for heat equations},
  journal = {J. Eur. Math. Soc. (JEMS)}, volume = {25}, number = {4}, pages = {1349--1377}, year = {2023},
  doi = {10.4171/JEMS/1213}
}

@article{EV2018,
  author = {Egidi, Michela and Veseli\'c, Ivan},
  title = {Sharp geometric condition for null-controllability of the heat equation on {$\mathbb{R}^d$} and consistent estimates on the control cost},
  journal = {Arch. Math. (Basel)}, volume = {111}, number = {1}, pages = {85--99}, year = {2018},
  doi = {10.1007/s00013-018-1185-x}
}

@article{WWZZ2019,
  author = {Wang, Gengsheng and Wang, Ming and Zhang, Can and Zhang, Yubiao},
  title = {Observable set, observability, interpolation inequality and spectral inequality for the heat equation in {$\mathbb{R}^n$}},
  journal = {J. Math. Pures Appl. (9)}, volume = {126}, pages = {144--194}, year = {2019},
  doi = {10.1016/j.matpur.2019.04.009}
}

@article{BZ2019,
  author = {Burq, Nicolas and Zworski, Maciej},
  title = {Rough controls for {Schr\"odinger} operators on 2-tori},
  journal = {Ann. H. Lebesgue}, volume = {2}, pages = {331--347}, year = {2019},
  doi = {10.5802/ahl.19}
}

@article{PW2013,
  author = {Phung, Kim Dang and Wang, Gengsheng},
  title = {An observability estimate for parabolic equations from a measurable set in time and its applications},
  journal = {J. Eur. Math. Soc. (JEMS)}, volume = {15}, number = {2}, pages = {681--703}, year = {2013},
  doi = {10.4171/JEMS/371}
}

@article{AEWZ2014,
  author = {Apraiz, Jone and Escauriaza, Luis and Wang, Gengsheng and Zhang, Can},
  title = {Observability inequalities and measurable sets},
  journal = {J. Eur. Math. Soc. (JEMS)}, volume = {16}, number = {11}, pages = {2433--2475}, year = {2014},
  doi = {10.4171/JEMS/490}
}

@unpublished{AOSII,
  author = {Niu, Jingrui and Xiang, Shengquan},
  title = {Asymptotic observable symmetry for spacetime observability of waves {II}: Smooth bounded domains},
  note = {Manuscript},
  year = {2026}
}

@article{Zhang2000Potential,
  author = {Zhang, Xu},
  title = {Explicit observability estimate for the wave equation
           with potential and its application},
  journal = {Proc. R. Soc. Lond. A},
  volume = {456},
  number = {1997},
  pages = {1101--1115},
  year = {2000},
  doi = {10.1098/rspa.2000.0553}
}

@unpublished{LLZ2026,
  author = {Liao, Zhonghua and L{\"u}, Qi and Zhang, Xu},
  title = {Observability of wave equations with time- and
           space-dependent potentials: a counterexample under
           {GCC} and a positive criterion on the disk},
  note = {arXiv:2610.25327},
  year = {2026}
}

@article{ALM2016,
  author = {Anantharaman, Nalini and L{\'e}autaud, Matthieu
            and Maci{\`a}, Fabricio},
  title = {{Wigner} measures and observability for the
           {Schr{\"o}dinger} equation on the disk},
  journal = {Inventiones Mathematicae},
  volume = {206},
  number = {2},
  pages = {485--599},
  year = {2016},
  doi = {10.1007/s00222-016-0658-4}
}

\end{document}